\documentclass[10pt,letterpaper]{article}
\usepackage{amsmath}               
\usepackage{amsfonts}              
\usepackage{amsthm}               
\usepackage{relsize}
\usepackage{enumerate}
\usepackage{mathtools}
\usepackage{commath}
\usepackage{ amssymb }
\usepackage{url}
\usepackage{hyperref}
\usepackage[T1]{fontenc}
\hypersetup{pdftex,colorlinks=true,allcolors=blue}
\usepackage{hypcap}
\usepackage{comment}
\usepackage{graphicx}
\usepackage[byname]{smartref}
\usepackage{datetime}
\usepackage{verbatim}
\usepackage{scalerel}
\usepackage{appendix}
\usepackage{latexsym}
\usepackage{appendix}
\usepackage{authblk}

\newtheorem{theorem}{Theorem}[section]
 \newtheorem{corollary}[theorem]{Corollary}
 \newtheorem{lemma}[theorem]{Lemma}
 \newtheorem{proposition}[theorem]{Proposition}
 \theoremstyle{definition}
 \newtheorem{definition}[theorem]{Definition}
 \theoremstyle{remark}
 \newtheorem{remark}[theorem]{Remark}
 \newtheorem*{acknowledgements}{\bf{Acknowledgements}}
 \numberwithin{equation}{section}

\newcommand{\RR}{\mathbb{R}}      
\newcommand{\ZZ}{\mathbb{Z}}      
\newcommand{\CC}{\mathbb{C}}
\newcommand{\Rpos}{\mathbb{R}^{+}}  
\newcommand{\edge}{\mathcal{E}}
\newcommand{\cone}{\mathcal{X}^{\wedge}}
\newcommand{\MM}{\mathbb{M}}

\newcommand{\smooth}{C^{\infty}}
\newcommand{\local}{(r,\sigma_{k},u_{l})}

\newcommand{\Xcal}{\mathcal{X}}

\newcommand{\scone}{\mathcal{K}^{s,\gamma}(\cone)}

\newcommand{\CA}{\mathcal{A}}
\newcommand{\CB}{\mathcal{B}_k}
\newcommand{\cC}{\mathcal{C}_l}

\newcommand{\coordsimple}{(r,\sigma,u)}
\newcommand{\xy}{\mathbb{R}_{x}^{n}\oplus\mathbb{R}_{y}^{n}}
\newcommand\scalemath[2]{\scalebox{#1}{\mbox{\ensuremath{\displaystyle #2}}}}

\begin{document}

\author{Josue Rosario-Ortega}
\affil{The University of Western Ontario\\
Huron University College\\
London, Ontario, Canada\\
jrosari3@alumni.uwo.ca\\
Spring 2017} 
\date{}


\title
{Moduli space and deformations of special Lagrangian submanifolds with edge singularities}

\maketitle

\begin{abstract}
Special Lagrangian submanifolds are submanifolds of a Calabi -Yau manifold calibrated by the real part of the holomorphic volume form.  In this paper we use elliptic theory for edge-degenerate differential operators on singular manifolds  to study the moduli space of deformations of special Lagrangian submanifolds with edge singularities. We obtain a general theorem describing the local structure of the moduli space. When the obstruction space vanishes the moduli space is a smooth, finite dimensional manifold. 
\end{abstract}


\tableofcontents

\section{Introduction} \label{intro}

The study of moduli spaces of deformations of a special Lagrangian submanifold  in a Calabi-Yau manifold started with the work of McLean \cite{mclean}, where he studied the deformation of compact special Lagrangian submanifolds (without boundary). He proved that the moduli space is a finite dimensional, smooth manifold with dimension equal to the dimension of its space of harmonic 1-forms. The role of special Lagrangian fibrations of Calabi-Yau manifolds in mirror symmetry and especially the presence of singular fibers, motivated the study of special Lagrangian submanifolds with conical singularities/ends \cite{Marshall02,pacini}, \cite{pacini2,joyceII}. Moreover in the simplest example of a Calabi-Yau manifold, $\CC^{n}$, the fact that special Lagrangian submanifolds are minimal implies the nonexistence of compact special Lagrangian submanifolds. Hence the search for special Lagrangian submanifolds in $\CC^{n}$ must be done in the category of non-compact or singular spaces.

Broadly speaking, the study of moduli spaces of special Lagrangian deformations is performed by identifying nearby special Lagrangian submanifolds with elements in the zero set of a non-linear elliptic partial differential operator that governs the deformations. By means of the Implicit Function theorem for Banach spaces, this is reduced to the analysis of the linearised equation. Hence the study of moduli spaces of deformations of a special Lagrangian submanifold requires a good understanding of elliptic equations on the base space. 

The theory of linear elliptic partial differential equations on smooth, compact manifolds without boundary is well-developed: the construction of parametrices of inverse order, a complete calculus of elliptic $\Psi DOs$, elliptic regularity, the equivalence between ellipticity and the existence of an a priori  estimate, the equivalence of ellipticity and Fredholmness and finally the celebrated Atiyah-Singer index formula. All these elements are already classical tools when studying elliptic equations on compact manifolds. 

In contrast, in  non-compact or singular spaces there is no canonical approach or methods  to study elliptic equations. Even the concept of ellipticity on a non-compact or singular manifold is not canonical as it is in the compact case. The basic model of singularity in the theory of PDEs on singular manifolds is the conical singularity. Near a vertex, a manifold with conical singularities looks like $\frac{\overline{\RR}^{+}\times \Xcal}  { \{0\}\times\Xcal}$, where $\Xcal$ is a compact manifold without boundary. The usual approach in this case is to blow-up the vertices to obtain a compact manifold with boundary, the stretched manifold, with collar neighborhood  $\overline{\RR}^{+}\times \Xcal$. 

The most common type of degenerate differential operator studied on the collar neighborhood is the Fuchs type operator
\begin{center}
$\operatorname{P}=r^{-m}\sum\limits_{j\leq m}a_{j}(r)\left(-r\partial_{r}\right)^{j}$
\end{center}
with coefficients $a_{j }\in C^{\infty} (\overline\RR^{+}, \operatorname{Diff}^{m-  j }\left(\mathcal{X}\right) )$, where $\operatorname{Diff}^{m-  j }\left(\mathcal{X}\right)$ is the set of classical differential operators of order $m-j$ on the compact manifold $\Xcal$.

 Many authors have studied this type of equation with different approaches: Lockhart and McOwen \cite{lockhart},\cite{lockhart2}, Melrose \cite{aps}, Schulze \cite{schulze1},\cite{schulze2}, \cite{schulze3}, Kozlov, Mazya and Rossmann \cite{kozlov}, among possibly others. The b-calculus of Melrose \cite{aps} and the cone algebra of Schulze \cite{schulze1} are robust and systematic approaches to Fuchs operators with the goal of constructing a calculus or an algebra of $\Psi DOs$ that contains parametrices of Fusch type operators (see \cite{lauter} for a comparison of both approaches).
 
 The edgification of a manifold with conical singularities produces a manifold with edge singularities, where locally near the singularity it looks like the $\RR^{n}\times\left(\frac{\overline{\RR}^{+}\times \Xcal}  { \{0\}\times\Xcal}\right)$. Here we also have a class of edge-degenerate differential operators. A typical example is the Laplace-Beltrami operator associated to an edge metric $r^{2}g_{_{\Xcal}}+dr^{2}+g_{_{\edge}}$ where $g_{_{\edge}}$ is a Riemannian metric on a smooth manifold $\edge$ (the edge) without boundary (see section~\ref{examples}).

This paper is concerned  with deformations of special Lagrangian submanifolds with edge singularity (see section~\ref{examples} for the precise definition of a manifold with edge singularity). The motivations to study special Lagrangian submanifolds with edge singularities are the following: on one hand,  it is a natural next step in the category of singularities where there is  a well-developed elliptic theory \cite{schulze1,schulze2,schulze3}, hence the analysis of the linearised equation that governs the deformations is accessible. An alternative approach to
study edge-degenerate operators developed by R. Mazzeo and B. Vertman can be found
in \cite{mazzeo90}, \cite{vertman}.

On the other hand we are interested in the deformation of calibrated vector bundles, specially, special Lagrangian submanifolds obtained as a conormal bundle $\mathcal{N}^{\ast}(M)\subset T^{\ast}(\RR^{n})\cong\CC^{n}$ of an austere submanifold $M$ in $\RR^{n}$. In this direction, Karigiannis and Leung \cite{leung} obtained special Lagrangian deformations of  $\mathcal{N}^{\ast}(M)$ by affinely translating the fibers. One of the main examples of austere submanifolds in $\RR^{n}$ is the class of austere cones \cite{bryant}. These cones are of the form $\Rpos\times\Xcal\subset \RR^{n}$ where $\Xcal\subset S^{n-1}$ is an austere submanifold of the sphere.  If we assume that the conormal bundle is a trivial bundle near the vertex of the cone (for example if $\Rpos\times\Xcal$ is an orientable hypersurface in $\RR^{n}$ then the conormal bundle is trivial)  then it is diffeomorphic to $\Rpos\times\Xcal\times\RR^{q}$  where $q$ is the codimension of $\Rpos\times\Xcal$ in $\RR^{n}$. This implies that we can consider  $\mathcal{N}^{\ast}(M)$ as a manifold with an edge singularity. More general, we can consider special Lagrangian submanifolds $M$ with edge singularity $\edge$ in a Calabi-Yau manifold $\mathfrak{X}$. For example if $\mathfrak{X}$ is a Calabi-Yau manifold, $\edge$ a compact special Lagrangian submanifold in $\mathfrak{X}$ and $M$ a special Lagrangian submanifold with conical singularity in $\CC^{n}$ then $M\times\edge$ is a special Lagrangian submanifold in $\CC^{n}\times\mathfrak{X}$ singular along $\edge$.

In \cite{schulze1,schulze2,schulze3}, B.-W. Schulze and coauthors have developed a comprehensive elliptic theory of edge-degenerate differential operators:
  \begin{center}
$\operatorname{P}=r^{-m}\sum\limits_{j+\abs{\alpha}\leq m}a_{j\alpha}(r,y)\left(-r\partial_{r}\right)^{j}\left(rD_{s}\right)^{\alpha}.$
\end{center} 
with coefficients $a_{j \alpha}\in C^{\infty} (\overline\RR^{+}\times \Omega, \mathit{,} \operatorname{Diff}^{m- \left( j+\abs{\alpha}\right)}\left(\mathcal{X}\right) ).$

In this paper we use Schulze's approach to study deformation of special Lagrangian submanifolds with edge singularities. We use Schulze's approach to analyze the Hodge-Laplace $\Delta$ and Hodge-deRham $d+d^{\ast}$ operators acting on sections of differential forms induced by edge-degenerate vector fields on $M$, see section~\ref{linear operator}. 

In previous works of Joyce, McLean, Marshall and  Pacini, see \cite{joyceII}, \cite{mclean}, \cite{Marshall02}, \cite{pacini},\cite{pacini2}, finite dimension of the moduli space follows from the Fredholmness of the Hodge-Laplace operator acting on (weighted) Sobolev spaces. As it was mentioned above, there is no canonical notion of ellipticity in non-compact or singular manifolds, however in most approaches, once suitable Banach spaces have been defined, the ellipticity of an operator is defined in such a way that it implies the Fredholm property of the operator acting between those Banach spaces.  In manifolds with conical singularities with local model $\Rpos\times\Xcal$, the concept of ellipticity  is based on the symbolic structure of the Fuchs operator. This is given by two symbols $(\sigma^{m}_{b}(\operatorname{P}),\sigma^{m}_{M}(\operatorname{P})(z))$. The first symbol $\sigma^{m}_{b}(\operatorname{P})$ is the homogeneous boundary principal symbol and $\sigma^{m}_{M}(\operatorname{P})(z)$ is the Mellin conormal symbol. The symbol $\sigma^{m}_{M}(\operatorname{P})(z)$ is an operator-valued symbol given by a holomorphic family of continuous operators parametrized by $z\in\CC$ and acting on the base of the cone $\Xcal$, $\sigma^{m}_{M}(\operatorname{P})(z):H^{s}(\Xcal)\rightarrow H^{s-m}(\Xcal)$. The ellipticity of $\operatorname{P}$ on $\cone:=\Rpos\times\Xcal$ implies that   $\sigma^{m}_{M}(\operatorname{P})(z)$ is a family of isomorphisms for all $z\in \Gamma_{\frac{n+1}{2}-\gamma}=\lbrace z\in\CC : \operatorname{Re}(z)=\frac{n+1}{2}-\gamma\rbrace$ for some weight $\gamma\in\RR$. In the approaches of Melrose and Lockhart-McOwen similar symbolic structures are used to define ellipticity. See section~\ref{symbolsection} for a complete discussion of the symbolic structure. 

In the analysis of edge-degenerate operators, the symbolic structure for an adequate notion of ellipticity involves the edge symbol $\sigma^{m}_{\wedge}(u,\eta)$. This is an operator-valued symbol  given by a  family of continuous operators acting on cone-Sobolev spaces $\mathcal{K}^{s,\gamma}(\cone)$ (see definition~\ref{conespace}) and parametrized by the cotangent bundle of the edge. For   $(u,\eta)\in T^{\ast}\edge\setminus \{0\}$ we have a continuous operator
\begin{center}
$\sigma^{m}_{\wedge}(u,\eta):\mathcal{K}^{s,\gamma}(\cone)\longrightarrow\mathcal{K}^{s-m,\gamma-m}(\cone).$
\end{center}
Analogous to the conical case, a necessary condition for the ellipticity of $\operatorname{P}$ is that $\sigma^{m}_{\wedge}(u,\eta)$ is an isomorphism for every $ (u,\eta)\in T^{\ast}\edge\setminus \lbrace 0 \rbrace$. However, this is rarely the case (for example, in general, the Laplace-Beltrami operator induced by an edge metric does not satisfy this condition). It is more natural to expect the family $\sigma^{m}_{\wedge}(u,\eta)$ to be only Fredholm for every $ (u,\eta)\in T^{\ast}\edge\setminus \{ 0\}$.

 In this case, in order to have a family of isomorphisms, we need to complete the edge symbol with boundary and coboundary conditions.The need for completing the symbols means that $\operatorname{P}$ is not Fredholm unless we impose complementary edge boundary conditions. Moreover boundary and coboundary conditions are an essential part of the regularity of solutions of elliptic edge-degenerate equations (see section~\ref{ellipticsection}). Therefore, if we are interested in studying moduli spaces of deformations of a special Lagrangian submanifold with edge singularities, we need to consider deformations with boundary conditions in the edge in order to obtain regular enough deformations that allow the existence of a smooth, finite dimensional moduli space of deformations. These boundary conditions are given by the trace pseudo-differential operator  that appears in the completion of the symbol. Moreover, solutions of elliptic equations near singularities have a well-known conormal asymptotic expansion. Our case is not an exception and our deformations have conormal asymptotic expansions near the edge, see section~\ref{regularity}.

Once the symbol is completed (this is possible because, in our case, the topological obstruction vanishes, see \ref{symbol-linearop}) we obtain a Fredholm operator in the edge algebra with a parametrix (with asymptotics) of inverse order. At this point we want to use the Implicit Function theorem for Banach spaces to obtain finite dimensionality and smoothness of the moduli space of deformations. However, the possible non-surjectivity of the linearised deformation map produces an obstruction space. The presence of an obstruction space is not unexpected because even in the case of a compact manifold with isolated conical singularities each of the singular cones contributes to the obstruction space. This was studied in detail by Joyce \cite{joyceII}.

Given a special Lagrangian submanifold  in $\CC^{n}$ with edge singularity, $\Phi:M\longrightarrow\CC^{n}$, our moduli space has as parameters an admissible weight $\gamma>\frac{\operatorname{dim}{\Xcal+3}}{2}$ and a trace pseudo-differential operator, $\mathcal{T}$, such that it belongs to a set of boundary condition for an elliptic edge boundary value problem for the Hodge-deRham operator on $M$.

Our main result, theorem~\ref{main} in section~\ref{sectionTheModuliSpace}, is a theorem describing the local structure of the moduli space  $\mathfrak{M}(M, \Phi, \mathcal{T},\gamma)$ considering the possible obstructions (see section~\ref{sectionTheModuliSpace} for the precise definition of the moduli space and further details).
\begin{theorem}
Locally near $M$ the moduli space $\mathfrak{M}(M, \Phi, \mathcal{T},\gamma)$ is homeomorphic to the zero set of a smooth map $\mathfrak{G}$ between smooth manifolds $\mathcal{M}_1 $, $\mathcal{M}_2 $  given as  neighborhoods of zero in finite dimensional Banach spaces. The map $\mathfrak{G}: \mathcal{M}_1 \longrightarrow \mathcal{M}_2$ satisfies $\mathfrak{G}(0)=0$ and $\mathfrak{M}(M, \Phi, \mathcal{T},\gamma)$ near $M$ is a smooth manifold of finite dimension when $\mathfrak{G}$ is the zero map.
\end{theorem}

\section{Special Lagrangian submanifolds and deformations} \label{sec:1}
In this section we introduce the basics of deformation of special Lagrangian submanifolds in $\CC^{n}$. We will use the definitions and notation of this section throughout this paper. For further details on special Lagrangian geometry the reader is referred to   \cite{joyce}  and \cite{harveylawson}.

\subsection{Special Lagrangian submanifolds in $\CC^n$}\label{sec:1.1}
Let $\CC^{n}:=\lbrace(z_1,\cdots ,z_n):z_k\in\CC \text{ for all } 1\leq k\leq n \rbrace$ be the complex $n$-dimensional space. We identify $\CC^{n}$ with $\RR^{2n}=\RR^{n}_x\oplus\RR^{n}_y$ in the following, specific way

 \begin{equation}\label{identification}
(x_1,\cdots,x_n,y_1,\cdots y_n)\longrightarrow (x_1+\sqrt{-1}y_1,\cdots ,x_n+\sqrt{-1}y_n),
\end{equation} 
hence with this identification $z_k=x_k+\sqrt{-1}y_k$. 
Now let's consider the automorphism $J:\CC^{n}\longrightarrow\CC^{n}$ given by $J(z)=\sqrt{-1}z$. Then, under the identification $\CC^{n}\cong\xy$ we have

 \begin{equation}\label{complexstructure}
J=\begin{bmatrix}
0 & -\operatorname{Id}_{\RR^{n}} \\ 
\operatorname{Id}_{\RR^{n}} & 0
\end{bmatrix} :\xy\longrightarrow\xy.
\end{equation}

\begin{definition}
Let $\xi=\xi_1\wedge\cdots\wedge\xi_n$ be an oriented, real $n$-plane in $\CC^{n}$, where $\xi_1,\cdots,\xi_n$ is an oriented, orthonormal basis of $\xi$. We say that $\xi$ is a Lagrangian $n$-plane if $J(\xi)=\xi^{\bot}$ where \begin{equation}\label{lagrangiandef}
\xi^{\bot}=\lbrace \eta\in\CC^{n}: \langle \eta,v \rangle_{g_{\RR^{2n}}}=0 \text{ for all } v\in\xi \rbrace.
\end{equation}
\end{definition}

\begin{definition}
An oriented $n$-submanifold $\psi: M\longrightarrow\CC^{n}$ is a Lagrangian submanifold of $\CC^{n}$ if each tangent plane $\psi_{\ast}(T_{p}M)\subset T\CC^{n}\cong\CC^{n}$ is a Lagrangian $n$-plane in $\CC^{n}$ for every $p\in M$, where $\psi_{\ast}$ denotes the push-forward.
\end{definition}

\begin{definition}\label{hol.volume}
The complex $n$-form $\Omega=dz_1\wedge\cdots\wedge dz_n$\label{holomorphicvolume} is called the holomorphic volume form of $\CC^{n}$.
\end{definition}

\begin{definition}
An oriented Langrangian submanifold $M$ in $\CC^{n}$ is called a special Lagrangian submanifold with phase $\theta_0$ if the following equations are satisfied
 \begin{equation}\label{mastereq1}
 \begin{dcases*}
\operatorname{Re}(e^{-\sqrt{-1}\theta_{0}}\Omega)\big|_{M}=dV_{M} \\
 \operatorname{Im}(e^{-\sqrt{-1}\theta_{0}}\Omega)\big|_{M}=0
\end{dcases*}.
\end{equation}
\end{definition}

Observe that given a special Lagrangian submanifold $\Phi: M\longrightarrow \CC^{n}$ with phase $\theta$ i.e. a Lagrangian submanifold calibrated by  $\operatorname{Re}(e^{-\sqrt{-1}\theta_{}}\Omega)$, the submanifold given by \[e^{-\sqrt{-1}\frac{\theta}{n}}\Phi:M\longrightarrow\CC^{n} \]
is a special Lagrangian submanifold with phase $\theta=0$. 

Therefore by rotating a special Lagrangian submanifold with phase $\theta$ we transform it into a special Lagrangian submanifold with phase zero. Henceforth, when we consider special Lagrangian submanifolds in $\CC^n$, we shall focus and discuss only the case with phase zero.


Let $(\mathfrak{X}, \omega, J, g_{_{\mathfrak{X}}})$ be a K{\"a}hler manifold of complex dimension $n$ with K{\"a}hler form $\omega$, complex structure $J$ and K{\"a}hler metric $g_{_{\mathfrak{X}}}$. Recall that $\mathfrak{X}$ is called a Calabi-Yau manifold if the holonomy group of $g_{_{\mathfrak{X}}}$ is a subgroup of $\operatorname{SU}(n)$, i.e. \[\operatorname{Hol}(g_{_{\mathfrak{X}}})\subseteq \operatorname{SU}(n).\]
Note that $\CC^{n}$ is a Calabi-Yau manifold with the structure  \[(\CC^{n},g_{_{\CC^{n}}},\omega_{_{\CC^{n}}}, \operatorname{}\Omega)\]
 where $g_{_{\mathbb{C}^n}}= |{dz_1}|^2+\cdots+|{dz_N}|^2$, $\omega_{\mathbb{C}^n}=\frac{\sqrt{-1}}{2}\sum\limits_{i=1}^{n}dz_i \wedge d\bar z _{i}$ and $\Omega=dz_{1}\wedge\cdots \wedge dz_{N}$. 
 
Harvey and Lawson in \cite{harveylawson} characterized special Lagrangian submanifolds in a way that has been extremely useful to study the deformation theory.
\begin{proposition}
Let $(\mathfrak{X}, \omega, J, g_{_{\mathfrak{X}}},\Omega_{\mathfrak{X}})$ be a Calabi-Yau manifold and $M$ a $n$-dimensional real submanifold. Then $M$ admits an orientation making it into a special Lagrangian submanifold if and only if \begin{equation}\label{master}
 \begin{dcases*}
\omega \big |_{M}   \equiv 0   \\
 \operatorname{Im}\Omega_{\mathfrak{X}} \big |_{M} \equiv 0 
\end{dcases*}.
\end{equation} 
\end{proposition}

\section{Manifolds with singularities}\label{sec:2}

We are interested in the deformation of special Lagrangian submanifolds with singularities. In this section we provide the definitions and concepts related to singular manifolds that we use throughout this paper. We refer the reader to chapter 1 and 2 of \cite{naza} for further details.   

We denote by \[\operatorname{Diff}_{}(M)\label{algebraclassicaldiff}:=\bigcup\limits_{l\geq 0}\operatorname{Diff}^{l}(M)\] the algebra of classical differential operators on $M$. 

\begin{definition}\label{def1}
A singular manifold is a pair $(M, \mathfrak{D})$ where $M$ is a  smooth manifold possibly non-compact and $\mathfrak{D}\subset \operatorname{Diff}_{}(M) $  is a subalgebra of differential operators such that  its restriction $\mathfrak{D}\big|_{U}$  at every open subset $U$ with compact closure  $\overline{U}\subset M$ is equal to the restriction of the algebra of all differential operators $\operatorname{Diff}_{}(M)\big|_{U}.$
\end{definition}

The algebra $\mathfrak{D}$ is generated by a function space $\mathcal{F}$ such that  $\smooth_{0}(M)\subset\mathcal{F}\subset \smooth(M)$ and a space of vector fields $\mathbf{V}$ on $M$ such that $\smooth_{0}(M,TM)\subset\mathbf{V}\subset \smooth(M, TM).$ We obtain the function space $\mathcal{F}$ by embedding $M$ into a compact manifold with boundary $\MM$ and defining $\mathcal{F}$ as the restriction of the space of smooth functions on $\MM$ i.e. $\mathcal{F}:=\smooth(\MM)\big|_{M}.$

\subsection{Examples}\label{examples}

\begin{enumerate}[i)]
\item Let's consider a smooth manifold $\MM$ with smooth boundary $\partial \MM=\Xcal$, $\operatorname{dim}\Xcal=m$ and define $M:=\MM\setminus\partial \MM$. Let $g_{_{M}}$ be a conical metric on $M$ i.e. $g_{_{M}}$ is a Riemannian metric on $M$ such that on a collar neighborhood of $\partial\MM$ given by $(0,1)\times\Xcal\subset M$ we have \[g_{_{M}}=r^{2}g_{_{\Xcal}}+dr^{2}\]
where $g_{_{\Xcal}}$ is a Riemannian metric on $\Xcal$. Then $g_{_{M}}$ extends to a smooth, symmetric 2-tensor on $\MM$ that degenerates in each tangent direction to $\Xcal$. .

Now, let $\mathcal{V}\in \smooth(\MM,T\MM)$ be a vector field with length of the order of unity with respect to $g_{_{M}}$, i.e. \[\abs{\mathcal{V}(p)}_{g_{_{M}}}\leq C\]
for any $p\in \MM$ and $C>0$ independent of $p$.

On a neighborhood $[0,1)\times\mathcal{U}\subset[0,1)\times\Xcal$ it is easy to see that \begin{equation}\label{conevectorfield}
\mathcal{V}=\mathcal{A}\partial_{r}+\sum\limits^{m}_{k=1}\mathcal{B}_{k}\frac{1}{r}\partial_{k}
\end{equation} 
where $\partial_k$ are the local coordinate vector fields on $\mathcal{U}\subset\Xcal$ and $\mathcal{A},\mathcal{B}_k\in \smooth([0,1)\times\mathcal{U})$.

The algebra of degenerate operators $\mathfrak{D}$ is generated by functions on $M$ smooth up to $r=0$ i.e. $\smooth(\MM)\big|_{M}$ and vector fields $\mathcal{V}$ such that on the collar neighborhood $[0,1)\times\Xcal$ are given by \[\mathcal{V}=\mathcal{A}\partial_{r}+\Theta\] where $\mathcal{A}\in\smooth([0,1)\times\Xcal)$ and $r\Theta\in\smooth([0,1), T\Xcal).$

This algebra is called the algebra of cone-degenerate operators $\operatorname{Diff}_{\operatorname{cone}}(M).$ From the local expressions above we have that every cone-degenerate operator $\operatorname{P}$ of order $l$ can be written in the collar neighborhood as \begin{equation}\label{coneop}\operatorname{P}=r^{-l}\sum\limits_{i\leq l}a_{i}(r)(-r\partial_{r})^{i} 
\end{equation}
where $a_{i}\in\smooth\left([0,1),\operatorname{Diff}^{l-i}(\Xcal)\right),$ with $\operatorname{Diff}^{l-i}(\Xcal)$ denoting the space of classical differential operators of order $l-i$ on $\Xcal$. Cone-degenerate operators are also called Fuchs-type operators.

We say that $(M,\operatorname{Diff}_{\operatorname{cone}}(M))$ is a manifold with conical singularity.

\item Let $\MM$ be a smooth compact manifold with boundary $\partial\MM$. Let $\Xcal$ and $\edge$ be smooth, compact manifolds without boundary such that $\partial\MM$ is the total space of a smooth $\Xcal$-fibration over $\edge$ \[\pi:\partial\MM\longrightarrow\edge.\]\label{boundaryfibration}
Observe that any collar neighborhood of the boundary $[0,1)\times\partial\MM$ has the structure of a $\Xcal$-fibration over $\edge\times [0,1)$. By fixing a collar neighborhood, we use the bundle coordinates on $[0,1)\times\partial\MM$ as admissible coordinates i.e. coordinates of the form $(r,\sigma_k,u_l)$ where $(u_l,r)$ are coordinates on $\edge\times [0,1)$ and $(\sigma_k)$ local coordinates  on the fiber $\Xcal.$

Now, we apply the scheme to construct singular manifolds. Let $M=\MM\setminus \partial\MM$ and equip $M$ with an edge metric \[g_{_{M}}=r^{2}g_{_{\Xcal}}+dr^2+g_{_{\edge}}\] where 
$g_{_{\edge}}$ is a smooth Riemannian metric on $\edge$. 

Observe that the edge metric $g_{_{M}}$ extends to a smooth symmetric 2-tensor on $\MM$ that degenerates on each $\Xcal$-fiber over $\partial\MM.$ In order to define the algebra of degenerate differential operators on $M$ we set $\mathcal{F}=\smooth(\MM)\big|_{M}.$  In admissible coordinates on the collar neighborhood  $[0,1)\times\mathcal{U}\times\Omega\subset [0,1)\times\partial\MM$, the set of vector fields with length of the order of unity with respect to the edge metric $g_{_{M}}$ is given by \begin{equation}\label{edgefield}
{V}=\mathcal{A}\partial_{r}+\sum\limits^{m}_{k=1}\mathcal{B}_k\frac{1}{r}\partial_{k}+\sum\limits^{q}_{l=1}\mathcal{C}_l\partial_{u_l}
\end{equation} where $\mathcal{A},\mathcal{B}_k,\mathcal{C}_l\in\smooth([0,1)\times\mathcal{U}\times\Omega)$, $\partial_k$ are local coordinate vector fields on $\mathcal{U}\subset\Xcal$ and $\partial_{u_{l}}$ are local coordinate vector fields on $\Omega\subset\edge$.

The algebra generated by $\mathcal{F}$ and $\mathbf{V}$ is called the algebra of edge-degenerate operators $\operatorname{Diff}_{\operatorname{edge}}(M).$\label{edgediffoperators} In admissible coordinates, every edge-degenerate operator of order $l$ is given by \begin{equation}\label{edgeop} \operatorname{P}=r^{-l}\sum\limits_{i\leq l}a_{i,\alpha}(r,u)(-r\partial_{r})^{i}(rD_u)^{\alpha}
\end{equation}  where $a_{i,\alpha}\in\smooth\left([0,1)\times\Omega,\operatorname{Diff}^{l-i-\abs{\alpha}}(\Xcal)\right)$ and $D_{u_l}=-\sqrt{-1}\partial_{u_l}$. 

We say that $(M,\operatorname{Diff}_{\operatorname{edge}}(M))$ is a manifold with edge singularity.

\end{enumerate}

\subsection{Analysis on Manifolds with edges.}\label{sectionanalysis}
In this section we describe the necessary elements to study partial differential equations on manifolds with conical or edge singularities. In particular we introduce the relevant concepts and definitions needed for the analysis of deformations of singular special Lagrangian submanifolds carried out in section~\ref{chapterdeformationtheory}. We refer the reader interested in full details and explanations to \cite{schulze3}, chapters 2 and 3. 

\subsection{Sobolev spaces on singular manifolds}\label{subsectionSobolev}
We start by introducing suitable Banach spaces on which cone-degenerate operators act. Recall that the Mellin transformation $\mathcal{M}$ is a continuous operator 
$\mathcal{M}:\smooth _{0}(\Rpos)\longrightarrow\mathcal{A}(\CC)$ given by the integral formula
\[(\mathcal{M}f)(z)=\displaystyle\int_{0}^{\infty}r^{z-1}f(r)dr,\]\label{Mellintransf}
where $\mathcal{A}(\CC)$ is the space of holomorphic functions on $\CC$. 

Very often we need the restriction of the holomorphic function $\mathcal{M}f$ to subsets isomorphic to $\RR$ given by \begin{equation}\label{Mellindomain}
\Gamma_{\beta}=\lbrace z\in\CC : Re(z)=\beta\rbrace.
\end{equation} 

The role of the Mellin transformation in cone-degenerate operators is given by the following basic fact: $(-r\partial_{r})f=\mathcal{M}^{-1}z\mathcal{M}f$ for any $f(r,\sigma)\in\smooth_{0}(\Rpos_{r}\times\RR^{m}_{\sigma})$. Therefore, any cone-degenerate differential operator $\operatorname{P}=r^{-l}\sum\limits_{i\leq l}a_{i}(r)(-r\partial_{r})^{i}$ is given in terms of the Mellin transformation as follows \[\operatorname{P}=r^{-l}\mathcal{M}^{-1}h(r,z)\mathcal{M},\]
where $h(r,z)=\sum\limits_{i\leq l}a_{i}(r)z^{i}.$

\begin{definition}\label{localcone}
The local cone-Sobolev space of order $s$ and weight $\gamma$, denoted by $\mathcal{H}^{s,\gamma}(\Rpos\times\RR^{m})$,  with $s,\gamma\in\RR$, is defined as the closure of $\smooth_{0}(\Rpos\times\RR^{m})$ with respect to the norm 
\begin{align*}
&\norm{f}_{\mathcal{H}^{s,\gamma}(\Rpos\times\RR^{m})}:=
\\&\left(\frac{1}{2\pi i}\displaystyle\int\limits_{\Gamma_{\frac{m+1}{2}-\gamma}}\displaystyle\int\limits_{\RR^m}(1+\abs{z}^{2}+\abs{\xi}^{2})^{s}\abs{\mathcal{M}_{\gamma-\frac{m}{2},r\rightarrow z}\mathcal{F}_{\sigma\rightarrow\xi}f }^{2}dzd\xi \right)^{\frac{1}{2}},
\end{align*} where $\mathcal{F}_{\sigma\rightarrow 
\xi}$ denotes the Fourier transformation in the variable $\sigma\in\RR^{m}$ and the symbol  $\mathcal{M}_{\gamma-\frac{m}{2},r\rightarrow z}$ denotes the restricted Mellin transformation acting on the variable $r\in\Rpos$.
\end{definition}
The local cone-Sobolev spaces are Hilbert spaces with inner product given by  

\begin{align*}
&\frac{1}{2\pi i}\left\langle  (1+\abs{z}^{2}+\abs{\xi}^{2})^{\frac{s}{2}}\mathcal{M}_{\gamma-\frac{m}{2},r\rightarrow z}\mathcal{F}_{\sigma\rightarrow\xi}f,\right. 
\\&  (1+\abs{z}^{2}+\abs{\xi}^{2})^{\frac{s}{2}}\left.\mathcal{M}_{\gamma-\frac{m}{2},r\rightarrow z}\mathcal{F}_{\sigma\rightarrow\xi}g \right\rangle_{L^{2}(\RR\times\RR^{m})}.
\end{align*}

The relation between the local-cone spaces  $\mathcal{H}^{s,\gamma}(\Rpos\times\RR^{m})$ and the standard Sobolev spaces $H^{s}(\RR^{m+1})$ is given by the following transformation. First consider the transformation \begin{equation}\label{S_transf}
S_{\gamma-\frac{m}{2}}:\smooth_{0}(\Rpos\times\RR^{m})\longrightarrow \smooth_{0}(\RR_{t}\times\RR^{m}_x)
\end{equation}  such that $S_{\gamma-\frac{m}{2}}(f)(t,x):=e^{-(\frac{1}{2}-(\gamma-\frac{m}{2}))t}f(e^{-t},x).$ This transformation extends to a Banach space isomorphism between $\mathcal{H}^{s,\gamma}(\Rpos\times\RR^{m})$ and the standard Sobolev spaces $H^{s}(\RR^{m+1}).$ Therefore the norm $\norm{f}_{\mathcal{H}^{s,\gamma}(\Rpos\times\RR^{m})}$ and the norm $\norm{S_{\gamma-\frac{m}{2}}(f)}_{H^{s}(\RR^{m+1})}$ are equivalent. 

In order to define the global cone-Sobolev space on a manifold with conical singularities $M$, we choose a finite open covering $\lbrace \mathcal{U}_{_{\lambda}},\chi_{_{\lambda}}\rbrace$ of $\Xcal$ given by coordinate neighborhoods such that $\chi_{_{\lambda}}:\mathcal{U}_{_{\lambda}}\longrightarrow \RR^{m}$ and $I\times\chi_{_{\lambda}}:\Rpos\times\mathcal{U}_{_{\lambda}}\longrightarrow\Rpos\times\RR^{m}$ with $(I\times\chi_{_{\lambda}})(r,p)=(r,\chi_{_{\lambda}}(p))$ are diffeomorphisms for every $\lambda$. Let $\lbrace \varphi_{_{\lambda}} \rbrace$ be a partition of unity subordinate to  $\lbrace \mathcal{U}_{_{\lambda}}\rbrace$. The global cone-Sobolev space near the conical singularities is modelled on the space  $\mathcal{H}^{s,\gamma}(\cone)$ defined on the open cone $\cone:=\Rpos\times\Xcal$ as  the closure of $\smooth_{0}(\cone)$ with respect to the norm \begin{equation}\label{nearconespace}
\norm{f}_{\mathcal{H}^{s,\gamma}(\cone)}:=\left( \sum_{\lambda}\norm{(I\times\chi_{_{\lambda}}^{\ast})^{-1}\varphi_{_{\lambda}}f}
^{2}_{\mathcal{H}^{s,\gamma}(\Rpos\times\RR^{m})}  \right)^{\frac{1}{2}}
\end{equation} where $m=\operatorname{dim}\Xcal.$

In order to glue together the space $\mathcal{H}^{s,\gamma}(\cone)$ with the classical Sobolev space away from the edge we use a cut-off function $\omega(r)\in\smooth_{0}(\overline{\RR}^{+})$ such that $\omega(r)=1$ for $0\leq r <\varepsilon_1$ and $\omega(r)=0$ for $r\geq\varepsilon_2$ for some $0<\varepsilon_1<\varepsilon_2$.  

\begin{definition}\label{conespace}
Given a compact manifold $M$ with conical singularity, the cone-Sobolev space of order $s$ and weight $\gamma$ is defined as follows \[\mathcal{H}^{s,\gamma}(M):=[\omega]\mathcal{H}^{s,\gamma}(\cone)+[1-\omega]H^{s}(2\MM).\]
The cone-Sobolev space on the open cone $\cone$ is defined in an analogous manner \[\mathcal{K}^{s,\gamma}(\cone):=[\omega]\mathcal{H}^{s,\gamma}(\cone)+[1-\omega]H^{s}_{\operatorname{cone}}(\cone).\]
Both of these spaces are endowed with the topology of the non-direct sum.
\end{definition}

\begin{proposition}\label{cont.cone}
Let $\operatorname{P}\in\operatorname{Diff}^{l}_{\operatorname{cone}}(M)$, then \[  \operatorname{P}:\mathcal{H}^{s,\gamma}(M)\longrightarrow \mathcal{H}^{s-l,\gamma-l}(M) \] is a continuous operator for any $s,\gamma$.
\end{proposition}
See \cite{schulze2} page 153.

The $\Rpos$-action on the cone-Sobolev space is given by $(\kappa_{\lambda}f)(r,\sigma):=\lambda^{\frac{m+1}{2}}f(\lambda r,\sigma)$. This defines a continuous one-parameter group of invertible operators with the strong operator topology.

\begin{definition}Edge-Sobolev spaces.
\begin{enumerate}[i)]
\item We define the edge-Sobolev space on the open edge $\cone\times\RR^{q}$ as the the completion of the Schwartz space $\mathcal{S}(\RR^{q},\mathcal{K}^{s,\gamma}(\cone))$ with respect to the norm \begin{equation}\label{localedgenorm}
\norm{f}_{\mathcal{W}^{s,\gamma}(\cone\times\RR^{q})}=
\left( \displaystyle\int [\eta]^{2s}\norm{ \kappa^{-1}_{[\eta]}(\mathcal{F}_{u\rightarrow \eta} f(\eta))}^{2}_{\scone} d \eta \right)_{.}^{\frac{1}{2}}
\end{equation}

\item Given a compact manifold $M$ with edge singularity $\edge$ the edge-Sobolev space $\mathcal{W}^{s,\gamma}(M)$ is defined as the closure of $\smooth_{0}(M)$ with respect to the norm \begin{equation}\label{globaledgenorm}
\norm{f}_{\mathcal{W}^{s,\gamma}(M)}=\left(\sum_{j} \norm{\omega\phi_{j}f}^{2}_{\mathcal{W}^{s,\gamma}(\cone\times\RR^{q})}+\norm{(1-\omega)f }^{2}_{H^{s}(2\MM)}\right)_{.}^{\frac{1}{2}}
\end{equation}
\end{enumerate}
where $\phi_{j}$ is a partition of unity associated to a finite open cover $\lbrace \Omega_{j}\rbrace$ of $\edge$ and $\omega$ is the cut-off function supported near the edge.
\end{definition}
Similarly we can define $\mathcal{W}^{s,\gamma}(M, E)$ with $E$ an admissible vector bundle over $M$ (see definition~\ref{admissibleE}).
\begin{proposition}\label{cont.edge}
Let $\operatorname{P}\in\operatorname{Diff}^{l}_{\operatorname{edge}}(M)$, then \[  \operatorname{P}:\mathcal{W}^{s,\gamma}(M)\longrightarrow \mathcal{W}^{s-l,\gamma-l}(M) \] is a continuous operator for any $s,\gamma$.
\end{proposition}
See \cite{schulze1} section 3.1, proposition 5.

 \section{Deformation of special Lagrangian submanfolds with edges }\label{chapterdeformationtheory}
In this section we consider the problem of deforming special Lagrangian submanfolds with edges. We study and analyze the non-linear differential operator that governs the special Lagrangian deformations and its linearisation.

\subsection{Preliminaries}\label{subsectionPrelimDeformation}
Given a Calabi-Yau manifold $(\mathfrak{X}, \omega, J, g_{_{\mathfrak{X}}},\Omega_{\mathfrak{X}})$ and a special Lagrangian submanifold \[\Phi:M\longrightarrow \mathfrak{X},\] we are interested in deformations of $M$, as a submanifold of $\mathfrak{X}$, such that the deformed submanifold is special Lagrangian. More precisely we are looking for submanifolds $\Psi:M\longrightarrow \mathfrak{X}$ such that $\Phi$ is isotopic to $\Psi$ and $\Psi(M):=M_{\Psi}$ is special Lagrangian. If we are able to find special Lagrangian deformations of $M$ then we want to investigate the structure of the space containing those special Lagrangian deformations i.e. the moduli space of special Lagrangian deformations $\mathfrak{M}(M,\Phi)$. In general the moduli space will be the space of special Lagrangian embeddings $\Psi:M\longrightarrow\mathfrak{X}$ (equivalent up to diffeomorphism)  isotopic to our original $\Phi$. If we require the isotopy through special Lagrangian submanifolds then we are considering only the connected component in the moduli space containing $M$. If we do not require the intermediate submanifolds to be special Lagrangian then we are considering all the connected components of the moduli space. 

If we consider nearby enough submanifolds then it is possible to obtained deformations of $M$ by moving it in a normal direction $\mathcal{V}$ given by a section of the normal bundle $\mathcal{V}\in\smooth(M,\mathcal{N}(M))$. This is possible thanks to the tubular neighborhood theorem (see \cite{lang} Ch 4, theorem 5.1). Observe that the submanifold $M$ is not required to be closed, see \cite{Lee} theorem 10.19. 

\begin{theorem}\label{tubular}
Let $(\mathfrak{X},g_{_{\mathfrak{X}}})$ be a Riemannian manifold and $M$ an embedded submanifold. Then there exists an open neighborhood $\mathfrak{A}$  of the zero section in $\mathcal{N}(M)$ and an open neighborhood $U$ of $M$ in $\mathfrak{X}$ such that the exponential map $\operatorname{exp}_{g_{_{\mathfrak{X}}}}:\mathfrak{A}\subset\mathcal{N}(M)\longrightarrow U\subset\mathfrak{X}$ is a diffeomorphism.
\end{theorem} 

Therefore any normal section $\mathcal{V}$ lying in $\mathfrak{A}$ will produce an embedded submanifold given by \[ \operatorname{exp}_{g_{_{\mathfrak{X}}}}(\mathcal{V})\circ\Phi :M\longrightarrow \mathfrak{X}\] such that $(\operatorname{exp}_{g_{_{\mathfrak{X}}}}(\mathcal{V})\circ\Phi)(M):=M_{\mathcal{V}}\subset U$. Once we have the deformed submanifold $M_{\mathcal{V}}$ we want to investigate if it is special Lagrangian. Equations \eqref{master} imply that $M_{\mathcal{V}}$ is special Lagrangian if and only if 

 \begin{equation}
 \begin{dcases*}
 (\operatorname{exp}_{g_{_{\mathfrak{X}}}}(\mathcal{V})\circ\Phi)^{\ast}(\omega)  \equiv 0   \\
 (\operatorname{exp}_{g_{_{\mathfrak{X}}}}(\mathcal{V})\circ\Phi)^{\ast}( \operatorname{Im}\Omega_{\mathfrak{X}}) \equiv 0 
\end{dcases*}.
\end{equation} 
This provides us with two explicit equations that $\mathcal{V}$ must satisfy in order to produce a special Lagrangian deformation of $M$. Taking advantage of the fact that $M$ is a Lagrangian submanifold we can use the bundle isomorphisms \[T^{\ast}M \overset{g_{_{\mathfrak{X}}}^{-1}}{\cong} TM \overset{J}{\cong} \mathcal{N}(M)\] so that we obtain that each differential form $\Xi\in \mathcal{C}^{\infty}(M,T^{\ast}M)$ defines a unique section of the normal bundle  
\[\mathcal{V}_{\Xi}=J(g_{_{\mathfrak{X}}}^{-1}(\Xi))\in \mathcal{C}^{\infty}(M,\mathcal{N}(M)).\]

Therefore, we can express the deformation problem as a non-linear operator $\operatorname{P}$ acting on differential forms on $M$: 

\begin{equation}
 \operatorname{P}: \mathcal{C}^{\infty}_{}(M,T^{\ast}M)\big|_{\mathfrak{A}}\longrightarrow \mathcal{C}^{\infty}(M, \bigwedge\nolimits^{2}T^{\ast}M ) \oplus \mathcal{C}^{\infty}(M, \bigwedge\nolimits^{n}T^{\ast}M ) 
 \end{equation}
given by \begin{equation}\label{deformationoperator}
\operatorname{P}(\Xi)=\left( (\operatorname{exp_{g_{_{\mathfrak{X}}}}}(\mathcal{V}_{\Xi})\circ\Phi)^{\ast}(\omega),   (\operatorname{exp_{g_{_{\mathfrak{X}}}}}(\mathcal{V}_{\Xi})\circ\Phi )^{\ast}(\operatorname{Im}\Omega)\right),
\end{equation}
where $\mathcal{C}^{\infty}(M,T^{\ast}M)\big|_{\mathfrak{A}}$ is the space of differential forms $\Xi$ such that their image under the bundle map $J\circ g^{-1}_{\mathfrak{X}}$ belongs to $\mathfrak{A}$.

The zero set of this operator contains those special Lagrangian deformations of $M$ lying in $U$, that is, 

\[ \operatorname{P}^{-1}(0)=\]
\[\left\lbrace \Xi\in\smooth (M,T^{\ast}M)|_{\mathfrak{A}}:    \operatorname{exp}_{g_{_{\mathfrak{X}}}}(\mathcal{V}_{\Xi})\circ\Phi :M\longrightarrow \mathfrak{X} \quad \text{is a S.L. embedding}\right\rbrace. \]
Hence, the local structure of the moduli space $\mathfrak{M}(M,\Phi)$ near $M$  is given by $\operatorname{P}^{-1}(0)$, the zero set of a non-linear operator. A classical result in non-linear functional analysis that has been used to describe the zero set of non-linear operators in deformation of calibrated submanifolds is the Implicit Function Theorem for Banach spaces  (see, for example,  \cite{lang2} chapter 14, theorem 2.1).

\begin{theorem}\label{IFT}
Let $X$ and $Y$ be Banach spaces, $\mathfrak{A}\subset X$ an open neighborhood of zero and ${\operatorname{P}}:\mathfrak{A}\subset X\longrightarrow Y$ a $\mathcal{C}^{k}$-map such that
\begin{enumerate}[i)]
\item ${\operatorname{P}}(0)=0$ 
\item $\operatorname{DP[0]}: X\longrightarrow Y$ is surjective
\item $\operatorname{DP[0]}$ splits $X$ i.e. $X=\operatorname{Ker}\operatorname{DP[0]}\oplus Z$ for some closed subspace $Z$ 
\end{enumerate}
then there exist open subsets $W_1\subset \operatorname{Ker}\operatorname{DP[0]}$,  $W_2\subset Z$ and a unique $\mathcal{C}^{k}$-map $G:W_1\subset \operatorname{Ker}\operatorname{DP[0]} \longrightarrow W_2\subset Z$ such that 
\begin{enumerate}[i)]
\item $0\in W_1\cap W_2$
\item $W_1\oplus W_2 \subset \mathcal{U}$
\item $\operatorname{P}^{-1}(0)\cap (W_1\oplus W_2)=\lbrace (x,G(x)) : x\in W_1\rbrace .$
\end{enumerate}
\end{theorem}

Ideally, in order to apply theorem \ref{IFT}, we expect to define Banach spaces of differential forms $X$ and $Y$ such that the deformation operator $\operatorname{P}$ acts smoothly,  adapt the tubular neighborhood given by theorem~\ref{tubular} such that an open neighborhood of zero $\mathfrak{A}\subset X$ fits into it. Moreover we would like that with this choice of Banach spaces the linearisation of the deformation operator at zero $\operatorname{DP[0]}$ is a Fredholm operator. Thus its kernel is a finite dimensional space and it splits $X$. Moreover if the cokernel of  $\operatorname{DP[0]}$ vanishes, then theorem~\ref{IFT} applies immediately and gives us that the moduli space $\mathfrak{M}(M,\Phi)$, locally around $M$, is a finite dimensional, smooth manifold with dimension equal to $\operatorname{dim} \operatorname{Ker}DP[0]$. Moreover any infinitesimal deformation, i.e. $x \in W_1\subset \operatorname{Ker}\operatorname{DP[0]}$, can be lifted to an authentic deformation given by $(x,G(x))$ with $\operatorname{P}(x+G(x))=0$ i.e. there are no obstructions. 

This ideal situation turned out to be true in the compact case. In \cite{mclean}, McLean studied  the deformation of a compact special Lagrangian submanifold inside a Calabi-Yau manifold $\mathfrak{X}$. McLean used the classical and well-developed elliptic theory on compact manifolds to analyze the deformation operation and its linearisation.

In order to set up a deformation framework on a singular manifold $M$ with edge singularity, we shall use edge-degenerate differential forms. These forms are dual to the edge-degenerate vector fields in (\ref{edgefield}) section~\ref{examples} with respect to the edge metric $g_{_{M}}=r^{2}g_{_{\Xcal}}+dr^2+g_{_{\edge}}.$ More precisely, let's consider the following space of differential forms $\gamma$ on the stretched manifold $\MM$ such that they vanish on all tangent directions to the fibers on $\partial\MM$:
\begin{center}
$\smooth (T^{\ast}_{\wedge} \MM):=\lbrace \gamma\in\smooth(T^{\ast}\MM): \gamma\vert_{T \mathcal{X}_y}=0$  $\forall y\in\edge\rbrace.$
\end{center}
The space $\smooth (T^{\ast}_{\wedge} \MM)$ is a locally free $\smooth(\MM)$-module. By the Swan theorem \cite{swan}, this is the space of sections of a vector bundle $T^{\ast}_{\wedge} \MM$ over $\MM$. The vector bundle $T^{\ast}_{\wedge} \MM$\label{stretchedcotagent} is called the stretched cotangent bundle of the manifold with edges $M$ (see \cite{naza} section 1.3.1). In local coordinates $\local$   we have  \[\gamma=\mathcal{A}dr+\sum\limits_{k=1}^{m}\mathcal{B}_{k}rd\sigma_{k}
 +\mathcal{C}_{l}du_{l}\] with $\mathcal{A},\mathcal{B}_{k},\mathcal{C}_{l}$ in $\smooth(\MM)$. Observe that these are differential forms that degenerate at each direction tangent to $\Xcal$. We will denote by $T^{\ast}_{\wedge}M$ the restriction of the stretched cotangent bundle to $M$.
 
At this point we have to make an assumption on the vector bundles we consider on a manifold with edge (or conical) singularity.

 \begin{definition}\label{admissibleE}
Let $M$ be a manifold with edge singularity. We say that a vector bundle $E$ over $M$ is admissible if on a collar neighborhood $(0,\varepsilon)\times\Xcal\times\edge$ the restriction $E|_{(0,\varepsilon)\times\Xcal\times\edge}$ is the pull-back of a vector bundle $E_{\Xcal}$ over $\Xcal$.  
\end{definition}
Now let's consider the stretched cotangent bundle $T^{\ast}_{\wedge}M$ as an admissible vector bundle. In order to do that let's define \[E_{\Xcal}:=\wedge^{0}\Xcal \oplus T^{\ast}\Xcal \oplus\underbrace{\wedge^{0}\Xcal\oplus\wedge^{0}\Xcal\oplus ... \oplus\wedge^{0}\Xcal}_\text{q-times}\] where $q=\operatorname{dim}\edge$.

We shall assume that on the collar neighborhood \[[0,\varepsilon)\times\Xcal\times\edge\] the stretched cotangent bundle  $T^{\ast}_{\wedge}M$ is isomorphic to the pull-back vector bundle $\pi^{\ast}_{\Rpos\times\edge}E_{\Xcal}$ where $\pi^{}_{\Rpos\times\edge}$ is the projection \begin{equation}\label{projectionedge}
\pi^{}_{\Rpos\times\edge}:\cone\times\edge\longrightarrow \Xcal
\end{equation} and $\cone=\Rpos\times\Xcal$.
We shall also define the bundle  $E_{\cone}:=\pi^{\ast}_{\Rpos}E_{\Xcal}$ as the pull-back of the bundle $E_{\Xcal}$ by the projection \begin{equation}\label{projectioncone}
\pi^{}_{\Rpos}:\cone\longrightarrow \Xcal.
\end{equation} 
In particular if the edge $\edge$ is a parallelizable manifold then the stretched cotangent bundle $T^{\ast}_{\wedge}M$ is an admissible bundle.

Throughout this section we consider $\CC^{n}$ with its standard Calabi-Yau structure \[(\CC^{n},g_{_{\CC^{n}}},\omega_{_{\CC^{n}}}, \operatorname{}\Omega)\] where $g_{_{\mathbb{C}^n}}= |{dz_1}|^2+\cdots+|{dz_n}|^2$, $\omega_{\mathbb{C}^n}=\frac{\sqrt{-1}}{2}\sum\limits_{i=1}^{n}dz_i \wedge d\bar z _{i}$ and $\Omega=dz_{1}\wedge\cdots \wedge dz_{n}$. We consider $\CC^{n}$ with a fictitious edge structure as follows: \[\CC^{n}=\RR^{n}\oplus\RR^{n}\cong \frac{\overline{\RR}^{+}\times S^{n-1}}{\lbrace 0 \rbrace \times S^{n-1}}\times \RR^{n}.\]
Associated with this edge structure we have the stretched space \[\CC^{n}_{\operatorname{Str}}:=\left( \overline{\RR}^{+} \times S^{n-1}\right)\times \RR^{n}\] such that \[ \CC^{n}_{\operatorname{Str}}\setminus \partial \CC^{n}_{\operatorname{Str}}=\left( \RR^{n}\setminus \lbrace 0 \rbrace \right) \times\RR^{n}\cong\CC^{n}\setminus \left( \lbrace 0 \rbrace \times \RR^{n}\right). \]
  
\subsection{Submanifolds with edge singularities in $\CC^{n}$ }\label{subsectionSubmanifoldswithEdges}
Let $M$ be a compact manifold with edge singularity $\edge$ (see section~\ref{examples}). Then the boundary of the stretched manifold $\MM$ has a $\Xcal$-fibration structure over $\edge$, $\pi:\partial\MM\rightarrow \edge$, where $\Xcal$ and $\edge$ are compact smooth manifolds (without boundary) with  $q=\operatorname{dim}\edge$ and $m=\operatorname{dim}\Xcal$. We assume that $\Xcal$ is diffeomorphic to an embedded submanifold of the sphere $S^{n-1}$ with diffeomorphism given by $\theta:\Xcal\longrightarrow S^{n-1}\subset\RR^{n}$. Consider the cone $\cone$ with cross section $\Xcal$ i.e. $\cone=\Xcal\times\Rpos$ and let's define a diffeomorphism of $\cone$ with a cone $\mathcal{C}\subset\RR^{n}$ by $\psi :\cone\longrightarrow\mathcal{C}\subset \RR^{n}$ where $\psi (r,p):=(r\theta_{1}(p),\dots,r \theta_{n}(p))\in\RR^{n}.$ We shall also assume that $\edge$ is embedded in $\RR^{n}$ by $\tau:\edge\longrightarrow\RR^{n}$. 

\begin{definition}\label{edgeembedding}Let $M$ be a compact manifold with edge singularity $\edge$. 
\begin{enumerate}[i)]
\item A smooth embedding $\Phi:M\longrightarrow \CC^{n}$ is called an edge embedding if on a collar neighborhood $(0,\varepsilon)\times\partial\MM $, which has the structure of a $\cone$-bundle over $\edge$, the embedding $\Phi$ splits as $\Phi(r,p,v)=(\psi(r,p),\tau(v))$ with respect to the identification $\CC^{n}\cong\xy$.    

\item If $\Phi:M\longrightarrow \CC^{n}$ is an edge embedding such that $\Phi(M)$ is a special Lagrangian submanifold of $\CC^{n}$, we say that $(M,\Phi)$ is a special Lagrangian submanifold with edge singularity. 
\end{enumerate}
\end{definition}

\subsection{The Deformation Operator}\label{subsectiondeformationoperator}

Let $\Phi:M\longrightarrow\CC^{n}$ be a compact special Lagrangian submanifold with edge singularity, see definition~\ref{edgeembedding} in section~\ref{subsectionSubmanifoldswithEdges}. In order to study the moduli space of special Lagrangian deformations of $M$ as a manifold with edge singularities, we have to study small deformations of $M$ inside $\CC^{n}$. These deformations are produced by  sections of the normal bundle $\varphi\in\mathcal{N}(M)$ via the exponential map $\operatorname{exp}_{g_{_{\CC^{n}}}}$. The equations \begin{equation}\label{mastereq}
 \begin{dcases*}
\omega_{_{\CC^n}} \big |_{M}   \equiv 0   \\
 \operatorname{Im}\Omega \big |_{M} \equiv 0 
\end{dcases*}
\end{equation} define a first order non-linear partial differential operator $\operatorname{P}$ such that $\varphi$ must satisfy the equation $\operatorname{P}(\varphi)=0$ in order to produce a special Lagrangian deformation (see \eqref{deformationoperator}). Because we are interested in small deformations we can use the Implicit Function Theorem for Banach spaces (if applicable) to describe small solutions of the equation $\operatorname{P}(\varphi)=0$ in terms of solutions of the linearised equation at zero i.e. $\operatorname{DP}[0](\varphi)=0$. In particular, on a collar neighborhood  $(0,\varepsilon)\times\partial\MM$, equipped with the edge metric $g_{_{M}}=r^{2}g_{_{\Xcal}}+dr^{2}+g_{_{\edge}}$, we want to solve the equation $\operatorname{DP[0]}(\varphi)=0$. This is a problem of analysis of PDEs on singular spaces and this observation suggests the approach to follow. First, we expect the operators $\operatorname{P}$ and $\operatorname{DP[0]}$ to be edge-degenerate on $(0,\varepsilon)\times\partial\MM$. This is achieved by using sections of the stretched cotangent bundle $T^{\ast}_{\wedge}M$ to produce small deformations. This is natural as differential forms in  $T^{\ast}_{\wedge}M$ have a degenerate behavior compatible with the edge singularity of $M$ in the sense that their degenerations are induced by the pairing of the edge metric $g_{_{M}}$ with edge-degenerate vector fields. Then, in order to invoke the Implicit Function Theorem for Banach spaces we need that $\operatorname{DP[0]}$ is an elliptic operator in the edge calculus (hence a Fredholm operator by theorem~\ref{ellipticityforedge} in section~\ref{ellipticsection}). This is achieved by completing the edge symbol $\sigma_{\wedge}^{1}(\operatorname{DP[0]})$ with boundary and coboundary conditions as the Atiyah-Bott obstruction vanishes, see section~\ref{linear operator} below.

\subsection{The non-linear deformation operator}\label{subsectionNonLinearOp}

Given a compact special Lagrangian submanifold with edge singularity $\Phi : M\longrightarrow\CC^{n}$, let $\mathcal{N}(M)\subset T(\CC^{n})$ be the normal bundle. By using the identification $\CC^{n}\cong\xy$ we have $T^{\ast}(\CC^{n})\cong T^{\ast}(\xy)$. Now, the complex structure $J$ induces an isomorphism of vector bundles  $J:T(M)\longrightarrow\mathcal{N}(M)$, hence we have a bundle isomorphism $J\circ\Phi_{\ast}\circ g^{\ast}_{_{M}}:T^{\ast}(M)\longrightarrow\mathcal{N}(M)\subset T^{}(\xy)|_{M}$ where $ g^{\ast}_{_{M}}$ is the dual metric on the cotangent bundle $T^{\ast}M$ inducing a bundle map $g_{_{M}}^{\ast}:T^{\ast}M\longrightarrow TM$.

\begin{lemma}\label{localemma}
Let $\Xi\in\smooth(M,T^{\ast}M)$ with local expression in an edge neighborhood $(0,\varepsilon)\times\mathcal{U}\times\Omega\subset\cone\times\edge$, in local coordinates $\coordsimple$,  be given by  \[\Xi\coordsimple=\CA\coordsimple dr+\sum\limits_{k=1}^{m}\CB\coordsimple d\sigma_k+\sum\limits_{l=1}^{q}\cC\coordsimple du_{l},\]
then, its image under the map $J\circ\Phi_{\ast}\circ g^{\ast}_{M}$ is given by the following  expression in the restriction of the tangent bundle $T(\xy)\big|_{(0,\varepsilon)\times\mathcal{U}\times\Omega}$
 \[\mathcal{V}_{\Xi}:=J(\Phi_{\ast}(g^{\ast}_{M}(\Xi)))=\sum\limits_{i=1}^{n}-\tilde{\mathcal{C}}_{i}\coordsimple \partial_{x_i}+(\CA\coordsimple \theta^{\circ}_{i}+\frac{1}{r_{}}\tilde{\mathcal{B}}_{i}\coordsimple)\partial_{y_i}\]
 where $
\tilde {\mathcal{B}}_i$ and $\tilde{\mathcal{C}}_l$ are the components of the corresponding pushforwards \[\theta_{\ast}\left( g^{\ast}_{_{\Xcal}}\left(\sum\limits_{k=1}^{m}\CB\coordsimple d\sigma_k \right)\right),\]

\[\tau_{\ast}\left(g^{\ast}_{_{\edge}}\left(\sum\limits_{l=1}^{q}\cC\coordsimple du_l\right)\right)\]
and $(\theta^{\circ}_1,\dots\,\theta^{\circ}_{n})=\theta(\sigma_1,\dots\sigma_{m})$.
\end{lemma}

\begin{proof}
It follows from the expression of the dual edge metric $g^{\ast}_{_{M}}=\frac{1}{r^2}g^{\ast}_{_{\Xcal}}+\partial_r\otimes\partial_r+g^{\ast}_{_{\edge}}$ that \[g^{\ast}_{_{M}}(\Xi)=\CA\partial_r+\frac{1}{r^2}\sum\limits_{k=1}^{m}\hat\CB \partial_k+\sum\limits_{l=1}^{q}\hat\cC \partial_{u_l}\]
where $\sum\limits_{k=1}^{m}\hat\CB \partial_k=g^{\ast}_{_{\Xcal}}(\sum\limits_{k=1}^{m}\CB d\sigma_k)$ and $\sum\limits_{l=1}^{q}\hat\cC \partial_{u_l}=g^{\ast}_{_{\edge}}(\sum\limits_{l=1}^{q}\cC d u_l).$
Let $p\in (0,\varepsilon)\times\mathcal{U}\times\Omega$ and take a curve $\mathcal{J}:I\subset \RR\longrightarrow M$ given by $\mathcal{J}(t)=(r(t),\sigma_{}(t),u_{}(t))$, such that $\mathcal{J}(0)=p$ and $\mathcal{J}'(0)=g^{\ast}_{_{M}}(\Xi (p))$. 
Then \[\Phi\circ\mathcal{J}:I\subset\RR\longrightarrow\CC^{n}\cong\xy\] defines a curve given by 
\begin{align*}
(\Phi\circ\mathcal{J})(t)&=\Phi(r(t),\sigma_{}(t),u_{}(t))
\\&=(r(t)\theta_{1}(\sigma_{}(t)),\dots,r(t)\theta_{n}(\sigma_{}(t)),\tau_{1}(u_{}(t)),\dots ,\tau_{n}(u_{}(t))),
\end{align*}
therefore
\begin{align*}
& (\Phi\circ\mathcal{J})'(0) 
\\ & =(r'(0)\theta_{1}(\sigma_{k}(0))+r(0)\theta'_{1}(\sigma_{k}(0)),\dots ,r'(0)\theta_{n}(\sigma_{k}(0))+r(0)\theta'_{n}(\sigma_{k}(0))
\\& ,\tau'_{1}(u_{l}(0)),\dots,\tau'_{n}(u_{l}(0)))
\\& =(\CA\theta^{\circ}_{1}+\frac{1}{r}\tilde{\mathcal{B}}_{1},\dots ,\CA\theta^{\circ}_{n}+\frac{1}{r}\tilde{\mathcal{B}}_{n},\tilde{\mathcal{C}}_{1},\dots,\tilde{\mathcal{C}}_{n})\in\xy\cong T_{p}(\xy).
\end{align*}

By applying the standard complex structure $J$ on $\CC^{n}$ under the identification \eqref{complexstructure} we obtain the result. 
\end{proof}

\begin{proposition}\label{localpropedge}
If $\Xi\in\mathcal{W}^{s,\gamma}(M,T^{\ast}_{\wedge}M)$ then $\mathcal{V}_{\Xi}$ belongs to $\mathcal{W}^{s,\gamma}(M,\mathcal{N}(M))$ where $\mathcal{N}(M)$ is endowed with the restriction of the standard flat metric $g_{_{\CC^{n}}}=g_{_{\RR^{2n}}}$.
\end{proposition}

\begin{proof}
We have to prove that $\norm{\mathcal{V}_{\Xi}}_{\mathcal{W}^{s,\gamma}(M,\mathcal{N}(M))}<\infty$. By \eqref{globaledgenorm} we have to estimate near the edge with the edge-Sobolev norm \[\norm{\omega\mathcal{V}_{\Xi}}_{\mathcal{W}^{s,\gamma}(M,\mathcal{N}(M))}\] and away from the edge with the classical Sobolev norm \[\norm{(1-\omega)\mathcal{V}_{\Xi}}_{H^{s}(2\MM,\mathcal{N}(M))}.\]

Let $\lbrace \Omega_{j},\beta_{j}\rbrace$ and $\lbrace \mathcal{U}_{\lambda},\chi_{_{\lambda}} \rbrace$ be  finite coverings of $\edge$ and $\Xcal$ respectively, given by coordinate neighborhoods such that \[\beta_{j}:\Omega_{j}\rightarrow\RR^{q}\] and  \[I\times\chi_{_{\lambda}}:\Rpos\times\mathcal{U}_{\lambda}\longrightarrow\Rpos\times\RR^{m}\] are diffeomorphism and let $\lbrace \phi_{j}\rbrace$ and $\lbrace \varphi_{\lambda} \rbrace$ be corresponding  subordinate partitions of unity. Let $\omega(r)$ be the cut-off function defining the edge-Sobolev space (see definition~\ref{conespace}).
In an edge neighborhood $\Rpos\times\mathcal{U}_{\lambda}\times\Omega_j$ we have \[\Xi\coordsimple=\CA\coordsimple dr+\sum\limits_{k=1}^{m}\CB\coordsimple rd\sigma_k+\sum\limits_{l=1}^{q}\cC\coordsimple du_l,\]
and the fact that $\Xi\in\mathcal{W}^{s,\gamma}(M,T^{\ast}_{\wedge}M)$ implies that $\omega\phi_{j}\varphi_{\lambda}\CA$, $\omega\phi_{j}\varphi_{\lambda}\CB$ and $ \omega\phi_{j}\varphi_{\lambda}\cC$ belong to $\mathcal{W}^{s,\gamma}(M)$. This follows from \eqref{nearconespace} and \eqref{globaledgenorm}. 

By lemma~\ref{localemma} we have \begin{equation}\label{localexpression}
\mathcal{V}_{\Xi}\coordsimple=\sum\limits_{i=1}^{n}-\tilde{\mathcal{C}}_{i}\coordsimple \partial_{x_i}+(\CA\coordsimple \theta^{\circ}_{i}+\tilde{\mathcal{B}}_{i}\coordsimple)\partial_{y_i}.
\end{equation} 
By~\ref{localedgenorm}, near the edge we want to estimate the terms  \[\norm{\omega\phi_{j}\varphi_{\lambda}\tilde{\mathcal{B}_{k}}}^{2}_{\mathcal{W}^{s,\gamma}(\cone\times\RR^{q})}=\int\limits_{\RR^{q}_{\eta}}[\eta]^{2s}\norm{\kappa^{-1}_{[\eta]}\mathcal{F}_{u\rightarrow\eta}(\omega\phi_{j}\varphi_{\lambda}\tilde{\mathcal{B}_{k}})}^{2}_{\mathcal{H}^{s,\gamma}(\cone)}d\eta \]
and
\[\norm{\omega\phi_{j}\varphi_{\lambda}\tilde{\mathcal{C}_{l}}}^{2}_{\mathcal{W}^{s,\gamma}(\cone\times\RR^{q})}=\int\limits_{\RR^{q}_{\eta}}[\eta]^{2s}\norm{\kappa^{-1}_{[\eta]}\mathcal{F}_{u\rightarrow\eta}(\omega\phi_{j}\varphi_{\lambda}\tilde{\mathcal{C}_{l}})}^{2}_{\mathcal{H}^{s,\gamma}(\cone)}d\eta .\]
First, we observe that \begin{align}\label{comparable}
& \norm{\kappa^{-1}_{[\eta]}\mathcal{F}_{u\rightarrow\eta}(\omega\phi_{j}\varphi_{\lambda}\tilde{\mathcal{B}_{k}})}^{2}_{\mathcal{H}^{s,\gamma}(\cone)}
\\& \approx\sum_{\lambda}\norm{ ((I\times\chi_{\lambda})^{\ast})^{-1}\kappa^{-1}_{[\eta]})\varphi_{\lambda}\omega\mathcal{F}_{u\rightarrow\eta}(\phi_{j}\tilde{\mathcal{B}}_{k})(\eta)}_{\mathcal{H}^{s,\gamma}(\Rpos\times\RR^{m})}^{2}
\end{align}
by \eqref{nearconespace}. By lemma~\ref{localemma} $\tilde{\mathcal{B}}_k$ is obtained by  applying $\theta_{\ast}g^{\ast}_{_{\Xcal}}$. This pull-back and push-forward acts locally on the components $\mathcal{B}_{k}$ by multiplying by $g^{ij}_{_{\Xcal}}$ and partial derivatives of the component functions of $\theta:\Xcal\rightarrow S^{n-1}\subset\RR^{n}$. We claim that both of these operations preserve the membership in $\mathcal{W}^{s,\gamma}(M)$. Indeed, in local coordinates  $\mathcal{U}_{\lambda}$ we have  $g^{\ast}_{_{\Xcal}}=\sum\limits_{i,j=1}^{m}g^{ij}_{_{\Xcal}}\partial_{i}\otimes\partial_{j}$ and \[g^{\ast}_{_{\Xcal}}(\sum\limits_{k=1}^{m}\mathcal{B}_{k}\coordsimple d\sigma_{k})=\sum\limits_{j=1}^{m}(\sum\limits_{k=1}^{m}g^{kj}_{_{\Xcal}}\mathcal{B}_{k}\coordsimple)\partial_{j}.\] 
The norm \[\norm{ ((I\times\chi_{\lambda})^{\ast})^{-1}\kappa^{-1}_{[\eta]}\varphi_{\lambda}\omega\mathcal{F}_{u\rightarrow\eta}(\phi_{j}g^{kj}_{_{\Xcal}}\mathcal{B}_{k})(\eta)}_{\mathcal{H}^{s,\gamma}(\Rpos\times\RR^{m})}\]
is equivalent to  \[\norm{  \varphi_{\lambda} g^{kj}_{_{\Xcal}} S_{{\gamma}-\frac{m}{2}} \left(((I\times\chi_{\lambda})^{\ast})^{-1} \kappa^{-1}_{[\eta]}\omega\mathcal{F}_{u\rightarrow\eta}(\phi_{j}\mathcal{B}_{k})(\eta)\right)}_{H^{s}(\RR^{m+1})}\]
by (\ref{S_transf}). The functions $ g^{kj}_{_{\Xcal}}$ are bounded on the support of $\varphi_{\lambda}$, hence they are bounded functions on $\RR^{m+1}$ under the coordinate map $I\times\chi_{\lambda}$. By the general theory of multipliers on Sobolev spaces $H^{s}(\RR^{m+1})$, (\cite{agranovich} theorem 1.9.1 and 1.9.2), multiplication by any bounded function defines a bounded operator, therefore there exists a constant $C_{kj}$ depending only on $g^{kj}_{\Xcal}$ such that  
\begin{align*}
&\norm{ ((I\times\chi_{\lambda})^{\ast})^{-1}\kappa^{-1}_{[\eta]}\varphi_{\lambda}\omega\mathcal{F}_{u\rightarrow\eta}(\phi_{j}g^{kj}_{\Xcal}\mathcal{B}_{k})(\eta)}_{\mathcal{H}^{s,\gamma}(\Rpos\times\RR^{m})}  \\ & \leq C_{kj}
\norm{ ((I\times\chi_{\lambda})^{\ast})^{-1}\kappa^{-1}_{[\eta]}\varphi_{\lambda}\omega\mathcal{F}_{u\rightarrow\eta}(\phi_{j}\mathcal{B}_{k})(\eta)}_{\mathcal{H}^{s,\gamma}(\Rpos\times\RR^{m})}.
\end{align*}
By hypothesis \[\int\limits_{\RR^{q}_{\eta}}[\eta]^{2s}\norm{ ((I\times\chi_{\lambda})^{\ast})^{-1}\kappa^{-1}_{[\eta]}\varphi_{\lambda}\omega\mathcal{F}_{u\rightarrow\eta}(\phi_{j}\mathcal{B}_{k})(\eta)}_{\mathcal{H}^{s,\gamma}(\Rpos\times\RR^{m})}^{2}d\eta<\infty\]
then, by \eqref{comparable}, $g^{\ast}$ preserves $\mathcal{W}^{s,\gamma}(M)$ near the edge.

The push-forward $\theta _{\ast}$ is induced by a diffeomorphism $\theta:\Xcal\longrightarrow S^{n-1}\subset\RR^{n}$. Locally this push-forward acts on the components of vector fields by multiplications by the derivatives of the component functions $\theta_{k}$, hence the same argument as above applies and we conclude that $\omega\phi_{j}\tilde{\mathcal{B}_{k}}\in\mathcal{W}^{s,\gamma}(M).$

Now, the components $\tilde{\mathcal{C}}_{l}$ are obtained  by applying $\tau_{\ast}g^{\ast}_{_{\edge}}$ to the components  $\mathcal{C}_{l}$. Given $g^{\ast}_{_{\edge}}=\sum g^{ij}_{_{\edge}}(u)\partial_{u_i}\otimes\partial_{u_j}$, it acts on the components $\mathcal{C}_{l}$ by multiplication by $g^{ij}_{_{\edge}}(u)$.

In the same way the push-forward $\tau_{\ast}$ acts through multiplication by derivatives of its components $\frac{\partial\tau^{i}}{\partial u_{l}}$. When composed with the coordinate function $\beta_{k}$, the maps \[(\phi_{k}g^{ij}_{_{\edge}})\circ\beta^{-1}_{k}:\RR^{q}\longrightarrow\RR\] and  \[(\phi_{k}\frac{\partial\tau^{i}}{\partial u_{l}})\circ\beta^{-1}_{k}:\RR^{q}\longrightarrow\RR\]
belong to $\mathcal{S}(\RR^{q})$ as they have compact support. By \cite{schulze3} theorem 1.3.34, multiplication by  an element in $\mathcal{S}(\RR^{q})$ defines a continuous operator on $\mathcal{W}^{s,\gamma}(\cone\times\RR^{q})$. Hence, by the same argument as in the first part of the proof, $\tau_{\ast}g^{\ast}_{_{\edge}}$ preserves  $\mathcal{W}^{s,\gamma}(M)$ near the edge and $\omega\phi_{k}\varphi_{\lambda}\tilde{\mathcal{C}}_{l}\in\mathcal{W}^{s,\gamma}(M)$.

Away from the edge on the compact manifold $M\setminus \left( (0,\varepsilon)\times\partial\MM\right)$ take a finite covering of coordinate neighborhoods $\lbrace W_{i}  \rbrace$ with a subordinate partition of unity $\lbrace  \mu_{i} \rbrace$. Then 

\begin{equation}
 \norm{(1-\omega)\mathcal{V}_{\Xi}}^{2}_{H^{s}(2\MM,\mathcal{N}(M))} =\sum_{i}\norm{\mu_{i}\mathcal{V}_{\Xi}}^{2}_{H^{s}(\RR^{n},\RR^{q})}.
\end{equation} As on each of those patches of local coordinates the support of $\mu_i\mathcal{V}_{\Xi}$ is compact, clearly $\norm{(1-\omega)\mathcal{V}_{\Xi}}^{2}_{H^{s}(2\MM,\mathcal{N}(M))}<\infty.$
\end{proof}

\begin{proposition} \label{localprop}
Let $\Xi\in\smooth_{0}(T^{\ast}_{\wedge}M)$  and  $\mathcal{V}_{\Xi}=J(\Phi_{\ast}( g_{M}^{\ast}(\Xi)))\in\smooth (\mathcal{N}(M)).$ Then the pull-back of the standard K{\"a}hler form $\omega_{_{\CC^{n}}}$ by the map $\operatorname{exp}(\mathcal{V}_{\Xi})\circ\Phi$ is given in terms of edge-degenerate differential operators of order 1 by the following expression in a neighborhood $(0,\varepsilon)\times\mathcal{U}_\lambda\times\Omega_j$ near the edge:
\begin{align*}
& \left(\operatorname{exp}(\mathcal{V}_{\Xi})\circ\Phi\right)^{\ast}(\omega_{\CC^{n}}) =
 \\& \sum\limits_{k=1}^{m}\Bigg( \sum\limits_{i=1}^{n} \operatorname{P}^{}_{k,i}(\mathcal{A})+\operatorname{Q}^{}_{k,i}(\tilde{\mathcal{B}}_{i})+(\operatorname{S}^{}_{k,i}(\mathcal{A})+\operatorname{T}^{}_{k,i}(\tilde{\mathcal{B}}_{i}))\operatorname{R}^{}_{k,i}(\tilde{\mathcal{C}_{i}})
  \\& +(\operatorname{L}^{}_{k,i}(\mathcal{A})+\operatorname{M}^{}_{k,i}(\tilde{\mathcal{B}}_{i}))\operatorname{O}^{}_{k,i}(\tilde{\mathcal{C}}_{i}) \Bigg)rdr\wedge d\sigma_{k}
  \\& + \sum\limits_{l=1}^{q}\Bigg( \sum\limits_{i=1}^{n} \operatorname{P}^{}_{l,i}(\mathcal{A})+\operatorname{Q}^{}_{l,i}(\tilde{\mathcal{B}}_{i})+(\operatorname{S}^{}_{l,i}(\mathcal{A})+\operatorname{T}^{}_{l,i}(\tilde{\mathcal{B}}_{i}))\operatorname{R}^{}_{l,i}(\tilde{\mathcal{C}_{i}})
  \\&+(\operatorname{L}^{}_{l,i}(\mathcal{A})+\operatorname{M}^{}_{l,i}(\tilde{\mathcal{B}}_{i}))\operatorname{O}^{}_{l,i}(\tilde{\mathcal{C}}_{i})\Bigg)dr\wedge d u_{l}
  \\& +\sum\limits_{k=1}^{m}\sum\limits_{l=1}^{q}\Bigg( \sum\limits_{i=1}^{n} \operatorname{P}^{}_{k,l,i}(\mathcal{A})+\operatorname{Q}^{}_{k,l,i}(\tilde{\mathcal{B}}_{i})+\operatorname{U}_{k,l,i}(\tilde{\mathcal{C}}_{i})+(\operatorname{S}^{}_{k,l,i}(\mathcal{A})
  \\& +\operatorname{T}^{}_{k,l,i}(\tilde{\mathcal{B}}_{i}))\operatorname{R}^{}_{k,l,i}(\tilde{\mathcal{C}_{i}})+(\operatorname{L}^{}_{k,l,i}(\mathcal{A})+\operatorname{M}^{}_{k,l,i}(\tilde{\mathcal{B}}_{i}))\operatorname{O}^{}_{k,l,i}(\tilde{\mathcal{C}}_{i})\Bigg) rd\sigma_{k}\wedge d u_{l}
  \\&+\sum\limits_{k=1}^{m}\sum\limits_{j=1}^{m}\Bigg( \sum\limits_{i=1}^{n}\operatorname{P}^{}_{k,j,i}(\mathcal{A})+\operatorname{Q}^{}_{k,j,i}(\tilde{\mathcal{B}}_{i})+(\operatorname{S}^{}_{k,j,i}(\mathcal{A})+\operatorname{T}^{}_{k,j,i}(\tilde{\mathcal{B}}_{i}))\operatorname{R}^{}_{k,j,i}(\tilde{\mathcal{C}_{i}})
\\& + (\operatorname{L}^{}_{k,j,i}(\mathcal{A})+\operatorname{M}^{}_{k,j,i}(\tilde{\mathcal{B}}_{i}))\operatorname{O}^{}_{k,j,i}(\tilde{\mathcal{C}}_{i})\Bigg)r^{2}d\sigma_{k}\wedge d \sigma_{j}
\end{align*}

 \begin{align*} 
\\& + \sum\limits_{s=1}^{q}\sum\limits_{l=1}^{q}\Bigg( \sum\limits_{i=1}^{n}\operatorname{P}^{}_{s,j,i}(\tilde{\mathcal{C}}_i)+(\operatorname{S}^{}_{s,j,i}(\mathcal{A})+\operatorname{T}^{}_{s,j,i}(\tilde{\mathcal{B}}_{i}))\operatorname{R}^{}_{s,j,i}(\tilde{\mathcal{C}_{i}})
\\& +(\operatorname{L}^{}_{s,j,i}(\mathcal{A})+\operatorname{M}^{}_{s,j,i}(\tilde{\mathcal{B}}_{i}))\operatorname{O}^{}_{s,j,i}(\tilde{\mathcal{C}}_{i})\Bigg) du_{s}\wedge d u_{l}
\end{align*}
where $\mathcal{A},\tilde{\mathcal{B}}_{i},\tilde{\mathcal{C}}_{i}$, $i=1,\dots,n$ are the components of $\mathcal{V}_{\Xi}$ in a neighborhood of the edge  as in \eqref{localexpression} and  the operators $\operatorname{L}_{\bullet},\operatorname{M}_{\bullet},\operatorname{O}_{\bullet},\operatorname{P}_{\bullet},\operatorname{Q}_{\bullet},\operatorname{R}_{\bullet},\operatorname{S}_{\bullet},\operatorname{T}_{\bullet},\operatorname{U}_{\bullet}$  belong to $\operatorname{Diff}_{\operatorname{edge}}^{1}(M).$
\end{proposition}

\begin{proof}
Let $\Xi=\mathcal{A}dr+\sum\limits_{k=1}^{m}\mathcal{B}_{k}rd\sigma_{k}+\sum\limits_{l=1}^{q}\mathcal{C}_{l}du_{l}$ in local coordinates near the edge, then
\[\mathcal{V}_{\Xi}=J\Phi_{\ast}( g_{M}^{\ast}(\Xi))=\sum\limits_{i=1}^{n}-\tilde{\mathcal{C}}_{i}\coordsimple\partial_{x_i}+(\CA\coordsimple\theta^{\circ}_{i}+
 \tilde{\mathcal{B}}_{i}\coordsimple)\partial_{y_i}\]
 and
\[
(\operatorname{exp}( \mathcal{V}_{\Xi})\circ\Phi)\coordsimple =(r\theta_1-\tilde{\mathcal{C}_1} ,\dots ,r\theta_{n}-\tilde{\mathcal{C}}_{n},  \tau_1+\CA \theta_{1}+\tilde{\mathcal{B}}_{1},\dots,\tau_{n}+\CA \theta_{n}+\tilde{\mathcal{B}}_{n}).
\]
The standard K{\"a}hler form in $\CC^{n}$ is given by 
\[\omega_{_{\CC^{n}}}=\sum\limits_{i=1}^{n}dx_{i}\wedge dy_{i},\]
then a direct computation shows that 
 \begin{align*}
&(\operatorname{exp}(\mathcal{V}_{\Xi})\circ\Phi )^{\ast}(\omega_{\CC^{n}})=
 \sum\limits_{i=1}^{n} (\operatorname{exp}(\mathcal{V}_{\Xi})\circ\Phi )^{\ast}( dx_{i}\wedge dy_{i})
   \\& =\sum\limits_{i=1}^{n}d\left( r\theta_{i}(\sigma)-\tilde{\mathcal{C}}_{i}\coordsimple\right)\wedge d\left( \tau_i(u_l)+\CA \theta_{i}(\sigma)+\tilde{\mathcal{B}}_{i}\right) 
    \\& =\sum\limits_{k=1}^{m}\bigg( \sum\limits_{i=1}^{n}\frac{1}{r}\left( \theta_{i}\partial_{k}(\theta_i)+\theta_i\partial_{k}+\theta_{i}\partial_{k}(\theta_i)(-r\partial_{r})   \right)(\CA)
    \\&+\frac{1}{r}\left( \theta_{i}\partial_{k}+\partial_{k}(\theta_i)(-r\partial_{r})(\tilde{\mathcal{B}}_{i}) \right)  
    - \frac{1}{r}\left(\left( \partial_{k}(\theta_i)+\theta_{i}\partial_{k}\right)(\CA)+\partial_{k}(\tilde{\mathcal{B}}_{i})\right)\partial_{r}(\tilde{\mathcal{C}}_{i})
   \\&+\frac{1}{r}\partial_{k}(\tilde{\mathcal{C}})\big(\theta_{i}\partial_{r}(\CA) 
    + \partial_{r}(\tilde{\mathcal{B}}_{i})  \big)\bigg)rdr\wedge d\sigma_k 
\end{align*}

 \begin{align*}    
   \\&  +\sum\limits_{l=1}^{q}\bigg(\sum\limits_{i=1}^{n} \theta^{2}_{i}\partial_{u_{l}}(\CA)+\theta_{i}\partial_{u_{l}}(\tilde{\mathcal{B}}_{i})-\partial_{u_{l}}(\tau_{i})\partial_{r}(\tilde{\mathcal{C}}_{i})  
   \\&-\partial_{r}(\tilde{\mathcal{C}}_{i})\left( \theta_{i}\partial_{u_{l}}(\CA) +\partial_{u_{l}}(\tilde{\mathcal{B}}_i)  \right)+\partial_{u_{l}}(\tilde{\mathcal{C}}_i)\left( \theta_{i}\partial_{r}(\CA) +\partial_{r}\tilde{\mathcal{B}}_i \right) \bigg)dr\wedge du_{l} 
    \\&+\sum\limits_{k=1}^{m}\sum\limits_{l=1}^{q}\bigg(\sum\limits_{i=1}^{n} \frac{1}{r}\left(\partial_{k}(\theta_{i})\theta_i (r\partial_{u_{l}})\right)(\CA)+\frac{1}{r}\left(\partial_{k}(\theta_{i}) (r\partial_{u_{l}})\right)(\tilde{\mathcal{B}}_{i})
    \\& -\frac{1}{r}\partial_{u_{l}}(\tau_{i})\partial_{k}(\tilde{\mathcal{C}}_i)-\frac{1}{r}\partial_{k}(\tilde{\mathcal{C}}_i)\left(\theta_{i}\partial_{u_{l}}(\CA)+\partial_{u_{l}}(\tilde{\mathcal{B}}_i)\right)
    \\&+\partial_{u_{l}}(\tilde{\mathcal{C}}_i)\left( (\partial_{k}(\theta_i) +\theta_{i}\partial_{k})(\CA)+\partial_{k}(\tilde{\mathcal{B}}_i) \right)\bigg)rd\sigma_{k}\wedge du_{l}     
   \\&+\sum\limits_{j=1}^{m}\sum\limits_{k=1}^{m}\Bigg(\sum\limits_{i=1}^{n}\frac{1}{r}\left(\theta_i \partial_{j}(\theta_i)\partial_{k}-\theta_i\partial_{k}(\theta_i)\partial_{j} \right)(\CA) 
      \\& +\frac{1}{r}\left( \partial_{j}(\theta_i)\partial_{k}-\partial_{k}(\theta_i)\partial_{j} \right)(\tilde{\mathcal{B}}_i)
      \\&-\frac{1}{r}\partial_{j}(\tilde{\mathcal{C}}_i)\left(\frac{1}{r}\left(\partial_{k}(\theta_i)+\theta_i\partial_k \right)(\CA)+\frac{1}{r}\partial_k(\tilde{\mathcal{B}}_i)  \right)
    \\&+ \frac{1}{r}\partial_{k}(\tilde{\mathcal{C}}_i)\left(\frac{1}{r}\left(\partial_{j}(\theta_i)+\theta_i\partial_j \right)(\CA)+\frac{1}{r}\partial_j(\tilde{\mathcal{B}}_i)  \right)\Bigg) r^{2}d\sigma_{j}\wedge d \sigma_{k}
 \\&+\sum\limits_{\lambda =1}^{q}\sum\limits_{l =1}^{q}\Bigg(\sum\limits_{i=1}^{n} \left( -\partial_{u_{l}}(\tau_i)\partial_{u_{\lambda}}+ \partial_{u_{\lambda}}(\tau_i)\partial_{u_{l}} \right)(\tilde{\mathcal{C}}_i)
\\& -\partial_{u_{\lambda}}(\tilde{\mathcal{C}}_i)\left( \theta_i\partial_{u_{l}}(\CA)+\partial_{u_{l}}(\tilde{\mathcal{B}}_i)  \right)  + \partial_{u_{l}}(\tilde{\mathcal{C}}_i)\left( \theta_i\partial_{u_{\lambda}}(\CA)+\partial_{u_{\lambda}}(\tilde{\mathcal{B}}_i)  \right)\Bigg)du_{\lambda}\wedge du_{l}.
\end{align*}

  Each of these terms are edge-degenerate differential operators acting on the components of $\Xi$ and products of these as it is claimed in the proposition. Note that we have used the fact that \[\left(\operatorname{exp}(\mathcal{V}_{0})\circ\Phi\right)^{\ast}\left( \omega_{\CC^{n}} \right)=0\]
to remove  products in each term that do not contain any of the component functions $\CA$, $\tilde{\mathcal{B}}_{i}$ and $\tilde{\mathcal{C}}_{i}.$ 
\end{proof}

\begin{corollary}\label{coro1}
The map \[\operatorname{P}_{\omega_{\CC^{n}}}: \smooth_{0}(T^{\ast}_{\wedge}M)
\longrightarrow \smooth_{0}(M,\bigwedge\nolimits^{2}T^{\ast}_{\wedge}M) \] defined by $\operatorname{P}_{\omega_{_{\CC^n}}}(\Xi):=\left(\operatorname{exp}(\mathcal{V}_{\Xi})\circ\Phi\right)^{\ast}(\omega_{_{\CC^n}})$,
extends to a continuous non-linear operator \[\operatorname{P}_{\omega_{\CC^{n}}}: \mathcal{W}^{s,\gamma}(M,T^{\ast}_{\wedge}M)
\longrightarrow  \mathcal{W}^{s-1,\gamma-1}(M,\bigwedge\nolimits^{2}T^{\ast}_{\wedge}M)\]
for $s>\frac{\operatorname{dim}\Xcal+\operatorname{dim}\edge+3}{2}$ and $\gamma>\frac{\operatorname{dim}\Xcal+1}{2}.$  
\end{corollary}

\begin{proof}
Let's consider a sequence $\lbrace \Xi_{i}\rbrace_{i\in\mathbb{N}} \subset\smooth_{0}(T^{\ast}_{\wedge}M)$ such that it is a Cauchy sequence in $\mathcal{W}^{s,\gamma}(M,T^{\ast}_{\wedge}M).$ Then, in a neighborhood near the edge, the components of the elements of the sequence $\lbrace \Xi_{i} \rbrace_{i\in\mathbb{N}}$ define Cauchy sequences in $\mathcal{W}^{s,\gamma}(M)$ i.e. 
\[\norm{\omega\phi_{j}\varphi_{\lambda}\mathcal{A}^{i}-\omega\phi_{j}\varphi_{\lambda}\mathcal{A}^{j}}_{\mathcal{W}^{s,\gamma}(M,T^{\ast}_{\wedge}M)}<\varepsilon,\]
 \[\norm{\omega\phi_{j}\varphi_{\lambda}\tilde{\mathcal{B}}_{k}^{i}-\omega\phi_{j}\varphi_{\lambda}\tilde{\mathcal{B}}_{k}^{j}}_{\mathcal{W}^{s,\gamma}(M,T^{\ast}_{\wedge}M)}<\varepsilon \text{    for all } k=1,2,\dots ,m\] 
 and
 \[\norm{\omega\phi_{j}\varphi_{\lambda}\tilde{\mathcal{C}}_{l}^{i}-\omega\phi_{j}\varphi_{\lambda}\tilde{\mathcal{C}}_{l}^{j}}_{\mathcal{W}^{s,\gamma}(M,T^{\ast}_{\wedge}M)}<\varepsilon \text{    for all } l=1,2,\dots ,q\]
for all $i,j>N(\varepsilon).$
Away from the edge we have \[\norm{(1-\omega)(\Xi_{i}-\Xi_{j})}_{H^{s}(2\MM,T^{\ast}_{}M)}< \varepsilon.\]
We want to estimate \[\norm{\operatorname{P}_{\omega_{\CC^{n}}}(\Xi_{i})-\operatorname{P}_{\omega_{\CC^{n}}}(\Xi_{j})}_{\mathcal{W}^{s-1,\gamma-1}(M,T^{\ast}_{\wedge}M)}.\]  
  Observe that the conditions on $s$ and $\gamma$ imply that the edge-Sobolev spaces are Banach algebras \eqref{Banachalg}, hence multiplication is well-defined and we have the estimate $\norm{fg}_{s,\gamma}\leq C_{s,\gamma}\norm{f}_{s,\gamma}\norm{g}_{s,\gamma}$ with a constant $ C_{s,\gamma}$ depending only on $s$ and $\gamma$. To simplify  the notation we will use $\norm{\cdot}_{s,\gamma}$ to denote $\norm{\cdot}_{\mathcal{W}^{s,\gamma}(M,T^{\ast}_{\wedge}M)}.$
  
Now, near the edge, the components of the operator $\operatorname{P}_{\omega_{\CC^{n}}}$ are given by expressions of the form 
$\operatorname{P}+\operatorname{Q}+(\operatorname{S}+\operatorname{T})\cdot \operatorname{R}$ where $\operatorname{P},\operatorname{Q},\operatorname{R},\operatorname{S},\operatorname{T}\in \operatorname{Diff}^{1}_{\operatorname{edge}}(M).$ By estimating one of these expressions we can apply the same argument to all components. 

Given $\operatorname{P},\operatorname{Q},\operatorname{R},\operatorname{S},\operatorname{T}$ in $\operatorname{Diff}^{1}_{\operatorname{edge}}(M)$ and $\mathcal{A,B,C,A',B',C'}$ in $\mathcal{W}^{s,\gamma}(M)$ by the continuity these operators (proposition~\ref{cont.edge}) and the elementary identity $ab-a'b'=\frac{1}{2}(a+a')(b-b')+\frac{1}{2}(a-a')(b+b')$ we have
\begin{align*}
& \parallel\operatorname{P}(\CA)+\operatorname{Q}(\mathcal{B})+(\operatorname{S}(\CA)+\operatorname{T}(\mathcal{B}))\operatorname{R}(\mathcal{C})
 -\left(\operatorname{P}(\CA')+\operatorname{Q}(\mathcal{B'})\right.
 \\& \left.+(\operatorname{S}(\CA')+\operatorname{T}(\mathcal{B'}))\operatorname{R}(\mathcal{C'})\right) \parallel_{s-1,\gamma-1}
\\& \leq \norm{\operatorname{P}}\norm{\CA-\CA'}_{s\gamma}+\norm{\operatorname{Q}}
\norm{\mathcal{B}-\mathcal{B}'}_{s,\gamma}\\& + \norm{(\operatorname{S}(\CA)+\operatorname{T}(\mathcal{B}))\operatorname{R}(\mathcal{C})-(\operatorname{S}(\CA')+\operatorname{T}(\mathcal{B'}))\operatorname{R}(\mathcal{C'})}_{s-1,\gamma-1}
\\& \leq \norm{\operatorname{P}}\norm{\CA-\CA'}_{s,\gamma}+\norm{\operatorname{Q}}\norm{\mathcal{B}-\mathcal{B}'}_{s,\gamma}\\&+\frac{1}{2}\left(\norm{\operatorname{R}}\norm{\operatorname{S}}\norm{\CA-\CA'}_{s,\gamma}+\norm{\operatorname{R}}\norm{\operatorname{T}}\norm{\mathcal{B}-\mathcal{B}'}_{s,\gamma}\right)\norm{ \mathcal{C}+\mathcal{C}' }_{s,\gamma}
\\&+ \frac{1}{2}\left( \norm{\operatorname{R}}\norm{\operatorname{S}}\norm{\CA+\CA'}_{s,\gamma}+\norm{\operatorname{R}}\norm{\operatorname{T}}\norm{\mathcal{B}+\mathcal{B}'}_{s,\gamma}\right)\norm{ \mathcal{C}-\mathcal{C}' }_{s,\gamma}.
\end{align*}

Therefore, if $\lbrace \Xi_{i} \rbrace $ is a Cauchy sequence in $\mathcal{W}^{s,\gamma}(M, T^{\ast}_{\wedge}M)$ then  \[\lbrace \operatorname{P}(\CA_i)+\operatorname{Q}(\mathcal{B}_i)+(\operatorname{S}(\CA_i)+\operatorname{T}(\mathcal{B}_i))\operatorname{R}(\mathcal{C}_i) \rbrace_{i\in\mathbb{N}}\] is a Cauchy sequence in  $\mathcal{W}^{s-1,\gamma-1}(M)$  which implies that $\lbrace \omega(r) \operatorname{P}_{\omega_{\CC^{n}}}(\Xi_{i})\rbrace_i$ is Cauchy too.

Away from the edge we are in the setting of the standard Sobolev spaces $H^{s}(2\MM,T^{\ast}_{\wedge}M)$ and the operator $ \operatorname{P}_{\omega_{\CC^{n}}}$ is given by products of two differential operators of order 1. By using their continuity on $H^{s}(2\MM,T^{\ast}_{}M)$, the fact that standard Sobolev spaces are  Banach algebras for $s>\frac{\operatorname{dim}M}{2}$ and  a similar argument give us that  $\lbrace (1-\omega) \operatorname{P}_{\omega_{\CC^{n}}}(\Xi_{i})\rbrace_i$ is Cauchy in $H^{s}(2\MM,T^{\ast}_{}M)$. Therefore the corollary follows immediately.
\end{proof}

\begin{proposition} 
Let $\Xi\in\smooth_{0}(T^{\ast}_{\wedge}M)$  and  $\mathcal{V}_{\Xi}=J\Phi_{\ast}( g_{M}^{\ast}(\Xi))\in\smooth (\mathcal{N}(M)).$ Then, on a neighborhood $(0,\varepsilon)\times\mathcal{U}_{\lambda}\times\Omega_j$ near the edge, the pull-back of the imaginary part of the holomorphic volume form in $\CC^{n}$ \[\operatorname{Im}\Omega=\operatorname{Im}(dz_1\wedge\dots\wedge dz_{n}),\] by the map $\operatorname{exp}(\mathcal{V}_{\Xi})\circ\Phi$ is given as a sum of $n$ products of the form \[\operatorname{P}_{i_{1}}(F_{i_{1}})\operatorname{P}_{i_{2}}(F_{i_{2}})\cdot\dots\cdot \operatorname{P}_{i_{n}}(F_{i_{n}})\]
where $\operatorname{P}_{i_{j}}\in\operatorname{Diff}^{1}_{\operatorname{edge}}(M)$ and $F_{i_{j}}\in\mathcal{W}^{s,\gamma}(M)$.
\end{proposition}

\begin{proof}
The holomorphic volume form in $\CC^{n}$ is given by 

\begin{align*}
\operatorname{Im}\Omega=\operatorname{Im}(dz_1\wedge\dots\wedge dz_{n})&=\operatorname{Im} d(x_1+iy_1)\wedge\dots\wedge d(x_{n}+iy_{n})
\\& =\sum\limits_{\abs{I}=odd}  c_{I}dy_{I}\wedge dx_{1}\wedge\dots\wedge\widehat{dx_{I}}\wedge\dots dx_{n}
\end{align*}
where the sum is taken over all increasingly ordered multi-indexes $I$ of odd length $k$, $I=(i_1 , \dots, i_k )$, the hat means that we omit the corresponding terms and  $c_{I}=\pm1. $
Then
\begin{align*}
  &(\operatorname{exp}(\mathcal{V}_{\Xi})\circ\Phi)^{\ast}\left( \operatorname{Im}\Omega \right)  =
  \\&\sum\limits_{\abs{I}=odd}  c_{I}(\operatorname{exp}(\mathcal{V}_{\Xi})\circ\Phi)^{\ast}\left(dy_{I}\wedge dx_{1}\wedge\dots\wedge\widehat{dx_{I}}\wedge\dots dx_{n}\right)
\\&=\sum\limits_{\abs{I}=odd}  c_{I}d(\tau_{i_{1}}+\CA\theta_{i_{1}}+\tilde{\mathcal{B}}_{i_{1}})\wedge\dots\wedge
d(\tau_{i_{k}}+\CA\theta_{i_{k}}+\tilde{\mathcal{B}}_{i_{k}})
 \wedge d (r\theta_1-\tilde{\mathcal{C}}_1)\wedge\dots  
 \\& \wedge\widehat{d(r\theta_I-\tilde{\mathcal{C}}_I)}
\wedge\dots\wedge d (r\theta_{n}-\tilde{\mathcal{C}}_{n}).
\end{align*}
Each of the terms in the sum is a $n$-form on $(0, \varepsilon)\times\mathcal{U}_{\lambda}\times\Omega_j$, hence
\begin{align*}
& c_{I}(\operatorname{exp}(\mathcal{V}_{\Xi})\circ\Phi)^{\ast}\left(dy_{I}\wedge dx_{1}\wedge\dots\wedge\widehat{dx_{I}}\wedge\dots dx_{n}\right)
\\& =F_{I}\coordsimple r^{m}dr\wedge d\sigma_{1}\wedge\dots\wedge d\sigma_{m}\wedge du_{1}\wedge\dots\wedge du_{q},
\end{align*}
where $ F_{I}\coordsimple$ is the determinant of the following matrix
 \[\begin{bsmallmatrix}
 \partial_{r}(\tau_{i_{1}}+\CA\theta_{i_{1}}+\tilde{\mathcal{B}}_{i_{1}}) & \dots  & \partial_{r}(r\theta_{1}-\tilde{\mathcal{C}}_{1}) & \dots & \widehat{\partial_{r}(r\theta_{I}-\tilde{\mathcal{C}}_{I})} & \dots & \partial_{r}(r\theta_{n}-\tilde{\mathcal{C}}_{n}) \\ 
 \vdots & \dots & \vdots & \dots & \vdots & \dots & \vdots \\ 
  \frac{1}{r}\partial_{j}(\tau_{i_{1}}+\CA\theta_{i_{1}}+\tilde{\mathcal{B}}_{i_{1}}) & \dots  & \frac{1}{r}\partial_{j}(r\theta_{1}-\tilde{\mathcal{C}}_{1}) & \dots & \widehat{\frac{1}{r}\partial_{j}(r\theta_{I}-\tilde{\mathcal{C}}_{I})} & \dots & \frac{1}{r}\partial_{j}(r\theta_{n}-\tilde{\mathcal{C}}_{n}) \\ 
 \vdots & \dots & \vdots & \dots & \vdots & \dots & \vdots \\ 
  \partial_{u_{q}}(\tau_{i_{1}}+\CA\theta_{i_{1}}+\tilde{\mathcal{B}}_{i_{1}}) &\dots & \partial_{u_{q}}(r\theta_{1}-\tilde{\mathcal{C}}_{1}) & \dots & \widehat{\partial_{u_{q}}(r\theta_{I}-\tilde{\mathcal{C}}_{I})} & \dots & \partial_{u_{q}}(r\theta_{n}-\tilde{\mathcal{C}}_{n}) \\ 
 \end{bsmallmatrix}.\]
Therefore $F_{I}\coordsimple$ is the sum of  products of the form \[\operatorname{P}_1(\tau_{i_{1}}+\CA\theta_{i_{1}}+\tilde{\mathcal{B}}_{i_{1}})\dots \operatorname{P}_{n}(r\theta_{n}-\tilde{\mathcal{C}}_{n})\]
with all the  operators $\operatorname{P}_k$, $k=1,\dots,n,$ in $\operatorname{Diff}^{1}_{\operatorname{edge}}(M).$ By expanding these products and observing that the sum of all terms of the form \[\operatorname{P}_1(\tau_{i_{1}})\dots \operatorname{P}_{n}(r\theta_{n}),\]
i.e. all those products not containing any of the functions $\mathcal{A},\tilde{\mathcal{B}}_{\bullet}$ or $\tilde{\mathcal{C}}_{\bullet}$, is equal to zero as \[\left(\operatorname{exp}(\mathcal{V}_{0})\circ\Phi\right)^{\ast}\left( \operatorname{Im}\Omega \right)= 0;\]
we obtain that $F_{I}\coordsimple$ is the sum of products of the form \begin{equation}\label{secondnonlinearterm}
\operatorname{P}_{i_{1}}(F_{i_{1}})\operatorname{P}_{i_{2}}(F_{i_{2}})\cdot\dots\cdot \operatorname{P}_{i_{n}}(F_{i_{n}})
\end{equation} 
where $\operatorname{P}_{i_{j}}\in\operatorname{Diff}^{1}_{\operatorname{edge}}(M)$ and $F_{i_{j}}\in\mathcal{W}^{s,\gamma}(M)$ as claimed.
\end{proof}

\begin{corollary}\label{coro2}
The map \[\operatorname{P}_{\operatorname{Im}\Omega}: \smooth_{0}(T^{\ast}_{\wedge}M)
\longrightarrow  \smooth_{0}(M,\bigwedge\nolimits^{n}T^{\ast}_{\wedge}M) \] given by $\operatorname{P}_{\operatorname{Im}\Omega}(\Xi):=\left( \operatorname{exp}(\mathcal{V}_{\Xi})\circ\Phi\right)^{\ast}(\operatorname{Im}(\Omega))$ 
extends to a continuous non-linear operator \[\operatorname{P}_{\operatorname{Im}\Omega}: \mathcal{W}^{s,\gamma}(M,T^{\ast}_{\wedge}M)
\longrightarrow  \mathcal{W}^{s-1,\gamma-1}(M,\bigwedge\nolimits^{n}T^{\ast}_{\wedge}M)\] for $s>\frac{\operatorname{dim}\Xcal+\operatorname{dim}\edge+3}{2}$ and $\gamma>\frac{\operatorname{dim}\Xcal+1}{2}.$ 
\end{corollary}

\begin{proof}
Given $\Xi$ and $\Xi'$ in $\smooth_{0}(T^{\ast}_{\wedge}M)$, near the edge, we have 
 \[\operatorname{P}_{\operatorname{Im}\Omega}(\Xi)-\operatorname{P}_{\operatorname{Im}\Omega}(\Xi')=\]
\begin{align*}
&\left(\sum\limits_{I}    \operatorname{P}_{i_{1}}(F_{i_{1}})\dots \operatorname{P}_{i_{n}}(F_{i_{n}}) -\operatorname{P}_{i_{1}}(F_{i_{1}}')\dots\operatorname{P}_{i_{n}}(F_{i_{n}}') \right)r^{m}dr\wedge d\sigma_{1}\dots
\\&\dots\wedge d\sigma_{m}\wedge du_1\wedge\dots\wedge du_{q}.
\end{align*} 
 
By applying the elementary identity $ab-a'b'=\frac{1}{2}(a+a')(b-b')+\frac{1}{2}(a-a')(b+b')$ we can decompose each term of the sum in factors, each of them containing one of the subtractions
\begin{equation}\label{subtraction}
{P}_{i_{n-k}}(F_{i_{n-k}})-\operatorname{P}_{i_{n-k}}(F_{i_{n-k}})
\end{equation}
for $k=0,\dots,n$.

Now, if  $\lbrace \Xi_{i}\rbrace_{i\in\mathbb{N}}\subset \smooth_{0}(T^{\ast}_{\wedge}M)$ is a Cauchy sequence in the Sobolev space $\mathcal{W}^{s,\gamma}(M, T^{\ast}_{\wedge}M)$ as in corollary~\ref{coro1}, the sequences $\operatorname{P}_{i_{j}}(F^{(k)}_{i_{j}})$ are Cauchy in the space $\mathcal{W}^{s-1,\gamma-1}(M)$ (by the continuity of edge-degenerate operators on edge-Sobolev spaces, see proposition~\ref{cont.edge})  and \[\norm{\operatorname{P}_{i_{n-k}}(F^{(k)}_{i_{n-k}})-\operatorname{P}_{i_{n-k}}(F^{(k')}_{i_{n-k}})}_{s-1,\gamma-1} <\varepsilon\] for all $k,k'>N(\varepsilon).$
Therefore, by the Banach algebra property \eqref{Banachalg} and \eqref{subtraction}, we have that $\left\lbrace \omega(r) \operatorname{P}_{\operatorname{Im}\Omega}(\Xi_{k}) \right\rbrace_{k\in\mathbb{N}}$ is a Cauchy sequence in the Sobolev space $\mathcal{W}^{s-1,\gamma-1}(M,\bigwedge\nolimits^{n}T^{\ast}_{\wedge}M).$ Away from the edge a completely similar argument using the classical Sobolev spaces $H^{s}(2\MM,T^{\ast}M)$ implies that the sequence $\left\lbrace (1-\omega) \operatorname{P}_{\operatorname{Im}\Omega}(\Xi_{k}) \right\rbrace_{k\in\mathbb{N}}$ is Cauchy in the space $H^{s}(2\MM,T^{\ast}M)$ and the corollary follows immediately. 
\end{proof}

Corollary~\ref{coro1} and ~\ref{coro2} tell us that we have a continuous non-linear deformation operator:
\[\operatorname{P}:=\begin{array}{ccc}
\operatorname{P}_{\omega_{_{\CC^{n}}}} & &   \mathcal{W}^{s-1,\gamma-1}(M,\bigwedge\nolimits^{2}T^{\ast}_{\wedge}M) \\ 
\oplus &: \mathcal{W}^{s,\gamma}(M,T^{\ast}_{\wedge}M) \longrightarrow &\oplus  \\ 
\operatorname{P}_{\operatorname{Im}\Omega} & &    \mathcal{W}^{s-1,\gamma-1}(M,\bigwedge\nolimits^{n}T^{\ast}_{\wedge}M)
\end{array} \] for $s>\frac{\operatorname{dim}\Xcal+\operatorname{dim}\edge+3}{2}$ and $\gamma>\frac{\operatorname{dim}\Xcal+1}{2}.$ 

In fact, this operator is smooth as the following corollary shows.

\begin{corollary}
The non-linear deformation operator 
\[\begin{array}{ccc}
\operatorname{P}_{\omega_{_{\CC^{n}}}} & &   \mathcal{W}^{s-1,\gamma-1}(M,\bigwedge\nolimits^{2}T^{\ast}_{\wedge}M) \\ 
\oplus &: \mathcal{W}^{s,\gamma}(M,T^{\ast}_{\wedge}M) \longrightarrow &\oplus  \\ 
\operatorname{P}_{\operatorname{Im}\Omega} & &    \mathcal{W}^{s-1,\gamma-1}(M,\bigwedge\nolimits^{n}T^{\ast}_{\wedge}M). 
\end{array}\]
is smooth.
\end{corollary} 

\begin{proof}
Recall that from proposition~\ref{localprop} the components of the operator $\operatorname{P}_{\omega_{_{\CC^n}}}$ are given by expressions of the form $\operatorname{P}+\operatorname{Q}+(\operatorname{S}+\operatorname{T})\cdot \operatorname{R}$ where $\operatorname{P}, \operatorname{Q}, \operatorname{R}, \operatorname{S}, \operatorname{T}\in\operatorname{Diff}^{1}_{\operatorname{edge}}(M).$ Therefore, given $f, \nu\in \mathcal{W}^{s,\gamma}(M)$ we have 
\begin{align*}
& (\operatorname{P}+\operatorname{Q}+(\operatorname{S}+\operatorname{T})\cdot \operatorname{R})(f+\nu)-(\operatorname{P}+\operatorname{Q}+(\operatorname{S}+\operatorname{T})\cdot \operatorname{R})(f)
\\& =\operatorname{P}(\nu)+\operatorname{Q}(\nu)+\operatorname{R}(f)(\operatorname{S}(\nu)+\operatorname{T}(\nu))+\operatorname{S}(f)\operatorname{R}(\nu)+\operatorname{T}(f)\operatorname{R}(\nu)
\\&+\operatorname{S}(\nu)\operatorname{R}(\nu)+\operatorname{T}(\nu)\operatorname{R}(\nu).
\end{align*}
Observe that by the continuity of $\operatorname{S}$ \[\norm{\frac{\operatorname{S}(\nu)\operatorname{R}(\nu)}{\norm{\nu}}}_{s-1,\gamma-1}\leq \norm{\operatorname{S}}_{s,\gamma}\norm{\operatorname{R}(\nu)}_{s-1,\gamma-1},\] hence  by the continuity of $\operatorname{R}$ we have  $\lim_{\norm{\nu}\to 0}\frac{\operatorname{S}(\nu)\operatorname{R}(\nu)}{\norm{\nu}}=0$ and the same holds for $\operatorname{T}(\nu)\operatorname{R}(\nu)$. This implies that the Fr\'echet derivative of the operator $\operatorname{P}+\operatorname{Q}+(\operatorname{S}+\operatorname{T})\cdot \operatorname{R}$ at $f$ is given by \[\operatorname{D}\left(\operatorname{P}+\operatorname{Q}+(\operatorname{S}+\operatorname{T})\cdot \operatorname{R}\right)[f]=\operatorname{P}+\operatorname{Q}+\operatorname{R}(f)(\operatorname{S}+\operatorname{T})+\operatorname{S}(f)\operatorname{R}+\operatorname{T}(f)\operatorname{R}\]
 and from this expression is clear that \[\operatorname{D}\left(\operatorname{P}+\operatorname{Q}+(\operatorname{S}+\operatorname{T})\cdot \operatorname{R}\right)[\bullet]:\mathcal{W}^{s,\gamma}(M)\longrightarrow\mathcal{L}\left(\mathcal{W}^{s,\gamma}(M),
 \mathcal{W}^{s-1,\gamma-1}(M) \right)\]
is a continuous operator. Moreover, observe that the second derivative is constant. More precisely, the second derivative of $\operatorname{P}+\operatorname{Q}+(\operatorname{S}+\operatorname{T})\cdot \operatorname{R}$ at $f\in\mathcal{W}^{s,\gamma}(M)$ is the linear, continuous map
\[\operatorname{D}^{2}\left(\operatorname{P}+\operatorname{Q}+(\operatorname{S}+\operatorname{T})\cdot \operatorname{R}\right)[f]\in\mathcal{L}\left(\mathcal{W}^{s,\gamma}(M),\mathcal{L}\left(\mathcal{W}^{s,\gamma}(M),
 \mathcal{W}^{s-1,\gamma-1}(M) \right)\right)\]
 such that 
\begin{align*}
&\left(\operatorname{D}\left(\operatorname{P}+\operatorname{Q}+(\operatorname{S}+\operatorname{T})\cdot \operatorname{R}\right)[f+\nu]-\operatorname{D}\left(\operatorname{P}+\operatorname{Q}+(\operatorname{S}+\operatorname{T})\cdot \operatorname{R}\right)[f]\right. 
  \\&- \left. \operatorname{D}^{2}\left(\operatorname{P}+\operatorname{Q}+(\operatorname{S}+\operatorname{T})\cdot \operatorname{R}\right)[f](\nu)\right)\Big/ {\norm{\nu}} 
\end{align*} 
  goes to zero in $\mathcal{L}\left(\mathcal{W}^{s,\gamma}(M),
 \mathcal{W}^{s-1,\gamma-1}(M) \right)$ when $\norm{\nu}\rightarrow 0$. 
 
 A direct computation shows that 
 \begin{align*}
 &\operatorname{D}\left(\operatorname{P}+\operatorname{Q}+(\operatorname{S}+\operatorname{T})\cdot \operatorname{R}\right)[f+\nu]-\operatorname{D}\left(\operatorname{P}+\operatorname{Q}+(\operatorname{S}+\operatorname{T})\cdot \operatorname{R}\right)[f]
 \\& =\operatorname{R}(\nu)(\operatorname{S}+\operatorname{T})+\operatorname{S}(\nu)\operatorname{R}+\operatorname{T}(\nu)\operatorname{R}
 \end{align*}
hence by uniqueness of derivatives we have   \[\operatorname{D}^{2}\left(\operatorname{P}+\operatorname{Q}+(\operatorname{S}+\operatorname{T})\cdot \operatorname{R}\right)[f](\nu)= \operatorname{R}(\nu)(\operatorname{S}+\operatorname{T})+\operatorname{S}(\nu)\operatorname{R}+\operatorname{T}(\nu)\operatorname{R} \]
for any $\nu\in\mathcal{W}^{s,\gamma}(M)$ and all $f\in\mathcal{W}^{s,\gamma}(M)$. Thus the second derivative is constant and given by \[\operatorname{D}^{2}\left(\operatorname{P}+\operatorname{Q}+(\operatorname{S}+\operatorname{T})\cdot \operatorname{R}\right)[f](g,h)=\operatorname{R}(g)(\operatorname{S}+\operatorname{T})(h)+(\operatorname{S}+\operatorname{T})(g)\operatorname{R}(h)\]
for any $f\in \mathcal{W}^{s,\gamma}(M)$. We conclude that $\operatorname{P}_{\omega_{_{\CC^n}}}$ is smooth.

Now, the operator $\operatorname{P}_{\operatorname{Im}\Omega}$ is given as a sum of products  $\operatorname{P}_{1}(F_1)\operatorname{P}_{2}(F_2)\cdot\dots\cdot \operatorname{P}_n(F_n)$ where $\operatorname{P}_{i}\in\operatorname{Diff}^{1}_{\operatorname{edge}}(M)$ and $F_{i}\in\mathcal{W}^{s,\gamma}(M)$. Then we have that the expression \[\operatorname{P}_{1}(F_1+\nu)\operatorname{P}_{2}(F_2+\nu)\cdot\dots\cdot \operatorname{P}_n(F_n+\nu)-\operatorname{P}_{1}(F_1)\operatorname{P}_{2}(F_2)\cdot\dots\cdot \operatorname{P}_n(F_n)\] is a sum of products of the operators $\operatorname{P}_i$ evaluated at $F_i$ or at $\nu$ with at least one $\nu$ in each product. Those products containing 2 or more terms with $\nu$ does not contribute to the Frechet derivative as in the first part of the proof. Hence the Fr\'echet derivative is computed only with products having one term with $\nu$ which produce continuous linear operators of the form $\operatorname{P}_{i_1}(F_{i_1})\operatorname{P}_{i_2}(F_{i_2})\cdot\dots\cdot \operatorname{P}_{i_{n-1}}(F_{i_{n-1}})\operatorname{P}_{i_n}$. Furthermore, it is easily seen that the $n$th derivative of each of the products  $\operatorname{P}_{1}(F_1)\operatorname{P}_{2}(F_2)\cdot\dots\cdot \operatorname{P}_n(F_n)$ is constant and equal to \[ \sum\limits_{(i_1,\dots,i_n)\in S_n}\operatorname{P}_{i_{1}}\operatorname{P}_{i_{2}}\cdot\dots\cdot \operatorname{P}_{i_{n}}\] 
where the sum is taken over  the group of permutations with $n$ elements. 
From this it follows that $\operatorname{P}_{\operatorname{Im}\Omega}$ is smooth.
\end{proof}

Observe that any element $\alpha\in \mathcal{W}^{s-1,\gamma-1}(M,\bigwedge\nolimits^{n}T^{\ast}_{\wedge}M) $ is given as $\alpha=f\operatorname{Vol}_{M}$ where $f\in  \mathcal{W}^{s-1,\gamma-1}(M)$ and $f_{k}\operatorname{Vol}_{M}\longrightarrow f_{}\operatorname{Vol}_{M}$ in the edge-Sobolev space $\mathcal{W}^{s-1,\gamma-1}(M,\bigwedge\nolimits^{n}T^{\ast}_{\wedge}M) $ if and only if $f_k\longrightarrow f$ in the space $\mathcal{W}^{s-1,\gamma-1}(M).$ Hence we can consider the operator $\operatorname{P}_{\operatorname{Im}\Omega}$ as an operator acting between the following spaces:
\[\operatorname{P}_{\operatorname{Im}\Omega}:\mathcal{W}^{s,\gamma}(M,T^{\ast}_{\wedge}M)\longrightarrow \mathcal{W}^{s-1,\gamma-1}(M). \]

\section{The linear operator DP{[}0{]}}\label{linear operator}

In this section we consider the operator $\operatorname{DP[0]}$, the linearisation at zero of the deformation operator $\operatorname{P}=\operatorname{P}_{\omega_{_{\CC^n}}}\oplus \operatorname{P}_{\operatorname{Im}\Omega}$. A careful analysis of this operator is necessary as we want apply the Implicit Function Theorem for Banach spaces to this linear operator in order to describe  solutions (nearby  zero) of the non-linear equation $\operatorname{P}(f)=0$ in term of $\operatorname{Ker}\operatorname{DP[0]}$. McLean's results in \cite{mclean} implies  that the linearisation of the deformation operator at zero acting on $\smooth_{0}(M,T^{\ast}_{\wedge}M)$ is given by  the Hodge-deRham operator i.e. \[\operatorname{DP[0]}\Big|_{\smooth_{0}(M,T^{\ast}_{\wedge}M)}=d+d^{\ast}.\]

In this section we analyse the extension of this operator to edge-Sobolev spaces, its ellipticity and the Fredholm property.

Observe that on a collar neighborhood any $\Xi\in \smooth_{}\left(\bigwedge\nolimits^{k}T^{\ast}_{\wedge}M\right)$ can be written as 
\begin{equation}\label{decomposition}\Xi=dr\wedge r^{k-1} \Theta_{\Xcal} +dr\wedge\Theta_{\edge}+dr\wedge \sum\limits_{i=1}^{k-2} r^{i}\Lambda^{i}_{\Xcal,\edge}+\sum\limits_{j=1}^{k-1}r^{j}\tilde{\Lambda}^{j}_{\Xcal,\edge}+r^{k}\tilde{\Theta}_{\Xcal}+\tilde{\Theta}_{\edge} 
\end{equation}
where: \begin{enumerate}[i)]
\item $\Theta_{\Xcal}$ is a smooth section of the  bundle $\pi^{\ast}_{\Rpos\times\edge}(\bigwedge\nolimits^{k-1} T^{\ast}\Xcal);$

\item  $\Theta_{\edge}$ is a smooth section of the  bundle $\pi^{\ast}_{\Rpos\times\Xcal}(\bigwedge\nolimits^{k-1}T^{\ast}\edge);$

\item  $\Lambda_{\Xcal,\edge}^{i}$ is the wedge product of a smooth section of the pull-back bundle $\pi^{\ast}_{\Rpos\times\edge}(\bigwedge\nolimits^{i}T^{\ast}\Xcal)$ with a smooth section of the pull-back bundle \\$\pi^{\ast}_{\Rpos\times\Xcal}(\bigwedge\nolimits^{k-1-i}T^{\ast}\edge);$

\item  $\tilde{\Lambda}_{\Xcal,\edge}^{j}$ is the wedge product of a smooth section of the pull-back bundle $\pi^{\ast}_{\Rpos\times\edge}(\bigwedge\nolimits^{j}T^{\ast}\Xcal)$ with a smooth section of the bundle $\pi^{\ast}_{\Rpos\times\Xcal}(\bigwedge\nolimits^{k-j}T^{\ast}\edge);$

\item $\tilde{\Theta}_{\Xcal}$ is a smooth section of the  bundle $\pi^{\ast}_{\Rpos\times\edge}(\bigwedge\nolimits^{k} T^{\ast}\Xcal);$

\item $\tilde{\Theta}_{\edge}$ is a smooth section of the  bundle $\pi^{\ast}_{\Rpos\times\Xcal}(\bigwedge\nolimits^{k} T^{\ast}\edge).$

\end{enumerate}
Recall that $\pi^{}_{\Rpos\times\Xcal}$ and $\pi^{}_{\Rpos\times\edge}$ are the corresponding projections \[ \pi^{}_{\Rpos\times\Xcal}:\cone\times\edge\longrightarrow \edge \] \[ \pi^{}_{\Rpos\times\edge}:\cone\times\edge\longrightarrow \Xcal. \]
More generally, the bundle $\bigwedge^{\bullet}T^{\ast}_{\wedge}M$ can be decomposed in the following way: 
\begin{align*}
\bigwedge\nolimits^{\bullet}T^{\ast}_{\wedge}M &=dr\wedge \sum\limits_{k=0}^{m}r^{k} \pi^{\ast}_{\Rpos\times\edge}\left(\bigwedge\nolimits^{k} T^{\ast}\Xcal \right) \oplus dr\wedge \pi^{\ast}_{\Rpos\times\Xcal}\left(\bigwedge\nolimits^{\bullet}T^{\ast}\edge\right)
\\& \oplus dr\wedge \sum\limits_{k=1}^{m} r^{k}\pi^{\ast}_{\Rpos\times\edge}\left(\bigwedge\nolimits^{k}T^{\ast}\Xcal\right)\wedge\pi^{\ast}_{\Rpos\times\Xcal}
\left(\bigwedge\nolimits^{\bullet}T^{\ast}\edge\right)\\&
\oplus \sum\limits_{k=1}^{m}r^{k}\pi^{\ast}_{\Rpos\times\edge}\left(\bigwedge\nolimits^{k}T^{\ast}\Xcal\right)\wedge\pi^{\ast}_{\Rpos\times\Xcal}
\left(\bigwedge\nolimits^{\bullet}T^{\ast}\edge\right)\\&
\oplus \sum\limits_{k=1}^{m}r^{k}\pi^{\ast}_{\Rpos\times\edge}(\bigwedge\nolimits^{k} T^{\ast}\Xcal)\oplus \pi^{\ast}_{\Rpos\times\Xcal}(\bigwedge\nolimits^{\bullet} T^{\ast}\edge).
\end{align*}
With respect to this splitting we can compute an explicit expression for $d+d^{\ast}$ acting on $\smooth_{0}(M,\bigwedge^{\bullet}T^{\ast}_{\wedge}M)$ to prove the following proposition. 

\begin{proposition}\label{hodgeprop}
The operator $\operatorname{DP[0]}\Big|_{\smooth_{0}(M,T^{\ast}_{\wedge}M)}=d+d^{\ast}$ extends to a continuous linear operator acting between edge-Sobolev spaces  
\[ \begin{array}{cccc}
\operatorname{DP[0]}:& \mathcal{W}^{s,\gamma}_{}(M,\bigwedge^{\bullet}T^{\ast}_{\wedge}M) & \longrightarrow & \mathcal{W}^{s-1,\gamma-1}_{}(M,\bigwedge^{\bullet}T^{\ast}_{\wedge}M)
\end{array}. \]
\end{proposition}

\begin{proof}
By using the splitting (\ref{decomposition}) and arranging in a vector the components of  the differential form $\Xi\in\smooth_{0}(M,\bigwedge^{k}T^{\ast}_{\wedge}M)$ in the following way 
\[\begin{pmatrix}
\Theta_{\Xcal},\Theta_{\edge}, \Lambda^{1}_{\Xcal,\edge},\dots ,\Lambda^{k-2}_{\Xcal,\edge} \\ 
\tilde{\Lambda}^{1}_{\Xcal,\edge},\dots , \tilde{\Lambda}^{k-1}_{\Xcal,\edge},\tilde{\Theta}_{\Xcal},\tilde{\Theta}_{\edge}
\end{pmatrix}, \]
a direct computation shows that the Hodge-deRham operator acting on $\Xi$ is given by \[d+d^{\ast}=
\left[
\begin{array}{c|c}
\textbf{A} & \textbf{B} \\
\hline
\textbf{C} & \textbf{D}
\end{array}
\right],
\]
where the operator matrices are given by

\[\textbf{A}=\begin{bsmallmatrix}
\frac{1}{r}(d_\Xcal+d^{\ast}_{\Xcal})  & 0  & 0 & 0 & \cdots &\cdots & -d^{\ast}_{\edge} \\ 
0 & -\frac{1}{r}(d_\edge+d^{\ast}_{\edge}) & -\frac{1}{r}d^{\ast}_{\Xcal} & 0 & \cdots & \cdots & 0 \\ 
0 & -\frac{1}{r}d^{\ast}_{\Xcal} & -(d_{\edge}+\frac{1}{r}d_{\Xcal}^{\ast}) & 0 & 0 &  \cdots & \vdots \\ 
\vdots & 0& -(d^{\ast}_{\edge}+\frac{1}{r}d^{}_{\Xcal}) & -(d^{}_{\edge}+\frac{1}{r}d^{\ast}_{\Xcal}) & 0 & \cdots & 0 \\ 
\vdots & \vdots & 0 & \ddots & \ddots & \ddots & 0 \\ 
\vdots & \vdots & \vdots & 0 &  -(d^{\ast}_{\edge}+\frac{1}{r}d^{}_{\Xcal}) & -(d_{\edge}+\frac{1}{r}d^{\ast}_{\Xcal}) & 0\\ 
\vdots & \vdots & \vdots & \vdots & 0 & -\frac{1}{r}d_{\Xcal} & -d_{\edge}\\ 
-d_{\edge} & 0 & 0 & \cdots & 0 &  0 & -\frac{1}{r}d_{\Xcal} \\ 
\end{bsmallmatrix}\]

\[\textbf{B}=\begin{bsmallmatrix}
 0 & \cdots  & \cdots & 0 & \frac{k}{r}+\partial_{r} & 0 \\ 
 0 & \cdots & \cdots   & 0  & 0 & \partial_{r} \\ 
 \frac{1}{r}+\partial_{r} & 0 & \cdots   & \vdots & 0 & 0 \\ 
0 & \ddots & \ddots  & 0 & \vdots & 0  \\ 
\vdots & \ddots &   & 0 & \vdots & 0  \\ 
0 & \cdots & 0 & \frac{k-1}{r}+\partial_{r} & 0 & 0 \\ 
\end{bsmallmatrix} 
\quad \quad \textbf{C}=\begin{bsmallmatrix}
0 &0 & \frac{1}{r}-\partial_{r} & 0  & \cdots & 0 \\ 
\vdots & 0 & \ddots  &  \ddots  & \ddots & \vdots \\ 
\vdots &\vdots & 0  &  \ddots& \ddots & 0 \\ 
\vdots & \vdots & \cdots & \cdots  & 0 & \frac{k-2}{r}-\partial_{r} \\ 
0 & 0 & \cdots &  \cdots  & \cdots & 0 \\ 
0 & 0 & \cdots &   \cdots & \cdots & 0 \\ 
\frac{k}{r}-\partial_{r} & 0 & \cdots  & \cdots & \cdots & 0 \\ 
0 & -\partial_{r} & 0  & \cdots & \cdots & 0
\end{bsmallmatrix}\]

\[\textbf{D}=\begin{bsmallmatrix}
d_{\edge}+d^{\ast}_{\edge} & \frac{1}{r}d^{\ast}_{\Xcal} & 0 & \cdots & \cdots & 0 & \frac{1}{r}d_{\Xcal} \\ 
 \frac{1}{r}d_{\Xcal} & d_{\edge}+d^{\ast}_{\edge} &  \frac{1}{r}d^{\ast}_{\Xcal} & 0 &\cdots  & \vdots & 0 \\ 
 0 & \ddots & \ddots & \ddots & \ddots & \vdots & 0 \\ 
 \vdots &  &  &  &  & \vdots & \vdots\\ 
  \vdots & \cdots &  \frac{1}{r}d_{\Xcal} & d_{\edge}+d^{\ast}_{\edge}  & \frac{1}{r}d^{\ast}_{\Xcal}& 0 & \vdots\\ 
 \vdots & \cdots & 0 &\frac{1}{r}d_{\Xcal}  & d_{\edge} & 0 & \vdots \\
 \vdots &\cdots  & \cdots &  0& \frac{1}{r}d_{\Xcal} & d_{\edge} & \vdots \\ 
 0 & \cdots & \cdots & 0 & d^{\ast}_{\edge}&\frac{1}{r}(d_{\Xcal}+d^{\ast}_{\Xcal}) & 0 \\ 
\frac{1}{r}d^{\ast}_{\Xcal} & 0 & \cdots & \cdots & 0 & 0 & d_{\edge}+d^{\ast}_{\edge}
\end{bsmallmatrix}. \]

Observe that each of the elements in these matrices is an element of the algebra $\operatorname{Diff}^{1}_{\operatorname{edge}}(M,\bigwedge^{\bullet}T^{\ast}_{\wedge}M)$. From this, it follows that  $d+d^{\ast}$ belongs to $\operatorname{Diff}^{1}_{\operatorname{edge}}\left(\bigwedge^{\bullet}T^{\ast}_{\wedge}M\right)$ which implies that $\operatorname{DP[0]}$ extends to a continuous linear operator between the corresponding edge-Sobolev spaces.
\end{proof}

In order to verify some properties of the symbolic structure of $d+d^{\ast}$ it shall be useful to have similar explicit expressions for the Hodge-Laplacian associated with the edge metric $g_{_{M}}$ acting on $\bigwedge^{\bullet}T^{\ast}_{\wedge}M$.
\begin{proposition}\label{laplaceprop}
The Hodge-Laplace operator $\Delta_{g_{_{M}}}\Big|_{\smooth_{0}(M,T^{\ast}_{\wedge}M)}$ extends to a continuous linear operator acting between edge-Sobolev spaces  
\[ \begin{array}{cccc}
\Delta_{g_{_{M}}}:& \mathcal{W}^{s,\gamma}_{}(M,\bigwedge^{\bullet}T^{\ast}_{\wedge}M) & \longrightarrow & \mathcal{W}^{s-2,\gamma-2}_{}(M,\bigwedge^{\bullet}T^{\ast}_{\wedge}M).
\end{array}\]
\end{proposition}

\begin{proof}
By arranging the components of $\Xi\in\smooth_{0}(M,\bigwedge^{k}T^{\ast}_{\wedge}M)$ in the same way as in the previous proposition, a direct computation shows that the Hodge-Laplacian acting on $k$-forms \[\Delta_{g_{_{M}}}:\smooth_{0}(M,\bigwedge\nolimits^{k}T^{\ast}_{\wedge}M)\longrightarrow\smooth_{0}(M,\bigwedge\nolimits^{k}T^{\ast}_{\wedge}M)\]
is given by
\[\Delta_{g_{_{M}}}=
\left[
\begin{array}{c|c}
\textbf{A}' & \textbf{B}' \\
\hline
\textbf{C}' &\textbf{ D}'
\end{array}
\right],
\]

where the operator matrices are given by
\[  \textbf{A}'=\]
\[ 
    \scalemath{0.62}{
    \begin{bmatrix}
     \mathsmaller{\frac{1}{r^2}\Delta_{\Xcal}+d^{\ast}_{\edge}d^{}_{\edge}-\partial_{r}^{2}+\frac{(k-1)(k-2)}{r^2}} & 0  & 0 & 0 & \cdots &\cdots & \frac{1}{r}(d^{\ast}_{\edge}d^{}_{\Xcal}+d^{}_{\Xcal}d^{\ast}_{\edge}) \\ 
0 & \mathsmaller{\Delta_{\edge}+\frac{1}{r^2}d^{\ast}_{\Xcal}d^{}_{\Xcal}-\partial_{r}^{2}} & \frac{1}{r}(d^{\ast}_{\Xcal} d^{}_{\edge}+d_{\edge}d^{\ast}_{\Xcal}) & 0 & \cdots & \cdots & 0 \\ 
0 & \frac{1}{r}(d^{\ast}_{\edge}d^{}_{\Xcal}+d^{}_{\Xcal}d^{\ast}_{\edge}) & \frac{1}{r^2}\Delta_{\Xcal}+\Delta_{\edge}-\partial^{2}_{r} & 0 &\cdots  &  \cdots & \vdots \\ 
\vdots & 0&  &  &  & & 0 \\ 
\vdots & \vdots & \ddots & \ddots &  & \ddots & 0 \\ 
\vdots & \vdots & &  &  &  & \frac{1}{r}(d^{\ast}_{\Xcal} d^{}_{\edge}+d_{\edge}d^{\ast}_{\Xcal})\\ 
\frac{1}{r}(d^{\ast}_{\Xcal} d^{}_{\edge}+d_{\edge}d^{\ast}_{\Xcal}) & 0 & \cdots & \cdots & 0 &  \frac{1}{r}(d^{\ast}_{\edge}d^{}_{\Xcal}+d^{}_{\Xcal}d^{\ast}_{\edge}) & \mathsmaller{\frac{1}{r^2}\Delta_{\Xcal}+\Delta_{\edge}-\partial^{2}_{r}+\frac{(k-2)^2-(k-2)}{r^2}} \\
    \end{bmatrix}
    }  
\]

\[\textbf{B}'=\begin{bsmallmatrix}
 0 & \cdots  & 0 & \frac{-1}{r}d^{\ast}_{\edge} & \frac{-2}{r^2}d^{\ast}_{\Xcal} & 0 \\ 
 \frac{-2}{r^2}d^{\ast}_{\Xcal} & 0 & \cdots   & 0  & 0 & \vdots \\ 
 0 &\ddots  & \ddots   & \vdots & \vdots & \vdots \\ 
\vdots & \ddots &   & 0 & \vdots & \vdots  \\ 
0 & \cdots & 0 & \frac{-2}{r}d^{\ast}_{\Xcal} & 0 &  0\\ 
\end{bsmallmatrix} \quad 
\textbf{C}'=\begin{bsmallmatrix}
0 & \frac{-2}{r^2}d_{\Xcal} & 0 & \cdots  & \cdots & 0 \\ 
\vdots & 0 &\ddots  &    & \ddots & \vdots \\ 
\vdots &\vdots & \ddots  &  & \ddots & 0 \\ 
\vdots & \vdots & \cdots &   & 0 & \frac{-2}{r^2}d_{\Xcal} \\ 
\frac{-2}{r^2}d_{\Xcal} & 0 & \cdots  & \cdots & \cdots & 0 \\ 
0 &0 & \cdots  & \cdots & \cdots & 0
\end{bsmallmatrix}\]

\[ 
    \scalemath{0.6}{
    \textbf{D}'=\begin{bmatrix}
\mathsmaller{\frac{1}{r^2}\Delta_{\Xcal}+\Delta_{\edge}-\partial^{2}_{r}} & \frac{1}{r}(d^{\ast}_{\Xcal}d_{\edge}+d_{\edge}d^{\ast}_{\Xcal}) & 0 & \cdots & \cdots & 0 & \frac{1}{r}(d^{\ast}_{\edge}d_{\Xcal}+d_{\Xcal}d^{\ast}_{\edge}) \\ 
\frac{1}{r}(d^{\ast}_{\edge}d_{\Xcal}+d_{\Xcal}d^{\ast}_{\edge}) &  &  &  &  & \vdots& 0 \\ 
 0 &\ddots  & \ddots &   & \ddots & \vdots & \vdots \\ 
 \vdots &  &  &  &  & \vdots & \vdots\\
 \vdots & \ddots &  & \ddots & \ddots & \vdots & \vdots\\
 \vdots &  &  &  &  & 0 & \vdots\\
 \vdots &  \cdots &  0 &  &   & \frac{1}{r}(d^{\ast}_{\Xcal}d_{\edge}+d_{\edge}d^{\ast}_{\Xcal}) & \vdots\\ 
 0 & \cdots & \cdots & 0 & \frac{1}{r}(d^{\ast}_{\edge}d_{\Xcal}+d_{\Xcal}d^{\ast}_{\edge}) & \mathsmaller{\frac{1}{r^2}\Delta_{\Xcal}+d^{\ast}_{\edge}d_{\edge}-\partial^{2}_{r} -\frac{k(k-1)}{r^2}} & 0 \\ 
\frac{1}{r}(d^{\ast}_{\Xcal}d_{\edge}+d_{\edge}d^{\ast}_{\Xcal}) & \cdots & \cdots & \cdots & 0 & 0 & \mathsmaller{\Delta_{\edge}+\frac{1}{r^2}d^{\ast}_{\Xcal}d_{\Xcal}-\partial^{2}_{r}}
 \end{bmatrix}
    }
\]

Again each of the operators in the matrices is an element of the algebra $\operatorname{Diff}^{2}_{\operatorname{edge}}(M,\bigwedge^{\bullet}T^{\ast}_{\wedge}M)$, hence the result follows. 
\end{proof}

\subsection{Symbolic structure of edge-degenerate differential operators}\label{symbolsection}
In this section we collect the basic facts about the symbolic structure of edge-degenerate differential operators the we use in the following sections. The reader is referred to \cite{schulze3} chapters 2 and 3 for a full presentation. 

Let's consider a cone-degenerate differential operator (see (\ref{coneop})) \[\operatorname{P}=r^{-l}\sum\limits_{i\leq l}a_{i}(r)(-r\partial_{r})^{i}.\] As an element in $\operatorname{Diff}(M)$ it has its classical principal homogeneous symbol given by \[\sigma^{l}(\operatorname{P})(r,x,\rho,\xi)=r^{-l}\sum\limits_{i+\abs{\alpha}= l}a_{i,\alpha}(r,x)(- \sqrt{-1} r\rho)^{i}\xi^{\alpha}\] with smooth coefficients $a_{i,\alpha}(r,x)$ up to $r=0$. In order to reflect the singular behavior on the symbolic structure two additional symbols are introduced. First we have the homogeneous boundary symbol \[\sigma^{l}_{b}(\operatorname{P})(r,x,\tilde{\rho},\xi):=\sum\limits_{i+\abs{\alpha}= l}a_{i,\alpha}(r,x)(- \sqrt{-1} \tilde{\rho})^{i}\xi^{\alpha}\] defined on the cotangent bundle of the stretched manifold $T^{\ast}\MM$ and smooth up to $r=0$. 

The second symbol is the Mellin symbol $\sigma^{l}_{M}(\operatorname{P})$. Recall that  any cone-degenerate operator $\operatorname{P}=r^{-l}\sum\limits_{i\leq l}a_{i}(r)(-r\partial_{r})^{i}$ is given in terms of the Mellin transformation as follows \[\operatorname{P}=r^{-l}\mathcal{M}^{-1}h(r,z)\mathcal{M},\]
where $h(r,z)=\sum\limits_{i\leq l}a_{i}(r)z^{i}.$ At the singular set i.e. when $r=0$ we have a holomorphic family of operators \[h(0,z):\CC\longrightarrow \mathcal{L}\left(H^{s}(\Xcal),H^{s-l}(\Xcal)\right).\] This holomorphic operator-valued function is the Mellin symbol of $\operatorname{P}$ i.e.  $\sigma^{l}_{M}(\operatorname{P})(z):=h(0,z).$ 

Now consider an edge-degenerate differential operator (see  (\ref{edgeop})  above) \[\operatorname{P}=r^{-l}\sum\limits_{i+\abs{\alpha}\leq l}a_{i,\alpha}(r,u)(-r\partial_{r})^{i}(rD_u)^{\alpha}.\] In the same way as the cone-degenerate case it has a classical homogeneous principal symbol and a homogeneous boundary symbol  \[\sigma^{l}(\operatorname{P})(r,x,u,\rho,\xi,\eta)=r^{-l}\sum\limits_{i+\abs{\alpha}+\abs{\beta}= l}a_{i,\alpha,\beta}(r,x,u)(-\sqrt{-1} r\rho)^{i}(r\eta)^{\alpha}\] \[\sigma^{l}_{b}(\operatorname{P})(r,x,u,\tilde{\rho},\xi,\tilde{\eta})=\sum\limits_{i+\abs{\alpha}+\abs{\beta}= l}a_{i,\alpha,\beta}(r,x,u)(-\sqrt{-1} \tilde{\rho})^{i}\tilde{\eta}^{\alpha}\] satisfying the relation \[\sigma^{l}_{b}(\operatorname{P})(r,x,u,\rho,\xi,\eta)=r^{l}\sigma^{l}(\operatorname{P})(r,x,r^{-1}\rho,\xi,r^{-1}\eta).\]
In the edge setting we have the edge symbol $\sigma^{l}_{\wedge}(\operatorname{P})$ that  is defined as an operator-valued function acting on spaces  one level below in the singular hierarchy, in this case those operators act on the cone-Sobolev spaces $\mathcal{K}^{s,\gamma}(\cone)$ as cone-degenerate operators. More precisely we have \[\sigma^{l}_{\wedge}(\operatorname{P}):T^{\ast}\edge\setminus \{ 0\}\longrightarrow\mathcal{L}\left(    \mathcal{K}^{s,\gamma}(\cone) ,\mathcal{K}^{s-l,\gamma-l}(\cone)
\right)  \]
given by
\[\sigma^{l}_{\wedge}(\operatorname{P})(u,\eta)=r^{-l}\sum\limits_{i+\abs{\alpha}\leq l}a_{i,\alpha}(0,u)(-r\partial_{r})^{i}(r\eta)^{\alpha}.\] 
This is a family of cone-degenerate operators parametrized by the cotangent bundle of the edge $\edge$. 

\subsection{Ellipticity on manifolds with conical and edge singularities}\label{ellipticsection}
In this subsection we present the definition and principal implications of ellipticity for cone and edge-degenerate differential operators. We use these results to study the regularity of our moduli spaces in section~\ref{regularity}. Further details can be found in \cite{schulze3} chapter 2 and 3.

\begin{definition}
A cone-degenerate operator of order $l$, $\operatorname{P}\in\operatorname{Diff}^{l}_{\operatorname{cone}}(M)$, is called elliptic with respect to the weight $\gamma$ if 
\begin{enumerate}[i)]
\item $\sigma^{l}_{b}(\operatorname{P})(\tilde{\rho},\xi)\neq 0$ on $T^{\ast}_{}\MM\setminus \lbrace 0 \rbrace$
\item the Mellin symbol $\sigma^{l}_{M}(\operatorname{P})(z)=h(0,z)$ defines a family of Banach space  isomorphisms in $\mathcal{L}\left( H^{s}(\Xcal),H^{s-l}(\Xcal)\right)$ for some $s\in\RR$ and all $z\in \Gamma_{\frac{\operatorname{dim}\Xcal+1}{2}-\gamma}$. 
\end{enumerate}
\end{definition}

\begin{definition}\label{admissible}
If $\gamma\in\RR$ such that the Mellin symbol is invertible along $\Gamma_{\frac{\operatorname{dim}\Xcal+1}{2}-\gamma}$ we say that $\gamma$ is an admissible weight.
\end{definition}
Now we have some important implications of being elliptic.
\begin{theorem}\label{ellipticthmcone} Let $\operatorname{P}$ be a cone-degenerate operator of order $l$.
\begin{enumerate}[i)]

\item (Fredholm property) $\operatorname{P}$ is elliptic with respect to the weight $\gamma$ if and only if \[\operatorname{P}: \mathcal{H}^{s,\gamma}(M)\longrightarrow\mathcal{H}^{s-l,\gamma-l}(M)\] is Fredholm for every $s\in \RR$. 

\item (Parametrix) If $\operatorname{P}$ is elliptic then there exists a parametrix of order $-l$ with asymptotics. 

\item (Elliptic regularity) If $\operatorname{P}$ is elliptic with respect to $\gamma$ and $\operatorname{P}f=g$ with $g\in \mathcal{H}^{s-l,\gamma-l}_{O}(M)$ and $f\in \mathcal{H}^{-\infty,\gamma}(M)$ for some $s\in\RR$ and some asymptotic type $O$ associated with $\gamma-l$ then $f\in \mathcal{H}^{s,\gamma}_{Q}(M)$ with an asymptotic type $Q$ associated with $\gamma$. 
\end{enumerate}
\end{theorem}

\begin{remark}\label{remarkasymptoticscone}
Let's consider a cone-degenerate operator $\operatorname{P}$ and $f\in\mathcal{H}^{s,\gamma}(M)$ such that $\operatorname{P}f=0$. The parametrix with asymptotics implies that $(\mathfrak{P}\operatorname{P}-I)(f)$ $=f\in\mathcal{H}^{\infty,\gamma}_{O}(M)$. Therefore solutions of cone-degenerate differential equations are smooth on the regular part of $M$ and have conormal asymptotics expansions near the vertex of the cone.  
\end{remark}
In order to introduce the notion of ellipticity in the edge singular setting we need to make some assumptions. Assume that there exist vector bundles $J^{+}$ and $J^{-}$ over $\edge$ and operator families parametrized by $T^{\ast}\edge\setminus\{ 0\}$ acting as follows \[\sigma^{l}_{\wedge}(\operatorname{T})(u,\eta):\mathcal{K}^{s,\gamma}(\cone)\longrightarrow J^{+}_{u},\]  \[\sigma^{l}_{\wedge}(\operatorname{C})(u,\eta): J^{-}_{u} \longrightarrow \mathcal{K}^{s-l,\gamma-l}(\cone),\]  \[\sigma^{l}_{\wedge}(\operatorname{B})(u,\eta): J^{-}_{u} \longrightarrow J^{+}_{u};\] such that \begin{equation}\label{symbolblock}
\begin{array}{cccc}
\begin{bmatrix}
 \sigma^{1}_{\wedge}(\operatorname{P}) (u,\eta) & \sigma^{l}_{\wedge}(\operatorname{C})(u,\eta)  \\ 
\sigma^{l}_{\wedge}(\operatorname{T})(u,\eta) & \sigma^{l}_{\wedge}(\operatorname{B})(u,\eta)
\end{bmatrix}: & \begin{array}{c}
 \mathcal{K}^{s,\gamma}(\cone ) \\ 
 \oplus \\ 
J^{-}_{u}
 \end{array} & \longrightarrow & \begin{array}{c}
 \mathcal{K}^{s-l,\gamma-l}(\cone ) \\ 
 \oplus \\ 
J^{+}_{u}
 \end{array}
\end{array} 
\end{equation}
 is a family of continuous operators for every $(u,\eta)\in T^{\ast}\edge\setminus\{ 0\}.$
 
 The existence of the vector bundles $J^{\pm}$ and operators acting between the fibers and cone-Sobolev spaces will be discussed in section~\ref{symbol-linearop}. Here we only mention that there is a topological obstruction (see theorem~\ref{toprestric}) that must be satisfied in order to guarantee the existence of $J^{\pm}$ and the operators. 
 
\begin{definition}\label{ellipticedgeoperator}
An edge-degenerate differential operator $\operatorname{P}\in\operatorname{Diff}^{l}_{\operatorname{edge}}(M)$ of order $l$ for which \eqref{symbolblock} exists,  is called elliptic with respect to the weight $\gamma$ if 
\begin{enumerate}[i)]
\item $\sigma^{l}_{b}(\operatorname{P})\neq 0$ on $T^{\ast}_{}\MM\setminus \lbrace 0 \rbrace$
\item the operator matrix (\ref{symbolblock}) defines an invertible operator for some $s\in\RR$ and each $(u,\eta)\in T^{\ast}\edge\setminus\{0\}$.
\end{enumerate}
\end{definition}

\begin{theorem}\label{ellipticityforedge} Let $\operatorname{P}$ be an edge-degenerate operator of order $l$.
\begin{enumerate}[i)]
\item (Fredholm property) $\operatorname{P}$ is elliptic with respect to the weight $\gamma$ if and only if  the operator \begin{equation}\label{block}\mathbf{A}_{_{\operatorname{P}}}:=\begin{bmatrix}
 \operatorname{P} & \operatorname{C}  \\ 
\mathcal{T} & \operatorname{B}
\end{bmatrix}=\mathcal{F}^{-1}_{\eta\rightarrow u} \begin{bmatrix}
 \sigma^{1}_{\wedge}(\operatorname{P}) (u,\eta) & \sigma^{l}_{\wedge}(\operatorname{C})(u,\eta)  \\ 
\sigma^{l}_{\wedge}(\operatorname{T})(u,\eta) & \sigma^{l}_{\wedge}(\operatorname{B})(u,\eta)
\end{bmatrix}\mathcal{F}_{u' \rightarrow \eta}
\end{equation}
acting on the spaces \[\mathbf{A}_{\operatorname{P}}: \begin{array}{c}
\mathcal{W}^{s,\gamma}(M) \\ 
\oplus \\ 
H^{s}(\edge,J^{-})
\end{array}\longrightarrow \begin{array}{c}
\mathcal{W}^{s-l,\gamma-l}(M) \\ 
\oplus \\ 
H^{s-l}(\edge,J^{+})
\end{array}\] is Fredholm for every $s\in\RR$.

\item (Parametrix) If $\operatorname{P}$ is elliptic then there exists a parametrix of order $-l$ with asymptotics for $\mathbf{A}_{\operatorname{P}}$. 

\item (Elliptic regularity) If $\operatorname{P}$ is elliptic with respect to $\gamma$ and $\mathbf{A}_{\operatorname{P}}f=g$ with \[g\in \begin{array}{c}
\mathcal{W}^{s-l,\gamma-l}_{O}(M) \\ 
\oplus \\ 
H^{s-l}(\edge,J^{+})
\end{array}\quad \text{and}\quad f\in\begin{array}{c}
\mathcal{W}^{-\infty,\gamma}_{}(M) \\ 
\oplus \\ 
H^{-\infty}(\edge,J^{-})
\end{array}\] for some $s\in\RR$ and some asymptotic type $O$ associated with $\gamma-l$ then \[f\in\begin{array}{c}
\mathcal{W}^{s,\gamma}_{Q}(M) \\ 
\oplus \\ 
H^{s}(\edge,J^{-})
\end{array}\] with an asymptotic type $Q$ associated with $\gamma$. 
\end{enumerate}
\end{theorem}

\begin{remark}\label{remarkasymptoticsedge}
Analogously to remark~\ref{remarkasymptoticscone} we have that solutions to the edge-degenerate equation $\mathbf{A}_{\operatorname{P}}f=0$ belong to $\mathcal{W}^{\infty,\gamma}_{O}(M)$, hence they are smooth and have conormal asymptotics near the edge $\edge$.
\end{remark}

\subsection{The symbolic structure of DP{[}0{]}}\label{symbol-linearop}

Recall from section~\ref{symbolsection} that the symbolic structure of the edge-degenerate operator $\operatorname{DP[0]}$ is given by the pair 
\[ \left(
\begin{array}{cc}
 \sigma^{1}_{b}\left(\operatorname{DP[0]}\right)  (r,\sigma,u,\tilde{\rho},\xi,\tilde{\eta}), &\sigma^{1}_{\wedge} \left(\operatorname{DP[0]}\right) (u,\eta)
\end{array} \right) \]
where $ \sigma^{1}_{b}\left(\operatorname{DP[0]}\right)$ is a bundle map on $\pi^{\ast}_{_{T^{\ast}\MM}}(\bigwedge^{\bullet}T^{\ast}_{\wedge}\MM).$ 

The edge symbol  $\sigma^{1}_{\wedge}\left(\operatorname{DP[0]}\right)$ is a family of continuous linear operators acting on  cone-Sobolev spaces and parametrized by the cosphere bundle over $\edge$:
\[\sigma^{1}_{\wedge}\left(\operatorname{DP[0]}\right):S^{\ast}\edge\longrightarrow \mathcal{L}\left( \mathcal{K}^{s,\gamma}(\Xcal^{\wedge}, \bigwedge\nolimits^{\bullet}T^{\ast}_{\wedge}M), \mathcal{K}^{s-1,\gamma-1}(\Xcal^{\wedge}, \bigwedge\nolimits^{\bullet}T^{\ast}_{\wedge}M)\right).\]
The ellipticity of the operator $\operatorname{DP[0]}$ shall require the invertibility of its symbolic structure. From proposition~\ref{hodgeprop} and ~\ref{laplaceprop} we obtain the first part of the desired result.

\begin{proposition}\label{symboliso}
\[\sigma^{1}_{b}(\operatorname{DP[0]})(r,\sigma, u, \tilde{\rho},\xi, \tilde{\eta}): \bigwedge\nolimits^{\bullet}T^{\ast}_{\wedge,(r,\sigma, u)}\MM\longrightarrow \bigwedge\nolimits^{\bullet}T^{\ast}_{\wedge,(r,\sigma, u)}\MM\]
is a bundle isomorphism for every non-zero $(r,\sigma, u, \tilde{\rho},\xi, \tilde{\eta})\in T^{\ast}\MM$ up to $r=0.$
\end{proposition}
 \begin{proof}
 To prove this result we shall use the symbolic structure of the Hodge-Laplace operator 
\[ \left(
\begin{array}{cc}
 \sigma^{2}_{b}(\Delta_{g_{_{M}}}) (r,\sigma,u,\tilde{\rho},\xi,\tilde{\eta}), & \sigma^{2}_{\wedge}(\Delta_{g_{_{M}}}) (u,\eta)
\end{array} \right).\]
From \cite{schulze3} theorem 3.4.56 we have the natural expected symbolic relation
\begin{equation}\label{laplacesymbol}
 \sigma^{2}_{b}(\Delta_{g_{_{M}}}) = \sigma^{1}_{b}(\operatorname{DP[0]})\circ  \sigma^{1}_{b}(\operatorname{DP[0]})
\end{equation}
\begin{equation} \label{laplaceedgesym}
 \sigma^{2}_{\wedge}(\Delta_{g_{_{M}}}) = \sigma^{1}_{\wedge}(\operatorname{DP[0]})\circ  \sigma^{1}_{\wedge}(\operatorname{DP[0]})
\end{equation}
 as $\operatorname{DP[0]}=d+d^{\ast}$ and $\Delta_{g_{M}}=(d+d^{\ast})\circ(d+d^{\ast})$.
Now, from the matrices representing $\Delta_{g_{_{M}}}$ in proposition~\ref{laplaceprop} we have that the elements in $\textbf{B}'$ and $\textbf{C}'$  are operators of order 1 hence they do not intervene in the computation of $\sigma^{2}_{b}(\Delta_{g_{M}})$. Hence let's focus on the operators in $\textbf{A}'$ and $\textbf{D}'$.

 Observe that for any  $\alpha\in\bigwedge^{\bullet}T^{\ast}_{\wedge}M$ 
\begin{align*}
 \sigma^{2}_{ b}(d^{\ast}_{\edge}d_{\Xcal}+d_{\Xcal}d^{\ast}_{\edge})(r,\sigma, u, \tilde{\rho},\xi, \tilde{\eta})(\alpha) & = \tilde{\eta}^{\ast}\lrcorner (\xi\wedge\alpha)+\xi\wedge( \tilde{\eta}^{\ast}\lrcorner\alpha)
\\& =(\tilde{\eta}^{\ast}\lrcorner\xi)\wedge\alpha - \xi\wedge( \tilde{\eta}^{\ast}\lrcorner\alpha)+ \xi\wedge( \tilde{\eta}^{\ast}\lrcorner\alpha)
\\& =0
\end{align*}
as $\tilde{\eta}^{\ast}\in T\edge$ and $\xi\in T^{\ast}\Xcal.$
Moreover
\begin{align*}
 \sigma^{2}_{ b}(d^{\ast}_{\edge}d_{\edge})( \tilde{\eta})(\Theta_{\Xcal}) & = \tilde{\eta}^{\ast}\lrcorner (\tilde{\eta}\wedge\Theta_{\Xcal})
\\& =\Theta_{\Xcal}+\tilde{\eta}^{}\wedge (\tilde{\eta}^{\ast}\lrcorner\Theta_{\Xcal})
\\& =\Theta_{\Xcal}
\end{align*}
and in the same way \[ \sigma^{2}_{ b}(d^{\ast}_{\Xcal}d_{\Xcal})( \xi)(\Theta_{\edge})=\Theta_{\edge}.\]
Therefore $ \sigma^{2}_{b}(\Delta_{g_{_{M}}})(r,\sigma,u,\tilde{\rho},\xi,\tilde{\eta})$ is a diagonal matrix with entries given by \[\sigma^{2}_{b}(\Delta_{\Xcal}+d^{\ast}_{\edge}d_{\edge}-\partial^{2}_{r})(r,\sigma, u, \tilde{\rho},\xi, \tilde{\eta})=\abs{\xi}_{g_{_{\Xcal}}}^{2}+1+\abs{\tilde{\rho}}^{2}\] \[\sigma^{2}_{b}(\Delta_{\edge}+d^{\ast}_{\Xcal}d_{\Xcal}-\partial^{2}_{r}) (r,\sigma, u, \tilde{\rho},\xi, \tilde{\eta})=\abs{\tilde{\eta}}_{g_{_{\edge}}}^{2}+1+\abs{\tilde{\rho}}^{2}\] and \[\sigma^{2}_{b}(\Delta_{\Xcal}+\Delta_{\edge}-\partial^{2}_{r})(r,\sigma, u, \tilde{\rho},\xi, \tilde{\eta})=\abs{\xi}_{g_{_{\Xcal}}}^{2}+\abs{\tilde{\eta}}_{g_{_{\edge}}}^{2}+\abs{\tilde{\rho}}^{2}.\]
Hence  \[\sigma^{2}_{b}(\Delta_{g_{_{M}}})(r,\sigma, u, \tilde{\rho},\xi, \tilde{\eta}):\bigwedge\nolimits^{\bullet}T^{\ast}_{\wedge,(r,\sigma,u)}\MM\longrightarrow \bigwedge\nolimits^{\bullet}T^{\ast}_{\wedge,(r,\sigma,u)}\MM\]
is an isomorphism for every non-zero $(r,\sigma, u, \tilde{\rho},\xi, \tilde{\eta})\in T^{\ast}\MM$ up to $r=0$. By \eqref{laplacesymbol} we have that \[\sigma^{1}_{b}(\operatorname{DP[0]})(r,\sigma, u, \tilde{\rho},\xi, \tilde{\eta}): \bigwedge\nolimits^{\bullet}T^{\ast}_{\wedge,(r,\sigma,u)}\MM\longrightarrow \bigwedge\nolimits^{\bullet}T^{\ast}_{\wedge,(r,\sigma,u)}\MM\] has the same property.
 \end{proof}

In order to obtain information about the invertibility of the edge symbol $\sigma^{1}_{\wedge}(\operatorname{DP[0]})$ we will use proposition~\ref{symboliso} together with theorem 2.4.18 and theorem 3.5.1 in \cite{schulze3}. These theorems state the existence of admissible weights $\gamma\in\RR$ such that $\sigma^{1}_{\wedge}(\operatorname{DP[0]})$ is a Fredholm operator on the corresponding cone-Sobolev spaces of any order. We adapt those theorems to our setting in the  following result. Its proof follows immediately from theorem 2.4.18 and theorem 3.5.1 in \cite{schulze3}.

\begin{theorem}\label{thmweight} 
The condition that 
  \[\sigma^{1}_{b}(\operatorname{DP[0]})(r,\sigma, u, \tilde{\rho},\xi, \tilde{\eta}): \bigwedge\nolimits^{\bullet}T^{\ast}_{\wedge,(r,\sigma,u)}\MM\longrightarrow \bigwedge\nolimits^{\bullet}T^{\ast}_{\wedge,(r,\sigma,u)}\MM\]
is an isomorphism for every non-zero $(r,\sigma, u, \tilde{\rho},\xi, \tilde{\eta})\in T^{\ast}_{\wedge}\MM $ up to $r=0$, implies that there exists a countable set $\Lambda\subset\CC$, where $\Lambda \cap K$ is finite for every $K\subset\subset\CC$, such that 
\[\sigma^{1}_{M}\left( \sigma^{1}_{\wedge}(\operatorname{DP[0]})(u,\eta)   \right)(z): H^{s}(\Xcal, \bigwedge\nolimits^{\bullet}T^{\ast}_{\wedge}M)\longrightarrow H^{s-1}(\Xcal, \bigwedge\nolimits^{\bullet}T^{\ast}_{\wedge}M)\]
is an isomorphism (invertible, linear, continuous operator) for every $z\in\CC\setminus\Lambda$ and all $s\in\RR$. 
This implies that there is a countable subset $D\subset\RR$ given by $D=\Lambda \cap \RR$, with the property that  $D\cap\lbrace z : a\leq \operatorname{Re}z \leq b \rbrace$ is finite for every $a\leq b$, such that 
\[\sigma^{1}_{\wedge}(\operatorname{DP[0]})(u,\eta): \mathcal{K}^{s,\gamma}(\cone, \bigwedge\nolimits^{\bullet}T^{\ast}_{\wedge}M)\longrightarrow\mathcal{K}^{s-1,\gamma-1}(\cone,\bigwedge\nolimits^{\bullet}T^{\ast}_{\wedge}M)\]
is a family of Fredholm operators for each $\gamma\in\RR\setminus D$ and $(u,\eta)\in S^{\ast}\edge$ with  $\eta\neq 0.$
\end{theorem}

Theorem~\ref{thmweight} tell us that for an admissible weight $\gamma$, the wedge symbol $\sigma^{1}_{\wedge}(\operatorname{DP[0]})(u,\eta)$ defines a Fredholm operator for each $(u,\eta)\in S^{\ast}\edge$. However, if we require the ellipticity of $\operatorname{DP[0]}$ we need to have a family of invertible operators. In some cases this can be achieved by adding boundary and coboundary operators that defines an elliptic edge boundary value problem. This can be done in the following way. 

For each $(u,\eta)\in S^{\ast}\edge$ we have that $\sigma^{1}_{\wedge}(\operatorname{DP[0]})(u,\eta)$ is a Fredholm operator, then \[\operatorname{Ker}\left( \sigma^{1}_{\wedge}(\operatorname{DP[0]})(u,\eta)\right)\subset\mathcal{K}^{s,\gamma}(\cone, \bigwedge\nolimits^{\bullet}T^{\ast}_{\wedge}M) \] is a finite dimensional subspace. Let $N{(u,\eta)}=\operatorname{dim}\operatorname{Ker}\left( \sigma^{1}_{\wedge}(\operatorname{DP[0]})(u,\eta)\right)$ and choose an isomorphism \[\operatorname{k}(u,\eta):\CC^{N(u,\eta)}\longrightarrow \operatorname{Ker}\left( \sigma^{1}_{\wedge}(\operatorname{DP[0]})(u,\eta)\right),\]
then \[\begin{pmatrix}
\sigma^{1}_{\wedge}(\operatorname{DP[0]}) & \operatorname{k}
\end{pmatrix}(y,\eta):\]\[ \begin{array}{ccc}
\mathcal{K}^{s,\gamma}(\cone, \bigwedge\nolimits^{\bullet}T^{\ast}_{\wedge}M) & • & • \\ 
\oplus & \longrightarrow & \mathcal{K}^{s-1,\gamma-1}(\cone, \bigwedge\nolimits^{\bullet}T^{\ast}_{\wedge}M) \\ 
\CC^{N(u,\eta)} & • & •
\end{array} \]
is a surjective operator. 

Now, because the set of surjective operators is an open set and the space $S^{\ast}\edge$ is compact, there exists $N^{+}\in\mathbb{N}$ and $\operatorname{c}
\in\mathcal{L}\left(\CC^{N^{+}},  \mathcal{K}^{s-1,\gamma-1}(\cone, \bigwedge\nolimits^{\bullet}T^{\ast}_{\wedge}M)\right)$ such that   \begin{equation}\label{Fredholmedgesymbol}
\begin{pmatrix}
\sigma^{1}_{\wedge}(\operatorname{DP[0]}) & \operatorname{c}
\end{pmatrix}(y,\eta): \begin{array}{c}
\mathcal{K}^{s,\gamma}(\cone ,\bigwedge\nolimits^{\bullet}T^{\ast}_{\wedge}M) \\ 
\oplus \\ 
\CC^{N^{+}}
\end{array} \longrightarrow\mathcal{K}^{s-1,\gamma-1}(\cone,\bigwedge\nolimits^{\bullet}T^{\ast}_{\wedge}M) 
\end{equation}
is Fredholm and surjective for each $(y,\eta)\in S^{\ast}\edge$ (see \cite{schulze3} theorem 1.2.30 for further details).
Because $\begin{pmatrix}
\sigma^{1}_{\wedge}(\operatorname{DP[0]}) & \operatorname{c}
\end{pmatrix}(y,\eta)$ is Fredholm and surjective we have that the kernel of $\begin{pmatrix}
\sigma^{1}_{\wedge}(\operatorname{DP[0]}) & \operatorname{c}
\end{pmatrix}(y,\eta)$ has constant dimension equal to its index for every $(y,\eta)\in S^{\ast}\edge$:
\[\operatorname{dim}\operatorname{Ker}\begin{pmatrix}
\sigma^{1}_{\wedge}(\operatorname{DP[0]}) & \operatorname{c}
\end{pmatrix}(y,\eta)=\operatorname{Ind}\begin{pmatrix}
\sigma^{1}_{\wedge}(\operatorname{DP[0]}) & \operatorname{c}
\end{pmatrix}(y,\eta):=N^{-}\] for all $(u,\eta)\in S^{\ast}\edge.$
The finite dimensional spaces $\operatorname{Ker}\begin{pmatrix}
\sigma^{1}_{\wedge}(\operatorname{DP[0]}) & \operatorname{c}
\end{pmatrix}(y,\eta)$ define a smooth vector bundle  over $S^{\ast}\edge$ (see section 1.2.4 in \cite{schulze3}).

Now consider the trivial bundle of dimension $N^{+}$ over $S^{\ast}\edge$, here we denote it simply as $\CC^{N^{+}}.$ The formal difference of these vector bundles defines an element in the K-theory of $S^{\ast}\edge$
\[\left[\operatorname{Ker} \begin{pmatrix}
 \sigma^{1}_{\wedge}(\operatorname{DP[0]})  & \operatorname{c}
\end{pmatrix} \right]-\left[ \CC^{N^{+}} \right] \in K(S^{\ast}\edge).\]
This element in the K-group represents a topological obstruction to the existence of an elliptic edge boundary value problem for the operator $\operatorname{DP[0]}$. More precisely we have the following theorem. For its proof and more details about the obstruction of ellipticity in the edge calculus see \cite{naza} section 6.2.

\begin{theorem}\label{toprestric}
A necessary and sufficient condition for the existence of an elliptic edge problem for $\operatorname{DP[0]}$ is given by \begin{equation}\label{restriction}\left[\operatorname{Ker} \begin{pmatrix}
 \sigma^{1}_{\wedge}(\operatorname{DP[0]})  & \operatorname{c}
\end{pmatrix} \right]-\left[ \CC^{N^{+}} \right] \in \pi^{\ast}_{_{S^{\ast}\edge}}K(\edge)
\end{equation}
where  $\pi^{}_{_{S^{\ast}\edge}}:S^{\ast}\edge\longrightarrow\edge$ is the natural projection and $\pi^{\ast}_{_{S^{\ast}\edge}}K(\edge)$ is the subgroup of $K(S^{\ast}\edge)$ generated by vector bundles lifted from $\edge$ by means of  $\pi^{}_{_{S^{\ast}\edge}}$. 
\end{theorem}

Now, assume for the moment that the condition in theorem~\ref{toprestric} is satisfied. Then the bundle defined by  $\operatorname{Ker} \begin{pmatrix}
\sigma^{1}_{\wedge}(\operatorname{DP[0]}) & \operatorname{c}
\end{pmatrix}(y,\eta)$ is stably equivalent to a vector bundle $J^{-}$ lifted from $\edge$. Then, by adding zeros to $\operatorname{c}$ if needed,  we can assume that the vector bundle given by $\operatorname{Ker} \begin{pmatrix}
\sigma^{1}_{\wedge}(\operatorname{DP[0]}) & \operatorname{c}
\end{pmatrix}(y,\eta)$ is isomorphic to $J^{-}.$ By extending this isomorphism by zero on the orthogonal complement of $\operatorname{Ker}\begin{pmatrix}
\sigma^{1}_{\wedge}(\operatorname{DP[0]}) & \operatorname{c}
\end{pmatrix}(y,\eta)$ we obtain a map \begin{equation}\label{symboltrace}
\begin{pmatrix}
\operatorname{t}(u,\eta) & \operatorname{b}(u,\eta)
\end{pmatrix}: \begin{array}{c}
\mathcal{K}^{s,\gamma}(\cone ,\bigwedge\nolimits^{\bullet}T^{\ast}_{\wedge}M) \\ 
\oplus \\ 
\CC^{N^{+}}
\end{array} \longrightarrow J^{-}_{(u,\eta)}
\end{equation} 
such that 
\[ 
 \begin{bmatrix}
 \sigma^{1}_{\wedge}(\operatorname{DP[0]}) (u,\eta) & \operatorname{c}(u,\eta)  \\ 
\operatorname{t}(u,\eta) & \operatorname{b}(u,\eta)
\end{bmatrix}: 
\]

\[
\begin{array}{ccc}
 \begin{array}{c}
 \mathcal{K}^{s,\gamma}(\cone ,\bigwedge\nolimits^{\bullet}T^{\ast}_{\wedge}M) \\ 
 \oplus \\ 
\CC^{N^{+}}
 \end{array} & \longrightarrow & \begin{array}{c}
 \mathcal{K}^{s-1,\gamma-1}(\cone ,\bigwedge\nolimits^{\bullet}T^{\ast}_{\wedge}M) \\ 
 \oplus \\ 
J^{-}_{(u,\eta)}
 \end{array}
 \end{array}
 \]
is an invertible, linear operator for every $\eta\neq 0$. 

Then, the operator \[\mathbf{A}_{_{\operatorname{DP[0]}}}=\begin{bmatrix}
 \operatorname{DP[0]} &  \operatorname{C}  \\ 
\mathcal{T} &  \operatorname{B}
\end{bmatrix}=\mathcal{F}^{-1}_{\eta\rightarrow u} \begin{bmatrix}
 \sigma^{1}_{\wedge}(\operatorname{DP[0]}) (u,\eta) &  \operatorname{c}(u,\eta)  \\ 
 \operatorname{t}(u,\eta) &  \operatorname{b}(u,\eta)
\end{bmatrix}\mathcal{F}_{u' \rightarrow \eta}\] acting on the spaces \[
\begin{array}{cccc}
\begin{bmatrix}
 \operatorname{DP[0]} & \operatorname{C}  \\ 
\mathcal{T} &  \operatorname{B}
\end{bmatrix}: & \begin{array}{c}
 \mathcal{W}^{s,\gamma}(M,\bigwedge\nolimits^{\bullet}T^{\ast}_{\wedge}M) \\ 
 \oplus \\ 
H^{s}(\edge, \CC^{N^{+}})
 \end{array} & \longrightarrow & \begin{array}{c}
 \mathcal{W}^{s-1,\gamma-1}(M,\bigwedge\nolimits^{\bullet}T^{\ast}_{\wedge}M) \\ 
 \oplus \\ 
 H^{s-1}(\edge, J^{-})
 \end{array}
\end{array}\]
is an elliptic edge operator ( see definition~\ref{ellipticedgeoperator}) for all $s\in\RR$ and $\gamma$ the admissible weight chosen at the beginning. 

In order to prove the claim that condition \eqref{restriction} is satisfied we have the following theorem which is contained in \cite{naza} theorem 6.30.
\begin{theorem}\label{obstructionvanishing}
If the Atiyah-Bott obstruction vanishes for an edge-degenerate operator $\operatorname{A}$ on the stretched manifold $\MM$, then there exists an elliptic edge problem  for $\operatorname{A}$.
\end{theorem}

In our case $\operatorname{DP[0]}$ is the Hodge-deRham operator, the vanishing of the Atiyah-Bott obstruction for this operator was proved by Atiyah-Patodi-Singer in \cite{patodi}, hence the topological condition  \eqref{restriction} is satisfied.

\section{Conormal deformations and regularity}\label{chaptermoduli}

\subsection{Conormal asymptotics}\label{asymp}
In this subsection we recall the basic facts of conormal asymptotics. For a complete presentation see \cite{schulze3} section 2.3.

A sequence  $O= \lbrace (p_j , m_j) \rbrace_{j\in\mathbb{N}}$ in $\CC\times\mathbb{Z}^{+}$ is called an asymptotic type for the weight data $\gamma\in\RR$ if \[\operatorname{Re}p_j <\frac{\operatorname{dim}\Xcal+1}{2}-\gamma\] and $\operatorname{Re}p_j \rightarrow -\infty$ when $j\rightarrow \infty.$   

\begin{definition}
Let  $O= \lbrace (p_j , m_j) \rbrace_{j\in\mathbb{N}}$ be an asymptotic type for the weight $\gamma\in\RR.$ The cone-Sobolev space with conormal asymptotics $O$, denoted by $\mathcal{K}^{s,\gamma}_{O}(\cone),$ is defined as the set of all $f\in \mathcal{K}^{s,\gamma}(\cone)$ such that for every $l\in\mathbb{N}$ there is $N(l)\in \mathbb{N}$ such that \[f(r,\sigma,y)-\omega(r)\sum\limits_{j=0}^{N(l)}\sum\limits_{k=0}^{m_j} c_{j,k}(\sigma) r^{-p_j}\operatorname{log}^{k}(r)\in \mathcal{K}^{s, \gamma+l}(\cone)\]
with $ c_{j,k}(\sigma)\in \smooth(\Xcal).$
\end{definition}

The space $\mathcal{K}^{s,\gamma}_{O}(\cone)$ has the structure of a Fr\'echet space given as an inductive limit of spaces with asymptotics of finite type $\mathcal{K}^{s,\gamma}_{O_k}(\cone)$ where $O_k=\lbrace (p_j,m_j)\in O: \frac{\operatorname{dim}\Xcal+1}{2}-\gamma-k<\operatorname{Re}p_j< \frac{\operatorname{dim}\Xcal+1}{2}-\gamma \rbrace$, see \cite{schulze2} sec. 8.1.1 for details. By using this inductive limit structure we define the edge-Sobolev space with conormal  asymptotics $O$ as the inductive limit of Fr\'echet spaces \[\mathcal{W}^{s,\gamma}\left(\RR^{q},\mathcal{K}^{s,\gamma}_{O}(\cone)\right):=\varprojlim\limits_{k} \mathcal{W}^{s,\gamma}\left(\RR^{q},\mathcal{K}^{s,\gamma}_{O_k}(\cone)\right).\] 
In particular,  if $f\in \mathcal{W}^{\infty,\gamma}_O (M)$ then for every $l\in\mathbb{N}$ there is $N(l)\in \mathbb{N}$ such that \[f(r,\sigma,y)-\omega(r)\sum\limits_{j=0}^{N(l)}\sum\limits_{k=0}^{m_j} c_{j,k}(\sigma) v_{j,k}(y) r^{-p_j}\operatorname{log}^{k}(r)\in \mathcal{W}^{\infty, \gamma+l}(M)\]
with $ c_{j,k}(\sigma)\in \smooth(\Xcal)$ and  $v_{j,k}(y)\in H^{\infty}(\edge)$, see \cite{schulze3} proposition  3.1.33.

\subsection{Conormal asymptotic embeddings}\label{subsectionPrelimRegularity}
Given a special Lagrangian submanifold of $\CC^{n}$ with edge singularity, $(M,\Phi)$ (see  ~\ref{examples}), in this section we define the moduli space of  special Lagrangian deformations of $(M,\Phi)$. Broadly speaking, we want to have in the moduli space all nearby special Lagrangian submanifolds with edge singularity. This rough idea has two aspects that must be considered for the moduli space. First, as the manifold $M$ is non-compact, the important aspect to consider when defining the concept of nearby submanifold is the behavior on the collar neighborhood $(0,\varepsilon)\times\Xcal\times\edge.$ Here we shall define the concept of nearby submanifold by means of its asymptotic behavior with respect to the conormal variable $r$ and weight $\gamma$. Second, the property of being special Lagrangian is completely determined by the equations \eqref{master}. As we mentioned in remark~\ref{remarkasymptoticscone} and~\ref{remarkasymptoticsedge} in section~\ref{ellipticsection}, solutions of the linearised deformation equation have conormal asymptotics (section~\ref{asymp}). This asymptotic behavior is transferred to the induced metric of the deformed submanifold making the induced metric asymptotic to the original edge metric $g_{_{M}}$ in a very special way that reflects the fact it comes from the solution of an edge-degenerate PDE on a singular space. All of these considerations are formalized in the following definition.
\begin{definition}\label{conormalembedding}
 Given an embedding $\Upsilon: M\longrightarrow \CC^{n}$ we say that $\Upsilon$ is  conormal asymptotic to  $(M,\Phi)$ with rate $\gamma$  if :
 \begin{enumerate}[i)]
 \item For every multi-index $\alpha$ we have  \[\abs{\partial ^{\alpha}_{(r,\sigma,u)}\left(\Upsilon (r,\sigma,u) -\Phi (r,\sigma,u)\right)} =O(r^{\gamma- \abs{\alpha}})\quad \forall(r,\sigma,u)\in(0,\varepsilon)\times\Xcal\times\edge;\]   
 \item $\Upsilon^{\ast}g_{_{\CC^{n}}}=r^{2}g_{_{\Xcal}}+dr^{2}+g_{_{\edge}}+\beta$ where $\beta$ is a symmetric 2-tensor on $\Upsilon(M)=M_{\Upsilon}$ such that their components $\beta_{ij}$ have conormal asymptotic expansions on the collar neighborhood $(0,\varepsilon)\times\Xcal\times\edge$ with respect to some asymptotic type associated to $\gamma$.  
\end{enumerate}
\end{definition}

Because we want to describe a small neighborhood of $(M,\Phi)$ in the moduli space by means of the Implicit Function Theorem~\ref{IFT} applied to a neighborhood of zero in edge-Sobolev spaces, we want to make sure that smooth elements in edge-Sobolev spaces with small norm will produce submanifolds. In order to show this, we will define a neighborhood of deformations i.e. we will define a small neighborhood of $M$ in $\CC^{n}$ such that small deformations  will be inside this neighborhood. Because of the geometric singularities of the manifold $M$ and the behavior of  the elements in edge-Sobolev spaces, this neighborhood will be constructed as an edge neighborhood to guarantee that small submanifolds induced by edge-degenerate forms will fit inside it.

\begin{proposition}
There exists an open edge neighborhood of the zero section in the normal bundle $\mathcal{N}\left((0,\varepsilon)\times\Xcal\times\edge\right)$ such that it is given by $V\times W$ where $V\subset\mathcal{N}((0,\varepsilon)\times\Xcal)\subset T\RR^{n}_x$ is an open conical set and $W\subset\mathcal{N}(\edge)\subset T\RR^{n}_y$ is an open set both of them being neighborhoods of the zero section in the corresponding normal bundles and diffeomorphic to an open edge set $\tilde{V}\times\tilde{W}\subset\xy\cong\CC^{n}$ with diffeomorphism given by the exponential map $\operatorname{exp}_{g_{_{\CC^{n}}}}$. Moreover, for every $\gamma>\frac{m+3}{2}$ and $s>\frac{q+m+1}{2}+ \mathfrak{c}_{\gamma}$, where $\mathfrak{c}_{\gamma}$ is the positive constant defined in \eqref{one-parameter bound}, there exists $\vartheta>0$ depending on $s$ and $\gamma$ such that \[\left\lbrace \mathcal{V}_{\Xi}\big|_{(0,\varepsilon)\times\Xcal\times\edge}:\Xi\in\mathcal{W}^{s,\gamma}(M, T^{\ast}_{\wedge}M) \text{ and } \norm{\Xi}_{s,\gamma}<\vartheta \right\rbrace\subset V\times W.\]  
\end{proposition}

\begin{proof}
First let's define the conical open neighborhood $V\subset\mathcal{N}((0,\varepsilon)\times\Xcal)\subset T\RR^{n}_x$. The tubular neighborhood theorem~\ref{tubular} applied to $\Xcal$ as a compact submanifold in $\RR^{n}$ gives us an open neighborhood of the zero section in $\mathcal{N}(\Xcal)$. Take $l_0>0$ to be  the maximum $l$ such that the uniform neighborhood $\lbrace X\in\mathcal{N}(\Xcal):\abs{X}_{g_{_{\RR^n}}}<l \rbrace$ is inside the tubular neighborhood. By applying the $\Rpos$-action defined on the cone $\cone$  to this uniform neighborhood we can obtain the desired open conical neighborhood $V$ of the zero section in the normal bundle $\mathcal{N}((0,\varepsilon)\times\Xcal)\subset T\RR^{n}_x$. Now, for any  section $\mathcal{V}$ of the normal bundle $\mathcal{N}(\Xcal)$ lying in $V$ we have  $\abs{\mathcal{V}(r,\sigma,u)}_{g_{_{\CC^n}}}< C_{1}\cdot  r$ for all $(r,\sigma,u)\in (0,\varepsilon)\times\Xcal\times\edge$, where the constant $C_1$ is independent of $\mathcal{V}$. The constant $C_1$ can be taken to be the maximum $l>0$ chosen above. Now choose a uniform tubular neighborhood of the zero section in the normal bundle $\mathcal{N}(\edge)$ given by $\lbrace Y\in\mathcal{N}(\edge):\abs{Y}_{g_{_{\RR^n}}}<\vartheta \rbrace$ for some $\vartheta>0$. Clearly this is possible because $\edge$ is compact. If necessary we can choose a smaller $\varepsilon$ such that $\vartheta>C_{1}\varepsilon$. Define $W$ as this uniform neighborhood of the zero section $W=\lbrace Y\in\mathcal{N}(\edge):\abs{Y}_{g_{_{\RR^n}}}<\vartheta \rbrace$. Then $V\times W$ is our open edge neighborhood of the zero section in $\mathcal{N}\left((0,\varepsilon)\times\Xcal\times\edge\right)$.

To prove the second part of the proposition let $\Xi\in\mathcal{W}^{s,\gamma}(M, T^{\ast}_{\wedge}M)$ with $s>\frac{q+m+1}{2}+ \mathfrak{c}_{\gamma}$ and consider its local expression in a neighborhood $(0,\varepsilon)\times\mathcal{U}\times\Omega\subset (0,\varepsilon)\times\Xcal\times\edge$ given by
\[\Xi\coordsimple=\CA\coordsimple dr+\sum\limits_{k=1}^{m}\CB\coordsimple rd\sigma_k+\sum\limits_{l=1}^{q}\cC\coordsimple du_{l},\]
where $\omega\phi_{j}\varphi_{\lambda}\CA$, $\omega\phi_{j}\varphi_{\lambda}\CB$ and $ \omega\phi_{j}\varphi_{\lambda}\cC$ belong to $\mathcal{W}^{s,\gamma}(M)$ as in lemma~\ref{localemma}. Then by proposition~\ref{edgeapplication}  there exists a constant $C>0$ depending only on $s$ and $\gamma$ such that \begin{equation}\label{boundforA}
\abs{\mathcal{A}(r,\sigma,u)}\leq C\norm{\Xi}_{s,\gamma}r^{\gamma-\frac{m+1}{2}}
\end{equation}
\begin{equation}\label{boundforB}
\abs{\mathcal{B}(r,\sigma,u)}\leq C\norm{\Xi}_{s,\gamma}r^{\gamma-\frac{m+1}{2}}
\end{equation}
\begin{equation}\label{boundforC}
\abs{\mathcal{C}(r,\sigma,u)}\leq C\norm{\Xi}_{s,\gamma}r^{\gamma-\frac{m+1}{2}}
\end{equation}
for all $(r,\sigma,u)\in (0,\varepsilon)\times\Xcal\times\edge$. Hence by lemma~\ref{localemma} there exists a constant $C'$ depending only on $s$ and $\gamma$  such that 
\begin{equation}\label{boundforC'}
\abs{\tilde{\mathcal{C}}(r,\sigma,u)}\leq C'\norm{\Xi}_{s,\gamma}r^{\gamma-\frac{m+1}{2}}
\end{equation} 
\begin{equation}
\abs{\mathcal{A}(r,\sigma,u)\theta_i+\tilde{\mathcal{B}_i}(r,\sigma,u)}\leq C'\norm{\Xi}_{s,\gamma}r^{\gamma-\frac{m+1}{2}}.
\end{equation}
Then, by \eqref{boundforC'} and because $0<\varepsilon<1$,  we have that \begin{equation}\label{boundforC'insideneigh}
\abs{\tilde{\mathcal{C}}(r,\sigma,u)}\leq C_{1}r
\end{equation}  for all $(r,\sigma,u)\in (0,\varepsilon)\times\Xcal\times\edge$ if \begin{equation}\label{gammaineq} 
\gamma>\operatorname{log}\left( \frac{C_1}{C'\norm{\Xi}_{s,\gamma}} \right) \frac{1}{\operatorname{log}(r)}+\frac{m+3}{2}.
\end{equation} 
Note that if  $\norm{\Xi}_{s,\gamma}$ is small enough, then   \eqref{gammaineq} is satisfied and this implies \eqref{boundforC'insideneigh}. More precisely, if  $\frac{C_1}{C'}\geq \norm{\Xi}_{s,\gamma}$ then \eqref{gammaineq} is satisfied  as $\operatorname{log}(r)<0$ for $r\leq\varepsilon$. Therefore $\frac{C_1}{C'}\geq \norm{\Xi}_{s,\gamma}$ implies \[\abs{\tilde{\mathcal{C}}(r,\sigma,u)}\leq C_{1}r.\]
Analogously, $\frac{C_1}{C'}>\norm{\Xi}_{s,\gamma}$ implies that
 \[\abs{\mathcal{A}(r,\sigma,u)\theta_i+\tilde{\mathcal{B}_i}(r,\sigma,u)}\leq C_{1}r\]
 for any $\gamma>\frac{m+3}{2}$. Then it follows from our chose of $\vartheta$ that  \[\abs{\mathcal{A}(r,\sigma,u)\theta_i+\tilde{\mathcal{B}_i}(r,\sigma,u)}\leq \vartheta\]for all $(r,\sigma,u)\in (0,\varepsilon)\times\Xcal\times\edge$.
\end{proof} 

Observe that the manifold  $M\setminus \left( (0,\varepsilon)\times\partial\MM\right)$ is compact hence we can extend our edge neighborhood $V\times W$ to this compact space to get a open neighborhood of the zero section in $\mathcal{N}(M)$ such that near the edge this neighborhood corresponds to the edge open neighborhood constructed above. We denote this neighborhood as $\mathfrak{A}$. Moreover this proposition implies that any smooth form $\Xi\in\mathcal{W}^{s,\gamma}(M,T^{\ast}_{\wedge}M)$ as above with $\norm{\Xi}_{s,\gamma}<\vartheta$ produces a smooth embedded submanifold inside the neighborhood of deformations $\mathfrak{A}$. This submanifold is defined by the embedding $\operatorname{exp}_{g_{_{\CC^{n}}}}(\mathcal{V}_{\Xi})\circ\Phi$. 

\section{Regularity of Deformations}\label{regularity}
Let's consider an elliptic edge problem (see section~\ref{symbol-linearop}) for the operator $\operatorname{DP[0]}$ acting on edge-Sobolev spaces with admissible weigh $\gamma>\frac{m+1}{2}$,

\[
\mathbf{A}_{\operatorname{DP[0]}}=
\begin{bmatrix}
 \operatorname{DP[0]} & \operatorname{C}  \\ 
\mathcal{T} & \operatorname{B}
\end{bmatrix}:\]

\[\begin{array}{ccc}
 \begin{array}{c}
 \mathcal{W}^{s,\gamma}(M,\bigwedge\nolimits^{\bullet}T^{\ast}_{\wedge}M) \\ 
 \oplus \\ 
H^{s}(\edge, \CC^{N^{+}})
 \end{array} & \longrightarrow & \begin{array}{c}
 \mathcal{W}^{s-1,\gamma-1}(M,\bigwedge\nolimits^{\bullet}T^{\ast}_{\wedge}M) \\ 
 \oplus \\ 
 H^{s-1}(\edge, J^{-})
 \end{array}
\end{array}.\]

Then, by augmenting the deformation operator $\operatorname{P}=\operatorname{P}_{\omega_{_{\CC^{n}}}}\oplus \operatorname{P}_{\operatorname{Im}\Omega}$ with the trace operator \[\mathcal{T}: \mathcal{W}^{s,\gamma}(M,\bigwedge\nolimits^{\bullet}T^{\ast}_{\wedge}M) \longrightarrow  H^{s-1}(\edge, J^{-}),\] we obtain a non-linear boundary value problem for $\operatorname{P}$: \[\begin{bmatrix}
\operatorname{P}_{\omega_{_{\CC^{N}}}}\oplus P_{\operatorname{Im}\Omega}  \\ 
\mathcal{T}
\end{bmatrix}: \mathfrak{A}\subset \mathcal{W}^{s,\gamma}(M,T^{\ast}_{\wedge}M)\longrightarrow \begin{array}{c}
 \mathcal{W}^{s-1,\gamma-1}(M,\bigwedge\nolimits^{\bullet}T^{\ast}_{\wedge}M) \\ 
 \oplus \\ 
 H^{s-1}(\edge, J^{-})
 \end{array}\]  whose linearisation at zero is given by \[\begin{bmatrix}
\operatorname{DP[0]}\\ 
\mathcal{T}
\end{bmatrix}:  \mathcal{W}^{s,\gamma}(M,T^{\ast}_{\wedge}M)\longrightarrow \begin{array}{c}
 \mathcal{W}^{s-1,\gamma-1}(M,\bigwedge\nolimits^{\bullet}T^{\ast}_{\wedge}M) \\ 
 \oplus \\ 
 H^{s-1}(\edge, J^{-})
 \end{array}.\]
In this section we consider some properties of solutions of the equation \begin{equation}\label{boundaryvalue}\begin{bmatrix}
\operatorname{P}_{\omega_{_{\CC^{N}}}}\oplus \operatorname{P}_{\operatorname{Im}\Omega}  \\ 
\mathcal{T} \end{bmatrix}\left(\Xi\right)=0
\end{equation} where  $\Xi\in    \mathcal{W}^{s,\gamma}_{}(M,T^{\ast}_{\wedge}M).$
We are mainly interested in those solutions given by the Implicit Function Theorem for Banach  spaces (when applicable)  i.e. we assume that $\Xi=\Xi_1+\Xi_2$ where $\Xi_1$ is solution of the linear boundary value problem  \begin{equation}\label{linearbvp}\begin{bmatrix}
\operatorname{DP[0]}\\ 
\mathcal{T}
\end{bmatrix}(\Xi_1)=0
 \end{equation} and $\Xi_2$ belongs to the Banach space complement in $\mathcal{W}^{s,\gamma}_{}(M,T^{\ast}_{\wedge}M)$ defined by a splitting (not unique)  induced by the finite dimensional space $\operatorname{Ker}\mathbf{A}_{\operatorname{DP[0]}}$.
First we have some straightforward  observations. The ellipticity of the operator $\mathbf{A}_{\operatorname{DP[0]}}$ (theorem~\ref{obstructionvanishing}) and the fact that \[\mathbf{A}_{\operatorname{DP[0]}}\begin{bmatrix}
\Xi_1 \\ 
0
\end{bmatrix}=0 \] implies that $\Xi_1\in  \mathcal{W}^{\infty,\gamma}(M,T^{\ast}_{\wedge}M)$ by elliptic regularity (theorem~\ref{ellipticityforedge}). Moreover, because $\mathcal{W}^{s,\gamma}(M,T^{\ast}_{\wedge}M)\subset H^{s}_{loc}(M,T^{\ast}_{\wedge}M)$ for all $s\in\RR$ (\cite{schulze2}, section 9.3, proposition  5), standard Sobolev embeddings (theorem~\ref{thmsobolevembe}) imply that $\Xi_1$ is smooth. The ellipticity of $\mathbf{A}_{\operatorname{DP[0]}}$ implies the existence of a parametrix  $\operatorname{B}_{\operatorname{DP[0]}}$ with asymptotics $O$ (theorem~\ref{ellipticityforedge}) i.e. 
\[
\operatorname{B}_{\operatorname{DP[0]}}\mathbf{A}_{\operatorname{DP[0]}}-\operatorname{I}:\begin{array}{cccc}
 & \begin{array}{c}
 \mathcal{W}^{s,\gamma}(M,\bigwedge\nolimits^{\bullet}T^{\ast}_{\wedge}M) \\ 
 \oplus \\ 
H^{s}(\edge, \CC^{N^{+}})
 \end{array} & \longrightarrow & \begin{array}{c}
 \mathcal{W}^{\infty,\gamma}_{O}(M,\bigwedge\nolimits^{\bullet}T^{\ast}_{\wedge}M) \\ 
 \oplus \\ 
 H^{\infty}(\edge, J^{-})
 \end{array}
\end{array}.\] Consequently, any element in the kernel of the operator $\mathbf{A}_{\operatorname{DP[0]}}$ belongs to $ \mathcal{W}^{\infty,\gamma}_{O}(M,\bigwedge\nolimits^{\bullet}T^{\ast}_{\wedge}M)$ for some asymptotic type $O$ associated to $\gamma$. In particular \begin{equation}\label{regularitylinearpart}
\Xi_1\in \mathcal{W}^{\infty,\gamma}_{O}(M,T^{\ast}_{\wedge}M).
\end{equation} 
Now let's consider the regularity properties of $\Xi_2$. 

Let $\Xi\in \mathcal{W}^{s,\gamma}_{}(M,T^{\ast}_{\wedge}M)$ such that
 \begin{equation}\label{masteroperator}(\operatorname{P}_{\omega_{_{\CC^{n}}}}\oplus \operatorname{P}_{\operatorname{Im}\Omega})(\Xi)=0. 
 \end{equation} Hence $\operatorname{exp}_{g_{_{\CC^{n}}}}(\mathcal{V}_{\Xi})\circ\Phi:M\longrightarrow\CC^{n}$ is a special Lagrangian submanifold. Harvey and Lawson pointed out in \cite{harveylawson} theorem 2.7 that $C^{2}$ special Lagrangian submanifolds in $\CC^{n}$ are real analytic, in particular they are smooth. Therefore, by choosing s large enough, \eqref{masteroperator} implies that $\Xi\in\smooth (M,T^{\ast}_{\wedge}M)$ which, together with \eqref{regularitylinearpart},  allow us to conclude that $\Xi_2$ is smooth.
 
Even though $\Xi_1+\Xi_2$ is solution of the non-linear edge boundary value problem \eqref{boundaryvalue}
we cannot conclude immediately that $\Xi$ has a conormal asymptotic expansion near the singular set $\edge$. The edge calculus tell us that solutions of the linearised equation \eqref{linearbvp}, here denoted by $\Xi_1$, have such asymptotics. It turns out that it is possible to prove that $\Xi_2$ also has a conormal expansion i.e. the whole solution of the non-linear edge boundary value problem has conormal expansion. In order to prove this we follow and adapt to our very specific setting in the next two propositions the general argument in \cite{rochon}  theorem 5.1. The author thanks Fr\'{e}d\'{e}ric Rochon for pointing out and explaining his work. 

Observe that 
\[\begin{bmatrix}
\operatorname{P}_{\omega_{_{\CC^{N}}}}\oplus \operatorname{P}_{\operatorname{Im}\Omega}  \\ 
\mathcal{T}
\end{bmatrix}(\Xi_1+\Xi_2)= \begin{bmatrix}
0  \\ 
0
\end{bmatrix}\] implies that $\mathcal{T}(\Xi_2)=0$ because $\mathcal{T}(\Xi_1)=0$ due to the fact that $\Xi_1$ is solution of the linearised equation  \eqref{linearbvp}. By writing the non-linear equation as $\operatorname{P}_{\omega_{_{\CC^{N}}}}\oplus \operatorname{P}_{\operatorname{Im}\Omega}=\operatorname{DP[0]}+\operatorname{Q}$ (see proposition~\ref{localprop} and \eqref{secondnonlinearterm}) where $\operatorname{Q}$ is a non-linear operator locally defined by the sum of products of 2 or more operators in $\operatorname{Diff}^{1}_{\operatorname{edge}}(M)$ acting on $\Xi_1$ or $\Xi_2$ we have \begin{equation}\label{decompositionnonlineareq}
(\operatorname{DP[0]}+\operatorname{Q})(\Xi_2)=-\operatorname{Q}(\Xi_1)-\sum\limits_{j\geq 2}\operatorname{Q}_{i_{1}}(\Xi_{\bullet})\cdots\operatorname{Q}_{i_{j}}(\Xi_{\bullet}).
\end{equation} In order to avoid cumbersome notation to keep track of the specific asymptotic types, we say that an element belongs to $\mathcal{W}^{s,\gamma}_{\operatorname{As}}(M,T^{\ast}_{\wedge}M)$ if it belongs to the edge-Sobolev space with some asymptotic type associated to $\gamma$.

\begin{proposition}\label{recursion} Let $\xi_1\in\mathcal{W}^{\infty,\gamma-1}_{\operatorname{\operatorname{As}}}(M,\bigwedge^{\bullet} T^{\ast}_{\wedge}M)$ and $\Xi_2\in\mathcal{W}^{\infty,\gamma}(M,T^{\ast}_{\wedge}M)$ such that 
\begin{equation}\label{regequation}
(\operatorname{DP[0]}+\operatorname{Q})(\Xi_2)=\xi_1.
\end{equation}
Assume that $\gamma$ is an admissible weight for $\operatorname{DP[0]}$ and there exists $\beta>0$ such that $\operatorname{Q}(\Xi_2)\in\mathcal{W}^{\infty,\gamma+\beta}(M,\bigwedge^{\bullet}T^{\ast}_{\wedge}M)$ and $\gamma+\beta+1$ is an admissible weight. Then $\Xi_2=E_1+E_2$ with $E_2\in\mathcal{W}^{\infty,\gamma}_{\operatorname{As}}(M,T^{\ast}_{\wedge}M)$ and $E_1\in\mathcal{W}^{\infty,\gamma+\beta+1}(M,T^{\ast}_{\wedge}M).$
\end{proposition}

\begin{proof}
Let's consider the Fredholm operator defined by $\operatorname{DP[0]}$ acting on edge-Sobolev spaces with weight $\gamma+\beta+1$
\[\mathbf{A}_{\operatorname{DP[0]},\gamma+\beta+1}: \]
\[\begin{array}{ccc}
 \begin{array}{c}
 \mathcal{W}^{s+1,\gamma+\beta+1}(M,\bigwedge\nolimits^{\bullet}T^{\ast}_{\wedge}M) \\ 
 \oplus \\ 
H^{s+1}(\edge, \CC^{N^{+}})
 \end{array} & \longrightarrow & \begin{array}{c}
 \mathcal{W}^{s,\gamma+\beta}(M,\bigwedge\nolimits^{\bullet}T^{\ast}_{\wedge}M) \\ 
 \oplus \\ 
 H^{s}(\edge, J^{-})
 \end{array}
\end{array}.\]
Because $\operatorname{Coker}\mathbf{A}_{\operatorname{DP[0]},\gamma+\beta+1}$ is finite dimensional and  $\smooth_{0}(M,\bigwedge\nolimits^{\bullet}T^{\ast}_{\wedge}M)$ is a dense subset of the edge-Sobolev spaces we have \[\begin{bmatrix}
-\operatorname{Q}(\Xi_2)\\ 0
\end{bmatrix}=\mathbf{A}_{\operatorname{DP[0]},\gamma+\beta+1}\begin{bmatrix}
E_1 \\ 
e_1
\end{bmatrix}+\begin{bmatrix}
\mathfrak{F}\\
f
\end{bmatrix}\] with \[\begin{bmatrix}
E_1 \\ 
e_1
\end{bmatrix}\in  \begin{array}{c}
 \mathcal{W}^{s+1,\gamma+\beta+1}(M,\bigwedge\nolimits^{\bullet}T^{\ast}_{\wedge}M) \\ 
 \oplus \\ 
H^{s+1}(\edge, \CC^{N^{+}})
 \end{array}\] and
$\mathfrak{F}\in\smooth_{0}(M,\bigwedge\nolimits^{\bullet}T^{\ast}_{\wedge}M)$, $ f\in\smooth_{0}(\edge,\CC^{N^{+}}).$
Observe that this implies \[\mathbf{A}_{\operatorname{DP[0]},\gamma+\beta+1}\begin{bmatrix}
E_1 \\ 
e_1
\end{bmatrix}\in \begin{array}{c}
 \mathcal{W}^{\infty,\gamma+\beta}(M,\bigwedge\nolimits^{\bullet}T^{\ast}_{\wedge}M) \\ 
 \oplus \\ 
H^{\infty}(\edge, \CC^{N^{+}})
 \end{array}\] as $\operatorname{Q}(\Xi_2)\in\mathcal{W}^{\infty,\gamma+\beta}(M,T^{\ast}_{\wedge}M)$. 
 
 Hence by elliptic regularity (theorem~\ref{ellipticityforedge}) \[\begin{bmatrix}
E_1 \\ 
e_1
\end{bmatrix}\in  \begin{array}{c}
 \mathcal{W}^{\infty,\gamma+\beta+1}(M,\bigwedge\nolimits^{\bullet}T^{\ast}_{\wedge}M) \\ 
 \oplus \\ 
H^{\infty}(\edge, \CC^{N^{+}})
 \end{array}.\]
 Now define $\begin{bmatrix}
E_2 \\  e_2 \end{bmatrix}:=\begin{bmatrix} \Xi_2\\ 0 \end{bmatrix}-\begin{bmatrix}
E_1 \\ 
e_1
\end{bmatrix}$, then  by \eqref{regequation} \[\mathbf{A}_{\operatorname{DP[0]},\gamma}\begin{bmatrix}
E_2 \\  e_2 \end{bmatrix}=\begin{bmatrix}
\xi_1 -\operatorname{Q}(\Xi_2)\\  0 \end{bmatrix}-\mathbf{A}_{\operatorname{DP[0]},\gamma}\begin{bmatrix}
E_1 \\ e_1 \end{bmatrix}.\]
Observe that as $\mathbf{A}_{\operatorname{DP[0]},\gamma}$ and $\mathbf{A}_{\operatorname{DP[0]},\gamma+\beta+1}$ are $2\times 2$ operator matrices with $\operatorname{DP[0]}$ in the upper left corner  they differ by Green operators with asymptotics  acting on the corresponding spaces (see \cite{schulze3} theorem 3.4.3). Hence we can write  $\mathbf{A}_{\operatorname{DP[0]},\gamma}=\mathbf{A}_{\operatorname{DP[0]},\gamma+\beta+1}-G_{\operatorname{DP[0]},\gamma+\beta+1}+G_{\operatorname{DP[0]},\gamma}$, where $G_{\operatorname{DP[0]},\gamma+\beta+1}$ is the Green operator matrix with the elliptic boundary conditions for $\operatorname{DP[0]}$ acting on spaces with  weight $\gamma+\beta+1$ and analogously  for $G_{\operatorname{DP[0]},\gamma}$. This implies \[\mathbf{A}_{\operatorname{DP[0]},\gamma}\begin{bmatrix}
E_2 \\  e_2 \end{bmatrix}=\begin{bmatrix}
\xi_1\\ 0
\end{bmatrix}+\begin{bmatrix}
\mathfrak{F}\\ {f}
\end{bmatrix}-(-G_{\operatorname{DP[0]},\gamma+\beta+1}+G_{\operatorname{DP[0]},\gamma})\begin{bmatrix}
E_1 \\  e_1 \end{bmatrix}\]
therefore, by the mapping properties of Green operators (see \cite{schulze3}, theorem 3.4.3) we have 
\[\mathbf{A}_{\operatorname{DP[0]},\gamma}\begin{bmatrix}
E_2 \\  e_2 \end{bmatrix}\in\mathcal{W}^{\infty,\gamma-1}_{\operatorname{As}}(M,\bigwedge\nolimits^{\bullet}T^{\ast}_{\wedge}M).\]
By elliptic regularity we conclude
\[\begin{bmatrix}
E_2 \\ 
e_2
\end{bmatrix}\in  \begin{array}{c}
 \mathcal{W}^{\infty,\gamma}_{\operatorname{As}}(M,\bigwedge^{\bullet}T^{\ast}_{\wedge}M) \\ 
 \oplus \\ 
H^{\infty}(\edge, \CC^{N^{+}})
 \end{array}\] and $\Xi_2=E_1+E_2$ as claimed.
\end{proof}

\begin{proposition}
Let $\Xi_1\in \mathcal{W}^{\infty,\gamma}_{\operatorname{As}}(M,T^{\ast}_{\wedge}M)$ and $\Xi_2\in \mathcal{W}^{\infty,\gamma}(M,T^{\ast}_{\wedge}M)$ with  an admissible weight $\gamma > \frac{m+5}{2}$ such that \begin{equation}\label{spliteq}
(\operatorname{DP[0]}+\operatorname{Q})(\Xi_1+\Xi_2)=0
\end{equation} and  $\operatorname{DP[0]}(\Xi_1)=0.$ Then  $\Xi_2\in \mathcal{W}^{\infty,\gamma}_{\operatorname{As}}(M,T^{\ast}_{\wedge}M).$
\end{proposition}

\begin{proof}
Equation \eqref{spliteq} can be written as 
\begin{equation}\label{rearrangement}
(\operatorname{DP[0]}+\operatorname{Q})(\Xi_2)=-\operatorname{Q}(\Xi_1)-\sum\limits_{j\geq 2}\operatorname{Q}_{i_{1}}(\Xi_{\bullet})\cdots\operatorname{Q}_{i_{j}}(\Xi_{\bullet}),
\end{equation}
  (see \eqref{decompositionnonlineareq}). The right hand side of \eqref{rearrangement} consists of products where at least one of the operators in each of these products is acting on $\Xi_1$, let's say $\operatorname{Q}_{1}(\Xi_1)$. Now, the term $\operatorname{Q}_{1}(\Xi_1)$ has asymptotics (associated to $\gamma-1$) and the other elements  $\operatorname{Q}_{i_{k}}(\Xi_{\bullet})$ in the products satisfy the estimate  \eqref{edgeestimate} near the edge. Hence by \cite{schulze3} theorem 2.3.13, multiplication by elements in $\smooth_{0}(\MM)$ induces a continuous operator on Sobolev spaces with asymptotics (with possibly different asymptotic type but associated to the same weight). Therefore we conclude that the right hand side in \eqref{rearrangement} belongs to $\mathcal{W}^{\infty,\gamma-1}_{\operatorname{As}}(M,T^{\ast}_{\wedge}M).$
  
Now, the fact that $\gamma> \frac{m+5}{2}$  together with \eqref{morereg} implies that $\operatorname{Q}(\Xi_2)\in\mathcal{W}^{\infty,\gamma+\beta}_{}(M,T^{\ast}_{\wedge}M)$ for some $\beta>0$.
If necessary we can choose $\beta$ small enough such that $\gamma+\beta+1$ is an admissible weight for $\operatorname{DP[0]}$. Then, proposition~\ref{recursion} implies  $\Xi_2=E_1+E_2$ with $E_1$ belonging to $\mathcal{W}^{\infty,\gamma+\beta+1}(M,T^{\ast}_{\wedge}M)$ and $E_2\in \mathcal{W}^{\infty,\gamma}_{\operatorname{As}}(M,T^{\ast}_{\wedge}M).$

 Define $S_1:=\Xi_2-E_2=E_1$ and observe that \eqref{rearrangement} and a similar argument as above implies   \[(\operatorname{DP[0]}+\operatorname{Q})(S_1):=\xi_2\in\mathcal{W}^{\infty,\gamma-1}_{\operatorname{As}}(M,T^{\ast}_{\wedge}M).\] Moreover, \eqref{morereg} and $\gamma>\frac{m+5}{2}$ imply that $\operatorname{Q}(S_1)$ belongs to the edge space $\mathcal{W}^{\infty,\gamma+\beta+1+\beta'}(M,T^{\ast}_{\wedge}M) $ for some $\beta'>0$. Then by following the same argument as in proposition~\ref{recursion} we have \[S_1=E_3+E_4\] with  $E_3\in\mathcal{W}^{\infty,\gamma+\beta+\beta'+2}(M,T^{\ast}_{\wedge}M)$ and $E_4=\mathcal{W}^{\infty,\gamma}_{\operatorname{As}}(M,T^{\ast}_{\wedge}M).$ Hence we have found an element $E_2+E_4\in\mathcal{W}^{\infty,\gamma}_{\operatorname{As}}(M,T^{\ast}_{\wedge}M) $ such that \[\Xi_2-(E_2+E_4)=E_3\in\mathcal{W}^{\infty,\gamma+\beta+\beta'+2}(M,T^{\ast}_{\wedge}M).\]
We continue this recursion argument by setting $S_2=S_1-E_4=\Xi_2-E_2-E_4=E_3$ and  \[(\operatorname{DP[0]}+\operatorname{Q})(S_2)=\xi_3.\] Therefore by means of this iterative process  we conclude that for every $l>0$ there exists $E\in \mathcal{W}^{\infty,\gamma}_{\operatorname{As}}(M,T^{\ast}_{\wedge}M)$ such that $\Xi_2-E\in\mathcal{W}^{\infty,\gamma+l}(M,T^{\ast}_{\wedge}M).$ 
Therefore $\Xi_2\in \mathcal{W}^{\infty,\gamma}_{\operatorname{As}}(M,T^{\ast}_{\wedge}M).$
\end{proof}

\section{The Moduli space}\label{sectionTheModuliSpace}
In this section we define the moduli space we are interested in and prove the main result of this paper. Let $(M,\Phi)$ be a special Lagrangian submanifold with edge singularity i.e. $M$ is a manifold with edge singularity (section~\ref{examples}) and $\Phi:M\longrightarrow\CC^{n}$ is an edge special Lagrangian embedding (see definition~\ref{edgeembedding}).

\begin{definition}
Given an admissible weight $\gamma>\frac{m+5}{2}$, we define the moduli space  of conormal asymptotic special Lagrangian deformations of $(M,\Phi)$ with rate $\gamma$ and elliptic  boundary trace condition $\mathcal{T}$ as the space of smooth embeddings $\Upsilon:M\longrightarrow \CC^{n}$, isotopic to $\Phi$ and conormal asymptotic to $(M,\Phi)$ with rate $\gamma$, such that they satisfy the boundary condition $\mathcal{T} (\Upsilon) =0$ where $\mathcal{T}$ is a trace pseudo-differential operator: \begin{equation}\label{tracepseudo}
\mathcal{T}:\mathcal{W}^{s,\gamma}(M,T^{\ast}_{\wedge}M)\longrightarrow H^{s-1}(\edge,J^{-})
\end{equation} 
with $s>\operatorname{max}\left\lbrace \frac{m+1+q}{2}+\mathfrak{c}_{\gamma}, \frac{m+3+q}{2}\right\rbrace$ such that $\mathcal{T}$ belongs to a set of boundary conditions for an elliptic edge boundary value problem for the operator $\operatorname{DP[0]}$ on the edge-Sobolev space $\mathcal{W}^{s,\gamma}(M,T^{\ast}_{\wedge}M)$.
We denote this moduli space as $\mathfrak{M}(M, \Phi, \mathcal{T},\gamma)$.   
\end{definition}

Given a special Lagrangian submanifold with edge singularity $(M,\Phi)$, the moduli space of conormal asymptotic special Lagrangian deformations depends on the parameters $\gamma$ and $\mathcal{T}$. The role of the weight $\gamma$ is explained in the definition of conormal asymptotic embedding (see definition~\ref{conormalembedding}).  The existence of a trace operator $\mathcal{T}$ and its role in the elliptic theory of edge degenerate equations was discussed in section~\ref{symbol-linearop}. Here we want to include further details about $\mathcal{T}$. The edge symbol defining $\mathcal{T}$ was defined in \eqref{symboltrace} as a family of continuous linear maps $\operatorname{t}(u,\eta):\scone\longrightarrow J^{-}_{(u,\eta)}$ continuously parametrized by $T^{\ast}\edge\setminus\{ 0\}$. Note that as $J^{-}$ is a finite rank vector bundle, the operators $\operatorname{t}(u,\eta)$ are finite rank operators (in particular compact operators). The trace operator $\mathcal{T}$  was locally defined by \[\mathcal{T}=\operatorname{Op}(\operatorname{t}(u,\eta))=\mathcal{F}^{-1}_{\eta\rightarrow u}\operatorname{t}(u,\eta)\mathcal{F}^{}_{u'\rightarrow \eta}.\]

Now recall from section~\ref{symbol-linearop} that the fibers of the vector bundle $J^{-}$ consist mainly of isomorphic images of finite dimensional kernels of Fredholm operators acting on the extension of cone-Sobolev spaces defined by \eqref{Fredholmedgesymbol}. The operator-valued symbols $\operatorname{t}(u,\eta)$ correspond to the projection of the cone-Sobolev space $\scone$ onto the finite dimensional kernel of \eqref{Fredholmedgesymbol}. By considering a local trivialization of $J^{-}$ over an open subset $\Omega\subset\edge$ and using the fact that $\operatorname{t}(u,\eta)$ are projections, it is possible to prove the following proposition.   
\begin{proposition}
 Locally on $\Omega\subset\edge$, the trace operator $\mathcal{T}$ in \eqref{tracepseudo} acts on each of the components of the stretched cotangent bundle  $T^{\ast}_{\wedge}M$ as an integral operator with kernel in $\smooth(\Omega\times\cone\times\Omega)\otimes\CC^{N^{-}}$. 
\end{proposition} 
This proposition  and its proof is contained in the more general result presented in \cite{schulze3} proposition 3.4.6. The reader is referred to that book for details.

Given a operator-valued trace symbol $\operatorname{t}(u,\eta)$, the trace operator $\mathcal{T}$ is unique module negligible operators from the point of view of ellipticity and smoothness. See \cite{naza} section 6.1 for details.  

\begin{theorem}\label{main}
Locally near $M$ the moduli space $\mathfrak{M}(M, \Phi, \mathcal{T},\gamma)$ is homeomorphic to the zero set of a smooth map $\mathfrak{G}$ between smooth manifolds $\mathcal{M}_1 $, $\mathcal{M}_2 $  given as  neighborhoods of zero in finite dimensional Banach spaces. The map $\mathfrak{G}: \mathcal{M}_1 \longrightarrow \mathcal{M}_2$ satisfies $\mathfrak{G}(0)=0$ and $\mathfrak{M}(M, \Phi, \mathcal{T},\gamma)$ near $M$ is a smooth manifold of finite dimension when $\mathfrak{G}$ is the zero map.
\end{theorem}

\begin{proof}
As $\gamma$ is an admissible weight we have that

\[ \mathbf{A}_{\operatorname{DP[0]}}=
\begin{bmatrix}
 \operatorname{DP[0]} & C  \\ 
\mathcal{T} & B
\end{bmatrix}:  \]
\[ \begin{array}{ccc}
 \begin{array}{c}
 \mathcal{W}^{s,\gamma}(M,\bigwedge\nolimits^{\bullet}T^{\ast}_{\wedge}M) \\ 
 \oplus \\ 
H^{s}(\edge, \CC^{N^{+}})
 \end{array} & \longrightarrow & \begin{array}{c}
 \mathcal{W}^{s-1,\gamma-1}(M,\bigwedge\nolimits^{\bullet}T^{\ast}_{\wedge}M) \\ 
 \oplus \\ 
 H^{s-1}(\edge, J^{-})
 \end{array}
\end{array}\]
is a Fredholm operator. Thus the cokernel is a finite dimensional space and it can be identified with a finite dimensional subspace in the codomain of $\mathbf{A}_{\operatorname{DP[0]}}$ (denoted as $\operatorname{Coker}\mathbf{A}_{\operatorname{DP[0]}}$) such that it splits the codomain in the following way  \begin{equation}
\begin{array}{c}
 \mathcal{W}^{s-1,\gamma-1}(M,\bigwedge\nolimits^{\bullet}T^{\ast}_{\wedge}M) \\ 
 \oplus \\ 
 H^{s-1}(\edge, J^{-})
 \end{array}=\operatorname{Im}\mathbf{A}_{\operatorname{DP[0]}}\oplus\operatorname{Coker}\mathbf{A}_{\operatorname{DP[0]}}.
\end{equation}
Now consider the Banach space  
 \[\left( \begin{array}{c}
 \mathcal{W}^{s,\gamma}(M,\bigwedge\nolimits^{\bullet}T^{\ast}_{\wedge}M) \\ 
 \oplus \\ 
H^{s}(\edge, \CC^{N^{+}})
 \end{array}\right)\oplus \operatorname{Coker}\mathbf{A}_{\operatorname{DP[0]}}.\]
 Consider the following extension $\hat{\operatorname{P}}$ of the deformation operator to this space 
 \[\hat{\operatorname{P}}:\left(\begin{array}{c}
 \mathcal{W}^{s,\gamma}(M,\bigwedge\nolimits^{\bullet}T^{\ast}_{\wedge}M) \\ 
 \oplus \\ 
H^{s}(\edge, \CC^{N^{+}})
 \end{array}\right) \oplus \operatorname{Coker}\mathbf{A}_{\operatorname{DP[0]}}\longrightarrow \begin{array}{c}
 \mathcal{W}^{s-1,\gamma-1}(M,\bigwedge\nolimits^{\bullet}T^{\ast}_{\wedge}M) \\ 
 \oplus \\ 
 H^{s-1}(\edge, J^{-})
 \end{array} \]
given by \[\hat{\operatorname{P}}\left( \begin{bmatrix}
 \Xi \\ 
 g
 \end{bmatrix}, \begin{bmatrix}  v\\ w    \end{bmatrix}  \right)=\mathbf{A}_{\operatorname{P}}\begin{bmatrix}
 \Xi \\ 
 g
 \end{bmatrix}+\begin{bmatrix}  v\\ w    \end{bmatrix},\]
where we are using the notation $\mathbf{A}_{\operatorname{P}}$ for the operator  $\begin{bmatrix}
 \operatorname{P} & \operatorname{C}  \\ 
\mathcal{T} & \operatorname{B}
\end{bmatrix}.$
 
Hence \[\operatorname{D\hat{\operatorname{P}}[0]}\left( \begin{bmatrix}
 \Xi \\ 
 g
 \end{bmatrix}, \begin{bmatrix}  v\\ w    \end{bmatrix}  \right)=\begin{bmatrix}
\operatorname{DP[0]} & \operatorname{C}  \\ 
\mathcal{T} & \operatorname{B}
\end{bmatrix}\begin{bmatrix}
 \Xi \\ 
 g
 \end{bmatrix}+\begin{bmatrix}  v\\ w    \end{bmatrix} \]
 and \[\operatorname{Ker}\operatorname{D\hat{\operatorname{P}}[0]}=\operatorname{Ker}\mathbf{A}_{\operatorname{DP[0]}}\times \lbrace 0 \rbrace.\] Observe that $\operatorname{D\hat{\operatorname{P}}[0]}$ is surjective and $\operatorname{Ker}\operatorname{D\hat{\operatorname{P}}[0]}$ is finite dimensional. Then 
 \[\left(\begin{array}{c}
 \mathcal{W}^{s,\gamma}(M,\bigwedge\nolimits^{\bullet}T^{\ast}_{\wedge}M) \\ 
 \oplus \\ 
H^{s}(\edge, \CC^{N^{+}})
 \end{array}\right)\oplus \operatorname{Coker}\mathbf{A}_{\operatorname{DP[0]}}=\left(\operatorname{Ker}\operatorname{D\hat{P}[0]}\oplus N\right) \oplus \operatorname{Coker}\mathbf{A}_{\operatorname{DP[0]}}\] for some closed subspace $N$.
 
 By the Implicit Function Theorem~\ref{IFT} there exists $\mathcal{U}_1\subset \operatorname{Ker}\mathbf{A}_{\operatorname{DP[0]}}$, $\mathcal{U}_2=\mathcal{U}'_2\times\mathcal{U}''_2\subset N\oplus \operatorname{Coker}\mathbf{A}_{\operatorname{DP[0]}}$ and a smooth map $\mathfrak{G}_1\times\mathfrak{G}_2:\mathcal{U}_1\longrightarrow\mathcal{U}'_2\times\mathcal{U}''_2$ such that \[\hat{\operatorname{P}}^{-1}(0)\cap(\mathcal{U}_1\times\mathcal{U}_2)=\left\lbrace \left( \begin{bmatrix} a \\ b  \end{bmatrix},\mathfrak{G}_1\left(\begin{bmatrix} a \\ b  \end{bmatrix}\right),\mathfrak{G}_2\left(\begin{bmatrix} a \\ b  \end{bmatrix}\right)\right):  \begin{bmatrix} a \\ b  \end{bmatrix}\in\mathcal{U}_1 \right\rbrace\] \[\subset \operatorname{Ker}\operatorname{D\hat{P}[0]}\oplus N\oplus \operatorname{Coker}\mathbf{A}_{\operatorname{DP[0]}}.\]
This give us a description of the elements in the null set of the non-linear operator $\hat{\operatorname{P}}$ in a neighborhood of zero in terms of elements in  $\operatorname{Ker}\operatorname{D\hat{P}[0]}.$ In order to pass to solutions of the deformation operator $\operatorname{P}$ in terms of  $\operatorname{Ker}\operatorname{DP[0]}$ we have the following. 

Observe that  \[\left( \begin{bmatrix} a \\ b  \end{bmatrix},\mathfrak{G}_1\left(\begin{bmatrix} a \\ b  \end{bmatrix}\right),\mathfrak{G}_2\left(\begin{bmatrix} a \\ b  \end{bmatrix}\right)\right)\in \hat{\operatorname{P}}^{-1}(0)\cap(\mathcal{U}_1\times\mathcal{U}_2),\]
implies \[\hat{P}\left( \begin{bmatrix} a \\ b  \end{bmatrix},\mathfrak{G}_1\left(\begin{bmatrix} a \\ b  \end{bmatrix}\right),\mathfrak{G}_2\left(\begin{bmatrix} a \\ b  \end{bmatrix}\right)\right)=\]\[\mathbf{A}_{P}\left( \begin{bmatrix} a \\ b  \end{bmatrix},\mathfrak{G}_1\left(\begin{bmatrix} a \\ b  \end{bmatrix}\right)\right)+\mathfrak{G}_2\left(\begin{bmatrix} a \\ b  \end{bmatrix}\right)=0.\]
Hence the term $\mathfrak{G}_2\left(\begin{bmatrix} a \\ b  \end{bmatrix}\right)$ (that belongs to $\operatorname{Coker}\mathbf{A}_{\operatorname{DP[0]}}$) represents an obstruction to lifting the infinitesimal solution $\begin{bmatrix} a \\ b  \end{bmatrix}$ to an authentic solution \[\left(\begin{bmatrix} a \\ b  \end{bmatrix},\mathfrak{G}_1\left(\begin{bmatrix} a \\ b  \end{bmatrix}\right)\right)\] of the non-linear operator $\mathbf{A}_{\operatorname{P}}$. Therefore if all obstructions vanish i.e. if \begin{equation}\label{obstructionmap}
\mathfrak{G}_2:\mathcal{U}_1\subset\operatorname{Ker}\mathbf{A}_{\operatorname{DP[0]}}\longrightarrow \mathcal{U}''_2\subset\operatorname{Coker}\mathbf{A}_{\operatorname{DP[0]}}
\end{equation}
is the zero map we have \[\mathbf{A}_{\operatorname{P}}^{-1}(0)\cap(\mathcal{U}_1\times\mathcal{U}'_2)=\left\lbrace \left( \begin{bmatrix} a \\ b  \end{bmatrix},\mathfrak{G}_1\left(\begin{bmatrix} a \\ b  \end{bmatrix}\right)\right):  \begin{bmatrix} a \\ b  \end{bmatrix}\in\mathcal{U}_1 \right\rbrace.\]
and the set $\mathbf{A}_{\operatorname{P}}^{-1}(0)\cap(\mathcal{U}_1\times\mathcal{U}'_2)$ is diffeomorphic to $\mathcal{U}_1\subset\operatorname{Ker}\mathbf{A}_{\operatorname{DP[0]}}.$

Consequently, if the obstructions vanish, small solutions of the non-linear boundary value problem \eqref{boundaryvalue} are given by
\begin{align*}
& \mathbf{A}_{\operatorname{P}}^{-1}(0)\cap(\mathcal{U}_1\times\mathcal{U}'_2)\cap\left(\begin{array}{c}
 \mathcal{W}^{s,\gamma}(M,T^{\ast}_{\wedge}M) \\ 
 \oplus \\ 
\lbrace 0 \rbrace
 \end{array}\right)
 \\&  =\left\lbrace \left(\Xi_1,\mathfrak{G}_1(\Xi_1)\right): \begin{bmatrix}
\operatorname{P}_{\omega_{_{\CC^{N}}}}\oplus \operatorname{P}_{\operatorname{Im}\Omega}  \\ 
\mathcal{T}
\end{bmatrix}(\Xi_1+\mathfrak{G}_1(\Xi_1))=0 \right\rbrace.
\end{align*}

This is a non-empty open neighborhood of zero  in  $\mathcal{W}^{s,\gamma}(M,T^{\ast}_{\wedge}M)$ diffeomorphic to an open set of the finite dimensional space  \begin{equation}\label{bvpfinal}
\operatorname{Ker}\begin{bmatrix} \operatorname{DP[0]} \\ \mathcal{T}  \end{bmatrix}.
\end{equation}
Thus we can conclude that when  $\mathfrak{G}_2$ is the zero map the moduli space is a smooth manifold of finite dimension less or equal to the dimension of the kernel of the linear boundary value problem \[  \begin{bmatrix} \operatorname{DP[0]} \\ \mathcal{T}  \end{bmatrix}: \mathcal{W}^{s,\gamma}(M,T^{\ast}_{\wedge}M)\longrightarrow \begin{array}{c}
\mathcal{W}^{s-1,\gamma-1}(M,\bigwedge^{\bullet}T^{\ast}_{\wedge}M) \\ 
\oplus \\ 
H^{s-1}(\edge, J^{-})
\end{array}. \]

\end{proof}

\section{Conclusions and final remarks}\label{sectionConclusions}
In this paper we have put the problem of deforming special Lagrangain submanifolds with edge singularities into the framework of the edge calculus developed by B.-W. Schulze. Our main theorem, theorem~\ref{main} in section~\ref{sectionTheModuliSpace},  says that when  the map \eqref{obstructionmap}  \[\mathfrak{G}_2:\mathcal{U}_1\subset\operatorname{Ker}\mathbf{A}_{\operatorname{DP[0]}}\longrightarrow \mathcal{U}''_2\subset\operatorname{Coker}\mathbf{A}_{\operatorname{DP[0]}}\]
is the zero map, the moduli space $\mathfrak{M}(M, \Phi, \mathcal{T},\gamma)$ is a smooth manifold of finite dimension. For every small solution $\Xi_1$ of the linearised boundary value problem $\mathbf{A}_{\operatorname{DP[0]}}$, the map $\mathfrak{G}_2$ gives us an obstruction \[\mathfrak{G}_2(\Xi_2)\in\mathcal{U}''_2\subset
\operatorname{Coker}\mathbf{A}_{\operatorname{DP[0]}},\] to lift the linearised solution to a solution of the non-linear deformation operator with boundary condition. When the obstruction space $\mathcal{U}''_2$ vanishes it follows immediately that there are no obstructions, as the map $\mathfrak{G}_2$ is trivially the zero map, and the moduli space is smooth and finite dimensional. 

A careful analysis of the obstruction space is needed to determine under which conditions it vanishes. In \cite{joyceII}, Joyce analyzed the obstruction space of the moduli space of deformations of special Lagrangian submanifolds with conical singularities. He found that the obstruction space depends only on the cones that model the singularities. In the edge singular case we expect a similar result i.e. the obstruction space  depends only on the geometric structures that model the singularity, namely,  the cone $\cone$ and the edge $\edge$.
  
If the obstruction space vanishes (therefore $\operatorname{Coker}\mathbf{A}_{\operatorname{DP[0]}}=\{ 0\}$)  the moduli space is a smooth manifold of finite dimension. The next step is to determine its expected dimension. From theorem~\ref{main} we only know that $\operatorname{dim}\mathfrak{M}(M, \Phi, \mathcal{T},\gamma)\leq \operatorname{dim}\operatorname{Ker}\mathbf{A}_{\operatorname{DP[0]}}=\operatorname{Ind}\mathbf{A}_{\operatorname{DP[0]}}$.
In order to compute the dimension we need to consider the index of edge-degenerate operators and Hodge theory in the edge singular context. In this direction we consider the material related to index theory in \cite{naza} chapter 5 quite relevant for this purpose. Moreover, some elements of Hodge theory on manifolds with edge singularities have been studied in \cite{mazzeo} and \cite{schulze4}. These references might be helpful to compute the expected dimension of our moduli space $\mathfrak{M}(M, \Phi, \mathcal{T},\gamma)$.

\begin{acknowledgements}
This work started in a series of informal conversations the author had with Spiro Karigiannis (U. of Waterloo) during the Geometric Analysis Colloquium at Fields Institute in the Winter of 2014. The author thanks Spiro Karigiannis for his interest and becoming co-supervisor of this research project. The author thanks Fr\'ed\'eric Rochon at UQAM for the suggestions he made and the hospitality during the author's visit to UQAM. The author was supported by the University of Western Ontario. Thanks to Tatyana Barron for all her support and assistance at UWO during this research project. 

The author was supported by the University of Western Ontario through a Western   Graduate Research Scholarship. The results in this paper were obtained during his years as a graduate student at the department of Mathematics.
\end{acknowledgements}

\newpage

\appendix 
\section{Vector-valued Sobolev embeddings}\label{vecor-valued emb}
In section~\ref{chapterdeformationtheory} and~\ref{chaptermoduli} we used results from this appendix. Most of the results here are based on vector-valued Sobolev embeddings. In the first section of this appendix we recall the basics of vector-valued Sobolev embeddings and derive some consequences related to cone and edge-Sobolev spaces. In the second part we prove  the estimate \eqref{extrareg} following a similar result of Witt and Dreher \cite{witt}. This estimate implies the Banach Algebra property of our edge-Sobolev spaces on $M$, \eqref{Banachalg}, and the regularity of the product of elements in $\mathcal{W}^{s,\gamma}(M)$, \eqref{morereg}. In order to simplify the notation we will denote $a\approx b$ and $a\lesssim b$ if $a=\kappa b$ or $a\leq \kappa b$ respectively with a positive constant $\kappa$ depending only on $s$ and $\gamma$.

Let's consider the classical Sobolev spaces $W^{m,p}(\RR^{q})$ (see \cite{brezis} for a detailed introduction). A classical tool in the analysis of partial differential equations on $\RR^{q}$ is the set of Sobolev embeddings,
see \cite{brezis} section 9.3.  
\begin{theorem}\label{thmsobolevembe}
Let $m\in\ZZ$, $m>1$ and $p\in[1,+\infty).$ \begin{align}&\text{If}\quad  \frac{q}{p}>m \quad \text{then }\quad W^{m,p}(\RR^{q})\hookrightarrow L^{k}(\RR^{q}) \quad \text{where } \frac{1}{q}=\frac{1}{p}-\frac{m}{q}.\\&\text{If}\quad  \frac{q}{p}=m \quad \text{then }\quad W^{m,p}(\RR^{q})\hookrightarrow L^{k}(\RR^{q}) \quad \text{for all }\quad k\in [p,+\infty).   \\& \text{If}\quad  \frac{q}{p}<m \quad \text{then }\quad W^{m,p}(\RR^{q})\hookrightarrow L^{\infty}(\RR^{q})\text{ and}\quad W^{m,p}(\RR^{q})\hookrightarrow \mathcal{C}^{r}(\RR^{q})
\label{boundedembedding}
 \end{align}
 where $r=[s-\frac{q}{2}]$ i.e. $r$ is the integer part of $s-\frac{q}{2}$. 
\end{theorem}

In this appendix we are interested in the vector-valued version of this theorem i.e. given a Banach space $\mathfrak{B}$ we want a version for the $\mathfrak{B}$-valued Sobolev spaces $W^{m,p}(\RR^{q},\mathfrak{B}).$ There are many books and monographs dealing with vector-valued spaces of all kinds like $L^{p}(\RR^{q},\mathfrak{B}),\mathcal{C}^{k}(\RR^{q},\mathfrak{B})$ and $\mathcal{S}(\RR^{q},\mathfrak{B})$, see \cite{treves},  \cite{jarchow}, \cite{aman}. In many cases they work in the more general context where  $\mathfrak{B}$ is a Fr\'{e}chet or locally convex Hausdorff space. For our specific purposes  we follow closely \cite{Kreuter}. Here Kreuter analyzes carefully the validity of theorem~\ref{thmsobolevembe} for the spaces $W^{m,p}(\RR^{q},\mathfrak{B})$ where $\mathfrak{B}$ is a Banach space.

Recall that the vector-valued space of distributions is defined as the space of continuous operators from  $\smooth_{0}(\RR^{q})$ to $\mathfrak{B}$ i.e. we have $\mathcal{D'}(\RR^{q},\mathfrak{B}):=\mathcal{L}(\smooth_{0}(\RR^{q}),\mathfrak{B}).$ The vector-valued $L^{p}$-spaces, $L^{p}(\RR^{q},\mathfrak{B})$, are defined by means of the Bochner integral. The Bochner integral is constructed by means of $\mathfrak{B}$-valued step functions in a similar way to the standard Lebesgue integral. See \cite{helmut} appendix A.4 for details.  The vector-valued $\mathcal{C}^{k}$-spaces, $\mathcal{C}^{k}(\RR^{q},\mathfrak{B})$, are defined with respect to the Fr\'{e}chet derivative. 
The vector-valued Sobolev space is defined by \begin{equation}
W^{m,p}(\RR^{q},\mathfrak{B}):=\left\lbrace  f\in L^{p}(\RR^{q},\mathfrak{B}):\partial^{\alpha}f\in L^{p}(\RR^{q},\mathfrak{B}) \quad \forall \quad\abs{\alpha}\leq m \right\rbrace
\end{equation}
where the derivatives of $f$ are taken in the distribution sense i.e. weak derivatives. 

 Here we recall the definition of the Radon-Nikodym property and some results related to it.  It turns out that  the key property that $\mathfrak{B}$ must satisfy in order to have vector-valued Sobolev embeddings for $W^{m,p}(\RR^{q},\mathfrak{B})$ is the Radon-Nikodym property. For extended details the reader is referred to \cite{Kreuter} chapter 2.

 \begin{definition} A Banach space $\mathfrak{B}$ has the Radon-Nikodym property if every Lipschitz continuous function $f:I\longrightarrow \mathfrak{B}$ is differentiable almost everywhere, where $I\subset\RR$ is an arbitrary interval.
\end{definition}

\begin{proposition}
Every reflexive space has the Radon-Nikodym property. In particular the spaces $L^{p}(\RR^{q})$ with $1<p<\infty$ and Hilbert spaces have the Radon-Nikodym property.
\end{proposition}

\begin{corollary}\label{vectorvaluedsobolev}
The Sobolev embeddings in theorem~\ref{thmsobolevembe} are valid for the spaces $W^{m,p}(\RR^{q},L^{p}(\RR^{q}))$ with $1<p<\infty$ and $W^{m,p}(\RR^{q},\mathfrak{H})$ where $\mathfrak{H}$ is a Hilbert space.
\end{corollary} 

As a consequence of these vector-valued results we have the following applications to cone and edge-Sobolev spaces. 

\begin{proposition}\label{coneapplication}
If $f\in\mathcal{H}^{s,\gamma}(\cone)$ (see \eqref{nearconespace}) and $s>\frac{m+1}{2}$ then there exists $C>0$ depending only on $s$ and $\gamma$ such that we have the following estimate on $(0,1)\times\Xcal$ \begin{equation}\label{coneestimate} \abs{\partial^{\alpha'}_{r}\partial^{\alpha''}_{\sigma}f(r,\sigma)}\leq C\norm{f}_{\mathcal{H}^{s,\gamma}(\cone)}r^{\gamma-\frac{m+1}{2}-\abs{\alpha'}}
\end{equation}for all $ (r,\sigma)\in (0,1)\times\Xcal$ and $\abs{\alpha'}+\abs{\alpha''}\leq [s-\frac{m+1}{2}] .$
\end{proposition}

\begin{proof}
We can work locally on $\Rpos\times\mathcal{U}_{\lambda}$ where $\lbrace \mathcal{U}_{\lambda}\rbrace$ is a finite open covering of $\Xcal$, $\lbrace\varphi_{\lambda}\rbrace$ is a subbordinate partition of unity and we consider $\omega\varphi_{\lambda}f$. For simplicity we write just $f$ instead of $\omega\varphi_{\lambda}f$. At the end we take the smallest constant among those obtained for each element in the finite covering.
Take $f\in\mathcal{H}^{s,\gamma}(\cone)$ then by \eqref{S_transf} we have $S_{\gamma-\frac{m}{2}}f\in H^{s}(\RR^{1+m}).$ Therefore if $s>\frac{m+1}{2}$ by \eqref{boundedembedding} we have $S_{\gamma-\frac{m}{2}}f\in L^{\infty}(\RR^{1+m})$ and
 \begin{equation}
\sup_{(t,\sigma)\in\RR^{1+m}}\abs{\partial^{\alpha}_{(t,\sigma)}(S_{\gamma-\frac{m}{2}}f)(t,\sigma)}\lesssim\norm{S_{\gamma-\frac{m}{2}}f}_{H^{s}(\RR^{1+m})}\lesssim \norm{f}_{\mathcal{H}^{s,\gamma}(\cone)} 
\end{equation} 
for all $\abs{\alpha}\leq [s-\frac{m+1}{2}]$.
Now by definition (see \eqref{S_transf}) \[(S_{\gamma-\frac{m}{2}}f)(t,x)=e^{-(\frac{1}{2}-(\gamma-\frac{m}{2}))t}f(e^{-t},x)\]with $r=e^{-t}$. Thus \eqref{coneestimate} follows immediately.
\end{proof}

 In general, if $\mathfrak{B}$ is a Banach space and $\lbrace \kappa_{\lambda}\rbrace _{\lambda\in\Rpos}\in\mathcal{C}\left(\Rpos,\mathcal{L}(\mathfrak{B})\right)$ is a continuous one-parameter group of invertible operators we have that there exist positive constants $K,\mathfrak{c}$ such that \begin{equation}\label{one-parameter bound}
\begin{array}{ccc}
\norm{\kappa_{\lambda}}_{\mathcal{L}(\mathfrak{B})} & \leq & \left\{ \begin{array}{c}
K\lambda^{\mathfrak{c}} \text{ for } \lambda\geq1 \\ 
K\lambda^{-\mathfrak{c}} \text{ for } 0<\lambda\leq1
\end{array}  \right.
\end{array}.
\end{equation} See \cite{schulze3} proposition 1.3.1 for details.

When $\mathfrak{B}=\mathcal{H}^{s,\gamma}(\cone)$ and $(\kappa_{\lambda}f)(r,\sigma)=\lambda^{\frac{m+1}{2}}f(\lambda r,\sigma)$ we can use~\ref{S_transf} to compute $\norm{\kappa_{\lambda}}_{\mathcal{H}^{s,\gamma}(\cone)}=\lambda^{\gamma}$ (see \cite{schulze1} section 1.1). By the proof of proposition 1.3.1 in \cite{schulze3} it is easy to see that the constant $\mathfrak{c}$ in \eqref{one-parameter bound} depends only on the weight $\gamma$. When $\mathfrak{B}=\mathcal{H}^{s,\gamma}(\cone)$ we denote this constant by $\mathfrak{c}_{\gamma}$.

As a consequence of \eqref{one-parameter bound}, we have the following continuous embeddings 
\begin{align}
& \mathcal{W}^{s,\gamma}(\cone\times\RR^{q})
\longrightarrow H^{s-\mathfrak{c}_{\gamma}}(\RR^{q},\mathcal{K}^{s,\gamma}(\cone))\label{edge-classical embedding} \\& H^{s}(\RR^{q},\mathcal{K}^{s,\gamma}(\cone))\longrightarrow \mathcal{W}^{s+\mathfrak{c}_{\gamma},\gamma}(\cone\times\RR^{q})
\end{align}
for all $s\in\RR$ where  $H^{s}(\RR^{q},\mathcal{K}^{s,\gamma}(\cone))$ is the standard vector-valued Sobolev space with norm given by \[ \norm{f}_{H^{s}(\RR^{q},\mathcal{K}^{s,\gamma}(\cone))}:=\left(\int\limits_{\RR^q}[\eta]^{2s}\norm{\mathcal{F}_{u\rightarrow\eta}f(\eta)}^{2}_{\mathcal{K}^{s,\gamma}(\cone)}d\eta\right)^{\frac{1}{2}}.\]
The reader is refer to \cite{schulze3} proposition 1.3.1 and remark 1.3.21 for  details.

\begin{proposition}\label{edgeapplication}
If $g\in\mathcal{W}^{s,\gamma}(M)$ (see \eqref{globaledgenorm}) and $s>\frac{m+1+q}{2}+\mathfrak{c}_{\gamma}$ where $\mathfrak{c}_{\gamma}$ is the constant defined in \eqref{one-parameter bound}, then there exists $C'>0$ depending only on $s$ and $\gamma$ such that we have the following estimate on $(0,1)\times\Xcal\times\edge$ \begin{equation}\label{edgeestimate} \abs{\partial^{\alpha'}_{r}\partial^{\alpha''}_{(\sigma,u)}g(r,\sigma,u)}\leq C'\norm{g}_{\mathcal{W}^{s,\gamma}(M)}r^{\gamma-\frac{m+1}{2}-\abs{\alpha'}}
\end{equation}for all $ (r,\sigma,u)\in (0,1)\times\Xcal\times\edge$ and $\abs{\alpha'}+\abs{\alpha''}\leq [s-\frac{m+1}{2}].$
\end{proposition}
\begin{proof}
We work locally on $(0,1)\times\mathcal{U}_{\lambda} \times\Omega_{j}$ as in proposition~\ref{localpropedge} section~\ref{subsectionNonLinearOp}.
Take $g\in\mathcal{W}^{s,\gamma}(M)$. Again by \eqref{boundedembedding} and  $s>\frac{m+1}{2}$ we have that for each $u\in\RR^{q}$ $$(S_{\gamma-\frac{m}{2}}g)(u)\in H^{s}(\RR^{1+m})$$ and  \begin{equation}\label{semiestimate1}
\norm{(S_{\gamma-\frac{m}{2}}g)(u)}_{L^{\infty}(\RR^{m+1})}\lesssim \norm{g(u)}_{\mathcal{H}^{s,\gamma}(\cone)}.
\end{equation}
Now \eqref{edge-classical embedding} implies $g\in H^{s-\mathfrak{c}_{\gamma}}(\RR^{q},\mathcal{H}^{s,\gamma}(\cone))$ and $s>\frac{q}{2}+\mathfrak{c}_{\gamma}$ together with \eqref{boundedembedding} and corollary~\ref{vectorvaluedsobolev} implies that we have a continuous embedding  \[H^{s-\mathfrak{c}_{\gamma}}(\RR^{q},\mathcal{H}^{s,\gamma}(\cone))\hookrightarrow L^{\infty}(\RR^{q},\mathcal{H}^{s,\gamma}(\cone))\] as  $\mathcal{H}^{s,\gamma}(\cone)$ is a Hilbert space, see definition~\ref{localcone} in section~\ref{sectionanalysis}.

Consequently \begin{align}
\norm{g}_{ L^{\infty}(\RR^{q},\mathcal{H}^{s,\gamma}(\cone))}& =\sup_{u\in\RR^{q}}\left\lbrace  \norm{g(u)}_{\mathcal{H}^{s,\gamma}} \right\rbrace \\&\lesssim \norm{g}_{H^{s-\mathfrak{c}_{\gamma}}(\RR^{q},\mathcal{H}^{s,\gamma}(\cone))}\\& \lesssim \norm{g}_{\mathcal{W}^{s,\gamma}(\cone\times\RR^{q})}. \label{semiestimate2}
\end{align} 
Hence \eqref{semiestimate1}, \eqref{semiestimate2} and the change of variable $r=e^{-t}$ implies \eqref{edgeestimate} as in proposition~\ref{coneapplication}.
\end{proof}

\section{Banach Algebra property of edge-Sobolev spaces }
In \cite{witt} Witt and Dreher used a variant of the edge-Sobolev spaces we use in this paper. In that paper they are interested in applications to weakly hyperbolic equations. Their spaces are defined on $(0,T)\times\RR^{n}$. In this context they proved (proposition 4.1 in \cite{witt}) that their edge-Sobolev spaces have the structure of a Banach algebra. With some modifications and by using vector-valued Sobolev embeddings it is possible to extend their result to our edge-Sobolev spaces on $M$. This extension follows closely the proof of Witt and Dreher. For completeness we include the details of this extension in our context as we used the estimates \eqref{morereg} and \eqref{Banachalg}  in Chapter 3 and 4. 

\begin{proposition}
Let $f,g\in \mathcal{W}^{s,\gamma}(M)$
with $s\in\mathbb{N}$ and $s>\frac{q+m+3}{2}$.  Then $fg\in\mathcal{W}^{s,2\gamma-\frac{m+1}{2}}(M)$ and we have the following estimate 
\begin{equation}\label{extrareg}
\norm{fg}_{\mathcal{W}^{s,2\gamma-\frac{m+1}{2}}(M)}\leq C\norm{f}_{\mathcal{W}^{s,\gamma}(M)}\norm{g}_{\mathcal{W}^{s,\gamma}(M)}
\end{equation}
with a constant $C>0$ depending only on $s$ and $\gamma$.
\end{proposition}

\begin{proof}
By means of finite open covers and partitions of unity on $\Xcal$ and $\edge$ we need to estimate in terms of $\omega\varphi_{\lambda}\phi_{j}f$ and $\omega\varphi_{\lambda}\phi_{j}g$ as in proposition~\ref{localprop} section\ref{subsectionNonLinearOp}. To avoid unnecessary long expressions we will denote them simply by $f$ and $g$. To save space in long expressions we use the notation $\hat{f}$ to denote the Fourier transformation with respect to the conormal variable $\eta$ i.e. $\hat{f}=\mathcal{F}_{u\rightarrow\eta}f.$ We will estimate on an open set $(0,\varepsilon)\times\mathcal{U}_{\lambda}\times\Omega_{j}$. Then the global estimate is obtained by adding these terms.
Take $f,g\in\smooth_{0}(M)$.  By definition of our edge-Sobolev \eqref{localedgenorm} norm and by \eqref{S_transf} we have  

\begin{align}\label{decompositionint}
&\norm{fg}^{2}_{\mathcal{W}^{s,2\gamma-\frac{m+1}{2}}(M)}  \approx  \int\limits_{\RR^{q}_{\eta}} [\eta]^{2s}\norm{ \kappa^{-1}_{[\eta]}\mathcal{F}_{u\rightarrow\eta}(fg)(\eta)}_{\mathcal{H}^{s,2\gamma-\frac{m+1}{2}}(\Rpos\times\RR^{m})}^{2}d\eta &
\\& \approx \sum\limits_{\abs{\alpha}\leq s} \int\limits_{\RR^{q}_{\eta}}\int\limits_{\RR^{}_{t}}\int\limits_{\RR^{m}_{\sigma}}[\eta]^{2s-(m+1)}\big|\partial^{\alpha}\left(e^{-(\frac{m+1}{2}-\gamma)2t}\mathcal{F}_{u\rightarrow\eta}(fg)([\eta]^{-1}e^{-t},\right.
\\&\sigma,\eta)\left. \right)\big|^{2} d\sigma dt d\eta\label{decompositionint2}
\end{align}
Here we will estimate the term with $\alpha=0$. The estimates on the other terms $\alpha\neq0$ are similar. For each term in \eqref{decompositionint2} we have
\begin{align}
&\int\limits_{\RR^{q}_{\eta}}\int\limits_{\RR^{}_{t}}\int\limits_{\RR^{m}_{\sigma}}[\eta]^{2s-(m+1)}\abs{e^{-(\frac{m+1}{2}-\gamma)2t}\mathcal{F}_{u\rightarrow\eta}(fg)}^{2}d\eta =
 \\&  \int\limits_{\RR^{}_{t}} \int\limits_{\RR^{m}_{\sigma}}e^{-(\frac{m+1}{2}-\gamma)4t}\norm{[\eta]^{s-\frac{m+1}{2}}\mathcal{F}_{u\rightarrow\eta}(fg)}_{L^{2}(\RR^{q}_{\eta})}^{2}d\sigma dt. \label{vectorL2}
\end{align} 

Now, the hypothesis $s>\frac{q+m+3}{2}$ allows us to use lemma 4.6 in \cite{witt}. Basically, this lemma  implies that for fixed $(t,\sigma)$ we have the following estimate
\begin{align*}
  &\norm{\Lambda (\eta)\widehat{fg(t,\sigma)}(\eta)}_{L^{2}(\RR^{q}_{\eta})}\leq  \norm{f(t,\sigma)(u)}_{L^{\infty}(\RR^{q}_{u})}\cdot\norm{\Lambda(\eta) \hat{g}(t,\sigma)(\eta)}_{L^{2}(\RR^{q}_{\eta})}
\\& +\norm{g(t,\sigma)(u)}_{L^{\infty}(\RR^{q}_{u})}\cdot\norm{\Lambda(\eta) \hat{f}(r,\sigma)(\eta)}_{L^{2}(\RR^{q}_{\eta})}\\&+C_{0}\norm{\Lambda(\eta)\hat{f}(t,\sigma)(u)\slash [\eta]}_{L^{2}(\RR^{q}_{\eta})}\cdot\norm{\Lambda(\eta) \hat{g}(t,\sigma)(\eta)}_{L^{2}(\RR^{q}_{\eta})}
\end{align*}with $C_0>0$ and $\Lambda(\eta)=[\eta]^{s-\frac{m+1}{2}}$.

Applying this estimate to  \eqref{vectorL2} we have
\begin{align}
& \left( \int\limits_{\RR^{}_{t}} \int\limits_{\RR^{m}_{\sigma}}e^{-(\frac{m+1}{2}-\gamma)4t}\norm{[\eta]^{s-\frac{m+1}{2}}\mathcal{F}_{u\rightarrow\eta}(fg)}_{L^{2}(\RR^{q}_{\eta})}^{2}d\sigma dt \right)^{\frac{1}{2}} \label{longterm}
\\& \leq \Bigg( \int\limits_{\RR^{}_{t}} \int\limits_{\RR^{m}_{\sigma}}\bigg(e^{-(\frac{m+1}{2}-\gamma)2t}\norm{f(t,\sigma)(u)}_{L^{\infty}(\RR^{q}_{u})}\cdot\norm{\Lambda(\eta) \hat{g}(t,x)(\eta)}_{L^{2}(\RR^{q}_{\eta})}+
 \\&  + e^{-(\frac{m+1}{2}-\gamma)2t}\norm{g(t,\sigma)(u)}_{L^{\infty}(\RR^{q}_{u})}\cdot\norm{\Lambda(\eta) \hat{f}(t,\sigma)(\eta)}_{L^{2}(\RR^{q}_{\eta})}
 \\&  +  e^{-(\frac{m+1}{2}-\gamma)2t}
 C_{0}\norm{\Lambda(\eta)\hat{f}(t,\sigma)(u)\slash [\eta]}_{L^{2}(\RR^{q}_{\eta})}
 \\&  \cdot\norm{\Lambda(\eta) \hat{g}(t,\sigma)(\eta)}_{L^{2}(\RR^{q}_{\eta})}\Bigg)^{2}    d\sigma dt \Bigg)^{\frac{1}{2}}.
 \end{align}
 By the Minkowski inequality we have that \eqref{longterm} is less or equal to the following terms
\begin{align}
&  \Bigg( \int\limits_{\RR^{}_{t}} \int\limits_{\RR^{m}_{\sigma}}e^{-(\frac{m+1}{2}-\gamma)4t}\norm{f(t,\sigma)(u)}^{2}_{L^{\infty}(\RR^{q}_{u})}\cdot\norm{\Lambda(\eta) \hat{g}(t,\sigma)(\eta)}^{2}_{L^{2}(\RR^{q}_{\eta})}d\sigma dt \Bigg)^{\frac{1}{2}}
 \\&  +   \Bigg( \int\limits_{\RR^{}_{t}} \int\limits_{\RR^{m}_{\sigma}}e^{-(\frac{m+1}{2}-\gamma)4t}\norm{g(t,\sigma)(u)}^{2}_{L^{\infty}(\RR^{q}_{u})}\cdot\norm{\Lambda(\eta) \hat{f}(t,\sigma)(\eta)}^{2}_{L^{2}(\RR^{q}_{\eta})}d\sigma dt \Bigg)^{\frac{1}{2}}
 \\&  +  \Bigg( \int\limits_{\RR^{}_{t}} \int\limits_{\RR^{m}_{\sigma}}e^{-(\frac{m+1}{2}-\gamma)4t}
 C_{0}\norm{\Lambda(\eta)\hat{f}(t,\sigma)(u)\slash [\eta]}^{2}_{L^{2}(\RR^{q}_{\eta})}
   \\& \cdot\norm{\Lambda(\eta) \hat{g}(t,\sigma)(\eta)}^{2}_{L^{2}(\RR^{q}_{\eta})}  d\sigma dt \Bigg)^{\frac{1}{2}},
\end{align}
hence, by the inequality in \eqref{longterm}, we have
 \begin{align}
 &  \left( \int\limits_{\RR^{}_{t}} \int\limits_{\RR^{m}_{\sigma}}e^{-(\frac{m+1}{2}-\gamma)4t}\norm{[\eta]^{s-\frac{m+1}{2}}\mathcal{F}_{u\rightarrow\eta}(fg)}_{L^{2}(\RR^{q}_{\eta})}^{2}d\sigma dt \right)^{\frac{1}{2}}
 \\& \leq\norm{e^{-(\frac{m+1}{2}-\gamma)t}f(t,\sigma,u)}^{}_{_{L^{\infty}(\RR^{m+1}\times\RR^{q}_{u})}}
  \\&  \cdot\norm{e^{-(\frac{m+1}{2}-\gamma)t}\Lambda(\eta) \hat{g}(t,\sigma)(\eta)}^{}_{_{L^{2}(\RR^{m+1},L^{2}(\RR^{q}_{\eta}))
 }}\label{termf}
   \\& + \norm{e^{-(\frac{m+1}{2}-\gamma)t}g(t,\sigma,u)}^{}_{_{L^{\infty}(\RR^{m+1}\times\RR^{q}_{u})}}
    \\&   \cdot\norm{e^{-(\frac{m+1}{2}-\gamma)t}\Lambda(\eta) \hat{f}(t,\sigma)(\eta)}^{}_{_{L^{2}(\RR^{m+1},L^{2}(\RR^{q}_{\eta}))}}\label{termg}
 \\& + C_{0}\norm{e^{-(\frac{m+1}{2}-\gamma)t}\Lambda(\eta)\hat{f}(t,\sigma)(u)\slash [\eta]}^{}_{_{L^{\infty}(\RR^{m+1},L^{2}(\RR^{q}_{\eta}))}}
    \\&\cdot\norm{e^{-(\frac{m+1}{2}-\gamma)t} \Lambda(\eta) \hat{g}(t,\sigma)(\eta)}^{}_{_{L^{2}(\RR^{m+1},L^{2}(\RR^{q}_{\eta}))}}\label{termlambda}.
\end{align} 
The edge-Sobolev norm of $f$ and $g$ written as in \eqref{vectorL2} implies that 

\begin{equation}\label{gbound}
\norm{e^{-(\frac{m+1}{2}-\gamma)t}\Lambda(\eta) \hat{g}(t,\sigma)(\eta)}^{}_{L^{2}(\RR^{m+1},L^{2}(\RR^{q}_{\eta}))}\leq \norm{g}_{\mathcal{W}^{s,\gamma}(M)}
\end{equation}
and
\begin{equation}\label{fbound}
\norm{e^{-(\frac{m+1}{2}-\gamma)t}\Lambda(\eta) \hat{f}(t,\sigma)(\eta)}^{}_{L^{2}(\RR^{m+1},L^{2}(\RR^{q}_{\eta}))}\leq \norm{f}_{\mathcal{W}^{s,\gamma}(M)},
\end{equation}
hence we only need to deal with the $L^{\infty}$ terms.

To analyze the $L^{\infty}$ terms recall that by hypothesis $s>\frac{q}{2}$ so we have the standard continuous Sobolev embedding $H^{s}(\RR^{q})\hookrightarrow L^{\infty}(\RR^{q})$. Consequently for fixed $(t,\sigma)$ we have 
\begin{align}
\norm{ e^{-(\frac{m+1}{2}-\gamma)t}g(t,\sigma)(\eta) }^{2}_{L^{\infty}(\RR^{q}_{u})}&
 \lesssim \norm{ e^{-(\frac{m+1}{2}-\gamma)t}g(t,\sigma)(\eta) }^{2}_{H^{s}(\RR^{q}_{u})}
\\&=\norm{ \langle \eta \rangle^{s} e^{-(\frac{m+1}{2}-\gamma)t}\hat{g}(t,\sigma)(\eta) }^{2}_{L^{2}(\RR^{q}_{\eta})},
\end{align}
therefore 
\begin{align}
&\norm{ e^{-(\frac{m+1}{2}-\gamma)t}g(t,\sigma)(\eta) }^{2}_{L^{\infty}(\RR^{m+1}\times\RR^{q}_{u})}
 \\&\lesssim \norm{ \langle \eta \rangle^{s} e^{-(\frac{m+1}{2}-\gamma)t}\hat{g}(t,\sigma)(\eta) }^{2}_{L^{\infty}(\RR^{m+1}, L^{2}(\RR^{q}_{\eta}))}
\\& \lesssim \norm{ \langle \eta \rangle^{s} e^{-(\frac{m+1}{2}-\gamma)t}\hat{g}(t,\sigma)(\eta) }^{2}_{W^{s,2}(\RR^{m+1}, L^{2}(\RR^{q}_{\eta}))}\label{vectorembedding}
\\& =\sum_{\abs{\beta}\leq s} \norm{ \langle \eta \rangle^{s} \partial^{\beta}_{(t,\sigma)}\Big( e^{-(\frac{m+1}{2}-\gamma)t}\hat{g}(t,\sigma)(\eta) \Big)}^{2}_{L^{2}(\RR^{m+1}, L^{2}(\RR^{q}_{\eta}))}
\\& \leq \norm{g}_{\mathcal{W}^{s,\gamma}(M)}^{2}.\label{boundg}
\end{align} In \eqref{vectorembedding} we have used the vector-valued version of the standard Sobolev embedding (see section \ref{vecor-valued emb}).
In the same way we obtain 
 \begin{align}
 &\norm{ e^{-(\frac{m+1}{2}-\gamma)t}f(t,\sigma)(\eta) }^{2}_{L^{\infty}(\RR^{m+1}\times\RR^{q}_{u})}
 \\& \lesssim  \norm{f}_{\mathcal{W}^{s,\gamma}(M)}^{2}.\label{boundf}
 \end{align}
Then \eqref{boundg} and \eqref{boundf} implies that \eqref{termf} and \eqref{termg} are bounded by \begin{equation}\label{finalestimate}C(s,\gamma) \norm{f}_{\mathcal{W}^{s,\gamma}(M)}^{2} \norm{g}_{\mathcal{W}^{s,\gamma}(M)}^{2}. 
\end{equation}Thus the only term remaining is \eqref{termlambda}. 

Again using the vector-valued Sobolev embedding we have 
\begin{align}
&\norm{e^{-(\frac{m+1}{2}-\gamma)t}\Lambda(\eta)\hat{f}(t,\sigma)(u)\slash [\eta]}^{2}_{L^{\infty}(\RR^{m+1},L^{2}(\RR^{q}_{\eta}))}
\\& \lesssim \norm{e^{-(\frac{m+1}{2}-\gamma)t}\Lambda(\eta)\hat{f}(t,\sigma)(u)\slash [\eta]}^{2}_{W^{s,2}(\RR^{m+1},L^{2}(\RR^{q}_{\eta}))}
\\& =\sum_{\abs{\beta}\leq s}\norm{\frac{\Lambda(\eta)}{ [\eta]}\partial^{\beta}_{(t,x)}\bigg( e^{-(\frac{m+1}{2}-\gamma)t}\hat{f}(t,\sigma)(u)\bigg)}^{2}_{L^{2}(\RR^{m+1},L^{2}(\RR^{q}_{\eta}))}
    \\&  \lesssim \norm{f}^{2}_{\mathcal{W}^{s,\gamma}(M)}\label{final-f-bound}
\end{align} 
as $\abs{\frac{\Lambda(\eta)}{[\eta]}}^{2}=[\eta]^{2s-(m+1)}\cdot [\eta]^{-2}\lesssim [\eta]^{2s-(m+1)}.$
Thus \eqref{final-f-bound} and \eqref{gbound} implies that \eqref{termlambda} is bounded by \eqref{finalestimate}.
\end{proof}
\begin{corollary}
If $s\in\mathbb{N}$ with $s>\frac{q+m+3}{2}$ and $\gamma \geq\frac{m+1}{2}$ then the  edge Sobolev space $\mathcal{W}^{s,\gamma}(M)$ is a Banach algebra under point-wise multiplication i.e.  given $f,g\in \mathcal{W}^{s,\gamma}(M)$ we have
\begin{equation}\label{Banachalg}
\norm{fg}_{\mathcal{W}^{s,\gamma}(M)}\leq C'\norm{f}_{\mathcal{W}^{s,\gamma}(M)}\norm{g}_{\mathcal{W}^{s,\gamma}(M)}
\end{equation}
with a constant $C'$ depending only on $s$ and $\gamma$.
\end{corollary}
\begin{proof}

\end{proof}
By \eqref{extrareg} we have $fg\in\mathcal{W}^{s,2\gamma-\frac{m+1}{2}}$. Note that $\gamma \geq\frac{m+1}{2}$ if and only if $2\gamma-\frac{m+1}{2}\geq\gamma$ from which the corollary follows immediately.
\begin{corollary}
Let $f,g\in \mathcal{W}^{s,\gamma}(M)$ such that  $s\in\mathbb{N}$ with $s>\frac{q+m+3}{2}$ and $\gamma >\frac{m+1}{2}$. Then \begin{equation}\label{morereg} f g\in \mathcal{W}^{s,\gamma+\beta}(M) 
\end{equation} for $\beta>0$ given by $\beta=\gamma-\frac{m+1}{2}$.
\end{corollary}
\begin{proof}
By \eqref{extrareg} we have $fg\in\mathcal{W}^{s,2\gamma-\frac{m+1}{2}}$. Moreover $\gamma >\frac{m+1}{2}$ implies $2\gamma-\frac{m+1}{2}=\gamma+\beta$ with $\beta=\gamma-\frac{m+1}{2}>0$.
\end{proof}

\newpage
\bibliographystyle{alpha}

\bibliography{final_version_reviewed_ARXIV}

\begin{thebibliography}{KMR97}

\bibitem[Abe12]{helmut}
Helmut Abels.
\newblock {\em Pseudodifferential and singular integral operators}.
\newblock De Gruyter Graduate Lectures. De Gruyter, Berlin, 2012.
\newblock An introduction with applications.

\bibitem[Agr15]{agranovich}
Mikhail~S. Agranovich.
\newblock {\em Sobolev spaces, their generalizations and elliptic problems in
  smooth and {L}ipschitz domains}.
\newblock Springer Monographs in Mathematics. Springer, Cham, 2015.
\newblock Revised translation of the 2013 Russian original.

\bibitem[Ama95]{aman}
Herbert Amann.
\newblock {\em Linear and quasilinear parabolic problems. {V}ol. {I}},
  volume~89 of {\em Monographs in Mathematics}.
\newblock Birkh\"auser Boston, Inc., Boston, MA, 1995.
\newblock Abstract linear theory.

\bibitem[APS75]{patodi}
M.~F. Atiyah, V.~K. Patodi, and I.~M. Singer.
\newblock Spectral asymmetry and {R}iemannian geometry. {I}.
\newblock {\em Mathematical Proceedings of the Cambridge Philosophical
  Society}, 77:43--69, 1975.

\bibitem[Bre11]{brezis}
Haim Brezis.
\newblock {\em Functional analysis, {S}obolev spaces and partial differential
  equations}.
\newblock Universitext. Springer, New York, 2011.

\bibitem[Bry91]{bryant}
Robert~L Bryant.
\newblock Some remarks on the geometry of austere manifolds.
\newblock {\em Boletim da Sociedade Brasileira de
  Matem{\'a}tica-Bulletin/Brazilian Mathematical Society}, 21(2):133--157,
  1991.

\bibitem[CMR15]{rochon}
Ronan~J Conlon, Rafe Mazzeo, and Fr{\'e}d{\'e}ric Rochon.
\newblock The moduli space of asymptotically cylindrical {C}alabi-{Y}au
  manifolds.
\newblock {\em Communications in Mathematical Physics}, 338(3):953--1009, 2015.

\bibitem[DW02]{witt}
Michael Dreher and Ingo Witt.
\newblock Edge {S}obolev spaces and weakly hyperbolic equations.
\newblock {\em Annali di Matematica Pura ed Applicata}, 180(4):451--482, 2002.

\bibitem[ES97]{schulze2}
Yuri~V. Egorov and Bert-Wolfgang Schulze.
\newblock {\em Pseudo-differential operators, singularities, applications},
  volume~93 of {\em Operator Theory: Advances and Applications}.
\newblock Birkh\"auser Verlag, Basel, 1997.

\bibitem[HL82]{harveylawson}
Reese Harvey and H.~Blaine Lawson, Jr.
\newblock Calibrated geometries.
\newblock {\em Acta Mathematica}, 148:47--157, 1982.

\bibitem[HM05]{mazzeo}
Eugenie Hunsicker and Rafe Mazzeo.
\newblock Harmonic forms on manifolds with edges.
\newblock {\em International Mathematics Research Notices}, (52):3229--3272,
  2005.

\bibitem[Jar81]{jarchow}
Hans Jarchow.
\newblock {\em Locally convex spaces}.
\newblock B. G. Teubner, Stuttgart, 1981.
\newblock Mathematische Leitf{\"a}den. [Mathematical Textbooks].

\bibitem[Joy04]{joyceII}
Dominic Joyce.
\newblock Special {L}agrangian submanifolds with isolated conical
  singularities. {I}{I}. {M}oduli spaces.
\newblock {\em Annals of Global Analysis and Geometry}, 25(4):301--352, 2004.

\bibitem[Joy07]{joyce}
Dominic Joyce.
\newblock {\em Riemannian holonomy groups and calibrated geometry}.
\newblock Oxford University Press, 2007.

\bibitem[KL12]{leung}
Spiro Karigiannis and Nat Chun-Ho Leung.
\newblock Deformations of calibrated subbundles of {E}uclidean spaces via
  twisting by special sections.
\newblock {\em Annals of Global Analysis and Geometry}, 42(3):371--389, 2012.

\bibitem[KMR97]{kozlov}
V.~A. Kozlov, V.~G. Mazya, and J.~Rossmann.
\newblock {\em Elliptic boundary value problems in domains with point
  singularities}, volume~52 of {\em Mathematical Surveys and Monographs}.
\newblock American Mathematical Society, Providence, RI, 1997.

\bibitem[Kre]{Kreuter}
{K}reuter {M}arcel., {S}obolev {S}paces of {V}ector-{V}alued {F}unctions,
  {M}aster {T}hesis. {F}aculty of {M}athematics and {E}conomics, {U}lm
  {U}niversity. {G}ermany. 2015.
\newblock
  \url{https://www.uni-ulm.de/fileadmin/website_uni_ulm/mawi.inst.020/abschlussarbeiten/MA_Marcel_Kreuter.pdf}.

\bibitem[Lan93]{lang2}
Serge Lang.
\newblock {\em Real and functional analysis}, volume 142 of {\em Graduate Texts
  in Mathematics}.
\newblock Springer-Verlag, New York, third edition, 1993.

\bibitem[Lan95]{lang}
Serge Lang.
\newblock {\em Differential and {R}iemannian manifolds}, volume 160 of {\em
  Graduate Texts in Mathematics}.
\newblock Springer-Verlag, New York, third edition, 1995.

\bibitem[Lee03]{Lee}
John~M. Lee.
\newblock {\em Introduction to smooth manifolds}, volume 218 of {\em Graduate
  Texts in Mathematics}.
\newblock Springer-Verlag, New York, 2003.

\bibitem[LM85]{lockhart}
Robert~B. Lockhart and Robert~C. McOwen.
\newblock Elliptic differential operators on noncompact manifolds.
\newblock {\em Annali della Scuola Normale Superiore di Pisa. Classe di
  Scienze. Serie IV}, 12(3):409--447, 1985.

\bibitem[Loc87]{lockhart2}
Robert Lockhart.
\newblock Fredholm, {H}odge and {L}iouville theorems on noncompact manifolds.
\newblock {\em Transactions of the American Mathematical Society},
  301(1):1--35, 1987.

\bibitem[LS01]{lauter}
R.~Lauter and J.~Seiler.
\newblock Pseudodifferential analysis on manifolds with boundary---a comparison
  of {$b$}-calculus and cone algebra.
\newblock In {\em Approaches to singular analysis ({B}erlin, 1999)}, volume 125
  of {\em Oper. Theory Adv. Appl.}, pages 131--166. Birkh\"auser, Basel, 2001.

\bibitem[Mar]{Marshall02}
{M}arshall {S}., {D}eformations of {S}pecial {L}agrangian submanifold,{P}h.{D}
  {T}hesis. {U}niversity of {O}xford. 2002.
\newblock \url{https://people.maths.ox.ac.uk/joyce/theses/MarshallDPhil.pdf}.

\bibitem[Maz91]{mazzeo90}
Rafe Mazzeo.
\newblock Elliptic theory of differential edge operators. {I}.
\newblock {\em Communications in Partial Differential Equations},
  16(10):1615--1664, 1991.

\bibitem[McL98]{mclean}
Robert~C. McLean.
\newblock Deformations of calibrated submanifolds.
\newblock {\em Communications in Analysis and Geometry}, 6(4):705--747, 1998.

\bibitem[Mel93]{aps}
Richard~B. Melrose.
\newblock {\em The {A}tiyah-{P}atodi-{S}inger index theorem}, volume~4 of {\em
  Research Notes in Mathematics}.
\newblock A K Peters, Ltd., Wellesley, MA, 1993.

\bibitem[MV14]{vertman}
Rafe Mazzeo and Boris Vertman.
\newblock Elliptic theory of differential edge operators, {II}: {B}oundary
  value problems.
\newblock {\em Indiana University Mathematics Journal}, 63(6):1911--1955, 2014.

\bibitem[NSSS06]{naza}
Vladimir~E. Nazaikinskii, Anton~Yu. Savin, Bert-Wolfgang Schulze, and Boris~Yu.
  Sternin.
\newblock {\em Elliptic theory on singular manifolds}, volume~7 of {\em
  Differential and Integral Equations and Their Applications}.
\newblock Chapman \& Hall/CRC, Boca Raton, FL, 2006.

\bibitem[Pac04]{pacini}
Tommaso Pacini.
\newblock Deformations of asymptotically conical special {L}agrangian
  submanifolds.
\newblock {\em Pacific Journal of Mathematics}, 215(1):151--181, 2004.

\bibitem[Pac13]{pacini2}
Tommaso Pacini.
\newblock Special {L}agrangian conifolds, {I}: moduli spaces.
\newblock {\em Proceedings of the London Mathematical Society. Third Series},
  107(1):198--224, 2013.

\bibitem[Sch91]{schulze1}
B.-W. Schulze.
\newblock {\em Pseudo-differential operators on manifolds with singularities},
  volume~24 of {\em Studies in Mathematics and its Applications}.
\newblock North-Holland Publishing Co., Amsterdam, 1991.

\bibitem[Sch98]{schulze3}
Bert-Wolfgang Schulze.
\newblock {\em Boundary value problems and singular pseudo-differential
  operators}.
\newblock Pure and Applied Mathematics (New York). John Wiley \& Sons, Ltd.,
  Chichester, 1998.

\bibitem[ST99]{schulze4}
Bert-Wolfgang Schulze and Nikolai Tarkhanov.
\newblock Elliptic complexes of pseudodifferential operators on manifolds with
  edges.
\newblock In {\em Evolution equations, {F}eshbach resonances, singular {H}odge
  theory}, volume~16 of {\em Math. Top.}, pages 287--431. Wiley-VCH, Berlin,
  1999.

\bibitem[Swa62]{swan}
Richard~G Swan.
\newblock Vector bundles and projective modules.
\newblock {\em Transactions of the American Mathematical Society}, pages
  264--277, 1962.

\bibitem[Tr{\`e}67]{treves}
Fran{\c{c}}ois Tr{\`e}ves.
\newblock {\em Topological vector spaces, distributions and kernels}.
\newblock Academic Press, New York-London, 1967.

\end{thebibliography}

\end{document}